\documentclass[preprint,3p,fleqn,sort,]{elsarticle}
\usepackage[labelfont=bf, singlelinecheck=false,justification=justified]{caption}
\usepackage{xcolor}
\usepackage{amsmath,amssymb,amsthm}
\usepackage{stmaryrd} 
\usepackage{enumitem}
\usepackage{bm,cancel}
\usepackage{subfig}
\usepackage{accents}
\usepackage{lineno}
\usepackage[T1]{fontenc}

\makeatletter
\g@addto@macro\normalsize{%
	\setlength\abovedisplayskip{5pt}
	\setlength\belowdisplayskip{5pt}
	\setlength\abovedisplayshortskip{6pt}
	\setlength\belowdisplayshortskip{6pt}
}
\makeatother

\usepackage[pdftex,pdfpagelabels=true,bookmarksdepth=3]{hyperref}
\hypersetup{colorlinks=true,
            bookmarksnumbered=true,
            pdftitle = {An entropy stable nodal discontinuous Galerkin method for the resistive MHD equations. Part I: Theory and Numerical Validation},
pdfsubject = {},
pdfauthor = {M.Bohm,A.Winters,G.Gassner,D.Derigs,F.Hindenlang},
pdfkeywords = {}
 }
\usepackage[all]{hypcap}  
\usepackage[dvips]{thumbpdf}

\definecolor{chcol}{rgb}{0.4,0.,0.9}
\newcommand{\change}[1]{\textcolor{black}{#1}}
\newcommand{\pderivative}[2]{\frac{\partial #1}{\partial #2}}
\let\rho\varrho
\newcommand{\avg}[1]{\left\{\hspace*{-3pt}\left\{#1\right\}\hspace*{-3pt}\right\}}
\newcommand{\jump}[1]{\ensuremath{\left\llbracket #1 \right\rrbracket}}
\newcommand{\dS}{{\,\operatorname{dS}}}		
\newcommand{\ec}{{\mathrm{EC}}}			
\newcommand{\es}{{\mathrm{ES}}}			
\newcommand{\ent}{{S}} 				
\newcommand{\ma}{-}					
\renewcommand{\sl}{+}					

\newcommand\iprod[1]{\left\langle #1\right\rangle} 			
\newcommand\iprodN[1]{\left\langle #1\right\rangle_{\!N}}		
\newcommand\isurfN{\int\limits_{ N }\! } 			                
\newcommand\isurfE{\int\limits_{\partial E }\! } 			                
\newcommand\isurfEN{\mkern-11mu\int\limits_{\partial E , N}\mkern-13mu } 			
\newcommand\isurfEnuN{\mkern-11mu\int\limits_{\partial E_\nu , N}\mkern-13mu } 		

\renewcommand\vec[1]{\accentset{\,\rightarrow}{#1}}
\newcommand\spacevec[1]{\accentset{\,\rightarrow}{#1}}		
\newcommand\contraspacevec[1]{\spacevec{\tilde{#1}}}		
\newcommand\statevec[1]{\mathbf #1}					
\newcommand\statevecGreek[1]{\boldsymbol #1}			
\newcommand\contrastatevec[1]{\tilde{\mathbf #1}} 			
\newcommand\acclrvec[1]{\accentset{\,\leftrightarrow}{#1}}	
\newcommand\bigstatevec[1]{\acclrvec{{\mathbf #1}}}		
\newcommand\biggreekstatevec[1]{\acclrvec{\boldsymbol #1}}	
\newcommand\bigcontravec[1]{\acclrvec{\tilde{\mathbf #1}}} 	
\newcommand\biggreekcontravec[1]{\acclrvec{\tilde{\boldsymbol #1}}} 	
\newcommand\threeMatrix[1]{\underline{ #1}}				
\newcommand\nineMatrix[1]{\mathsf{ #1}}					
\newcommand\twentysevenMatrix[1]{\underline{\mathsf{ #1}}}  
\newcommand\matrixvec[1]{\mathcal #1}					
\newcommand{\supEuler}{{\mathrm{Euler}}}
\newcommand{\supMHD}{{\mathrm{MHD}}}
\newcommand{\supGLM}{{\mathrm{GLM}}}

\newcommand{\noncon}{\statevecGreek{\Upsilon}}	
\newcommand{\Jan}{\statevecGreek{\Phi}^\supMHD}						
\newcommand{\JanC}{\statevecGreek{\phi}^\supMHD}				
\newcommand{\JanD}{\statevecGreek{\Phi}^\supMHD}				
\newcommand{\GalC}{\biggreekstatevec{\phi}^\supGLM}
\newcommand{\GalCs}{\statevecGreek{\phi}^\supGLM}
\newcommand{\GalCT}{\biggreekcontravec{\phi}^\supGLM}
\newcommand{\GalD}{\biggreekstatevec{\Phi}^\supGLM}	
\newcommand{\GalDs}{\statevecGreek{\Phi}^\supGLM}	
\newcommand{\GalDT}{\biggreekcontravec{\Phi}^\supGLM}
\newcommand{\wDotJan}{\theta}						
\newcommand{\dmat}{\matrixvec{D}}						
\newcommand{\qmat}{\matrixvec{Q}}						
\newcommand{\mmat}{\matrixvec{M}}					
\newcommand{\bmat}{\matrixvec{B}}						
\newcommand{\viscosity}{\mu_{\mathrm{NS}}}      	      				
\newcommand{\resistivity}{\mu_{\mathrm{R}}}						
\newcommand{\Bn}{B_n}		
\newcommand{\GalDn}{\GalDs_n}
\newcommand{\Bstar}{\left(\JanD B_n \right){\!\!}^{\Diamond}}		
\newcommand{\Psistar}{\left(\GalDs_n \psi\right){\!\!}^{\Diamond}}		
\newcommand{\interpolant}[1]{\mathbb{I}^N\!\!\left(#1\right)}	
\newcommand{\testfuncOne}{\statevecGreek{\varphi}}  
\newcommand{\testfuncTwo}{\biggreekstatevec{\vartheta}}

\newcommand{\DD}{\change{\spacevec{\mathbb{D}}\cdot}}
\newcommand{\DDs}{\change{\spacevec{\mathbb{D}}^{\mathrm{S}}\cdot}}
\newcommand{\DDncdiv}{\change{\spacevec{\mathbb{D}}^{\mathrm{NC}}_{\mathrm{div}}\cdot}}
\newcommand{\DDncgrad}{\change{\spacevec{\mathbb{D}}^{\mathrm{NC}}_{\mathrm{grad}}}}
\theoremstyle{plain}
\newtheorem{thm}{Theorem}
\newtheorem{lem}{Lemma}
\newtheorem{cor}{Corollary}
\theoremstyle{remark}
\newtheorem{rem}{Remark}

\newcommand{\el}{\nu}
\newcommand\interiorfaces{{\mathrm{faces}}}
\begin{document}

\begin{frontmatter}

\title{An entropy stable nodal discontinuous Galerkin method for the resistive MHD equations. Part I: Theory and Numerical Verification}
\author[mathematik]{Marvin Bohm\corref{correspondingauthor}}
\cortext[correspondingauthor]{Corresponding author}
\ead{mbohm@math.uni-koeln.de}
\author[mathematik]{Andrew R.~Winters}
\author[mathematik]{Gregor J.~Gassner}
\author[physik]{Dominik Derigs}
\author[max]{Florian Hindenlang}
\author[geophysik]{Joachim Saur}

\address[mathematik]{Mathematisches Institut, Universit\"at zu K\"oln, Weyertal 86-90, 50931 K\"oln}
\address[physik]{I.\,Physikalisches Institut, Universit\"at zu K\"oln, Z\"ulpicher Stra\ss{}e~77, 50937 K\"oln}
\address[max]{Max-Planck Institut f\"ur Plasmaphysik, Boltzmannstra\ss{}e 2, 85748 Garching}
\address[geophysik]{Institut f\"ur Geophysik und Meteorologie, Universit\"at zu K\"oln, Pohligstra\ss{}e 3, 50969 K\"oln}

\numberwithin{equation}{section}


\begin{abstract}
The first paper of this series presents a discretely entropy stable discontinuous Galerkin (DG) method for the resistive magnetohydrodynamics (MHD) equations on three-dimensional curvilinear unstructured hexahedral meshes. Compared to other fluid dynamics systems such as the shallow water equations or the compressible Navier-Stokes equations, the resistive MHD equations need special considerations because of the divergence-free constraint on the magnetic field. For instance, it is well known that for the symmetrization of the ideal MHD system as well as the continuous entropy analysis a non-conservative term proportional to the divergence of the magnetic field, typically referred to as the Powell term, must be included. 
As a consequence, the mimicry of the continuous entropy analysis in the discrete sense demands a suitable DG approximation of the non-conservative terms in addition to the ideal MHD terms. 

This paper focuses on the \textit{resistive} MHD equations: Our first contribution is a proof that the resistive terms are symmetric and positive-definite when formulated in entropy space as gradients of the entropy variables, which enables us to show that the entropy inequality holds for the resistive MHD equations. This continuous analysis is the key for our DG discretization and guides the path for the construction of an approximation that discretely mimics the entropy inequality, typically termed \textit{entropy stability}. Our second contribution is a detailed derivation and analysis of the discretization on three-dimensional curvilinear meshes. The discrete analysis relies on the summation-by-parts property, which is satisfied by the DG spectral element method (DGSEM) with Legendre-Gauss-Lobatto (LGL) nodes. Although the divergence-free constraint is included in the non-conservative terms, the resulting method has no particular treatment of the magnetic field divergence errors, which might pollute the solution quality. Our final contribution is the extension of the standard resistive MHD equations and our DG approximation with a divergence cleaning mechanism that is based on a generalized Lagrange multiplier (GLM). 

As a conclusion to the first part of this series, we provide detailed numerical validations of our DGSEM method that underline our theoretical derivations. In addition, we show a numerical example where the entropy stable DGSEM demonstrates increased robustness compared to the standard DGSEM. 
\end{abstract}

\end{frontmatter}

\section{Introduction}\label{Sec:Intro}

The resistive magnetohydrodynamic (MHD) equations are of great interest in many areas of plasma, space and astrophysics. This stems from a wide range of applications such as electromagnetic turbulence in conducting fluids, magnetically confined fusion for power generation, modeling the action of dynamos and predicting the interaction of the solar wind with planets or moons. The governing equations are able to describe both dense and thin plasmas that are time-dependent and include motions with a wide range of temporal and spatial scales, e.g., compressible MHD turbulence. In addition, the resistive MHD equations exhibit a mixed hyperbolic-parabolic character depending on the strength of the viscous and resistive effects. Another important property, in a closed physical system, is the second law of thermodynamics, i.e., the evolution of the entropy. In the absence of resistivity and viscosity, that is for the ideal MHD model, and for smooth solutions, the entropy of the system is an additional conserved quantity, although not explicitly built into the mathematical model. Further, in the presence of shocks, the second law of thermodynamics becomes the entropy inequality, e.g. \cite{harten1983}, which guarantees that entropy is always dissipated with the correct sign. It is assumed that the additional resistive terms have a pure entropy dissipative effect as well. But, to the best of our knowledge, no continuous entropy analysis of the resistive MHD equations has been presented in the literature yet and it is unclear if the entropy inequality holds for the resistive MHD equations. Thus, our first contribution in this work is to complete the continuous entropy analysis for the resistive MHD equations. A complication regarding this analysis of MHD models is the involution, that is, the divergence-free constraint of the magnetic field \cite{Barth1999,Godunov1972}
\begin{equation}\label{eq:divBIntro}
\spacevec{\nabla}\cdot\spacevec{B} = 0.
\end{equation}
The condition \eqref{eq:divBIntro} is an additional partial differential equation (PDE) not explicitly built into the resistive MHD equations similar to the entropy inequality. However, it is well known that an additional non-conservative PDE term proportional to the divergence-free constraint is necessary for the entropy analysis of the ideal MHD equations, see e.g. Godunov \cite{Godunov1972}. There are different variants in how to construct such non-conservative terms, e.g. 
Brackbill and Barnes \cite{Brackbill1980}, Powell \cite{Powell1999} and Janhunen \cite{Janhunen2000}. On the continuous level, adding a non-conservative term scaled by \eqref{eq:divBIntro} is a clever way of adding zero to the model. However, for numerical approximations, there are known stability and accuracy issues that differ between the three types of non-conservative terms \cite{Sjogreen2017}.

Mimicking the continuous entropy analysis in the discrete sense is a promising way to enhance the robustness of the resulting numerical approximation. A numerical scheme that satisfies a discrete entropy inequality is often referred to as an entropy stable scheme. Note that, entropy stability is insufficient to give strict non-linear stability, as the continuous and the discrete entropy analysis both assume positivity of the solution and its approximation, respectively. For instance, this assumption can break when simulating strong shocks with a high-order entropy stable method. It is an ongoing research focus how to extend entropy stable high order schemes to full non-linear stability.  However, in practice, entropy stable high order discontinuous Galerkin (DG) schemes show enhanced robustness compared to their standard variants for fluid dynamics problems with weak shocks and especially for compressible (under-resolved) turbulence, e.g. \cite{Carpenter2016,carpenter_esdg,fisher2013,wintermeyer2017,Gassner:2016ye,Chen2017,Ray2017}, as entropy stability provides the desired in-built de-aliasing. At the end of this work, we will show that these positive properties carry over to magnetized fluid dynamics. Entropy stable methods for ideal MHD equations have been studied by many authors, e.g. \cite{Barth1999,Rossmanith2013,Winters2017,Winters2016,Derigs2016,Chandrashekar2015}. We note that there is recent work on entropy stable DG methods applied to the ideal MHD equations by Rossmanith \cite{Rossmanith2013}, Gallego-Valencia \cite{Valencia2017} and most notably the recent work by Liu et al. \cite{Liu2017}, who introduced an entropy stable DG discretization on Cartesian meshes of the non-conservative PDE term proportional to the divergence-free constraint. Our second contribution is an alternative explanation of the entropy stable discretization of the non-conservative term and its extension to fully three-dimensional curvilinear unstructured hexahedral meshes. The geometric flexibility offered by unstructured curvilinear meshes is needed to decompose, e.g., a domain around a spherical object without singularities \cite{duling2014} or a torus-shaped Tokamak reactor \cite{Hindenlang2016}. Herein we consider a nodal discontinuous Galerkin method on unstructured hexahedral grids, as it is able to handle curved elements in a natural way while providing high computational efficiency \cite{Hindenlang2015}. The key to \textit{discrete} entropy stability on curvilinear meshes is to mimic the integration-by-parts property with the DG operators and satisfy the metric identities. This enables the construction of DG methods that are entropy stable without the assumption of exact evaluation of the variational forms. Discrete integration-by-parts, or summation-by-parts (SBP), is naturally obtained when using the Legendre-Gauss-Lobatto (LGL) nodes in the nodal DG approximation \cite{gassner_skew_burgers}. Furthermore, recent work showed that it is possible to construct nodal DG discretizations with LGL nodes that are discretely entropy stable for the viscous terms of the Navier-Stokes equations when using the gradients of the entropy variables instead of the gradients of the conservative or primitive variables \cite{carpenter_esdg,Gassner2017}. We use our first contribution in this paper, the proof that the coefficient matrices of the resistive terms are symmetric and positive semi-definite in terms of the entropy gradients, in order to apply the proof presented in Gassner et al. \cite{Gassner2017}, yielding the result that a Bassi-Rebay (BR1) type scheme  \cite{Bassi&Rebay:1997:B&F97} is entropy stable for the resistive MHD equations. 

As noted, an important difference to the construction of entropy stable DG schemes for non-magnetized fluid dynamics is the necessity to include the divergence-free constraint as a non-conservative PDE term for the continuous and discrete entropy analysis. 
However, it is well known in the MHD numerics community that even if the initial conditions of a problem satisfy \eqref{eq:divBIntro}, it is not guaranteed that the discrete evolution of the magnetized fluid will remain divergence-free in the magnetic field without additional mechanisms. Therefore, many numerical techniques have been devised to control errors introduced into the divergence-free constraint by a numerical discretization. These include the projection approach described e.g. in Brackbill and Barnes \cite{Brackbill1980}, the method of constrained transport introduced by Evans and Hawley \cite{Evans1988} and the generalized Lagrange multiplier (GLM) hyperbolic divergence cleaning technique originally proposed for the ideal MHD equations by Dedner et al. \cite{Dedner2002}. A thorough review of the behavior and efficacy of  these techniques, except hyperbolic divergence cleaning, is provided by  e.g. T\'{o}th \cite{Toth2000}. Due to its relative ease of implementation and computational efficiency we are most interested in the method of hyperbolic divergence cleaning. However, the current work is also concerned with constructing entropy stable numerical approximations. Recent work by Derigs et al. \cite{Derigs2017} modified the additional GLM divergence cleaning system in such a way that the resulting ideal GLM-MHD system is consistent with the continuous entropy analysis and provides in-built divergence cleaning capabilities. The novel entropy stable GLM-MHD system in \cite{Derigs2017} includes the Powell non-conservative term and a non-conservative GLM term in the energy equation, which is necessary for Galilean invariance. Our third contribution focuses on the extension of the entropy stable discontinuous Galerkin spectral element method (DGSEM) for the resistive MHD equations to include \change{this variant} of the entropy stable GLM-MHD system. The resulting DGSEM method is high order accurate, discretely entropy stable on curvilinear elements and has an inbuilt GLM divergence cleaning mechanism. 

The remainder of this paper is organized as follows: In Sec.~\ref{Sec:Systems}, the continuous entropy analysis of the three-dimensional resistive GLM-MHD equations is presented, which demonstrates that the model indeed satisfies the entropy inequality and that the resistive terms can be recast into a symmetric and positive semi-definite form. Next, we introduce the DGSEM on curvilinear hexahedral elements in Sec.~\ref{Sec:DG}. Furthermore, we prove the entropy stability of the numerical approximation in Sec.~\ref{Sec:ESDGres} by discretely mimicking the continuous entropy analysis with special attention given to the metric terms, GLM divergence cleaning and the resistive terms. Finally, in Sec.~\ref{Sec:NumVer} we validate the theoretical findings by numerical tests and demonstrate the increased robustness of the scheme. The final section contains concluding remarks.

\section{Continuous entropy analysis}\label{Sec:Systems}

 In general, we consider systems of conservation laws in a domain $\Omega \subset \mathbb{R}^3$ defined as
\begin{equation}
 \statevec{u}_t + \vec{\nabla} \cdot \bigstatevec{f} = \statevec{0},
\label{consLaw}
\end{equation}  
where $\statevec{u}$ denotes the vector of conserved variables and $\bigstatevec{f}$ the multidimensional flux vector. These definitions allow for a compact notation that will simplify the analysis, i.e., we define block vectors with the double arrow as
\begin{equation}\label{eq:firstBlockVector}
\bigstatevec{f} =
 \left[ {\begin{array}{*{20}{c}}
  {{\statevec f_1}} \\ 
  {{\statevec f_2}} \\ 
  {{\statevec f_3}} 
\end{array}} 
\right],
\end{equation}
and the spatial gradient of a state as
\begin{equation}
\spacevec{\nabla} \statevec u = \left[ {\begin{array}{*{20}{c}}
\statevec{u}_x\\
\statevec{u}_y\\
\statevec{u}_z
\end{array}}
\right].
\end{equation}
The gradient of a spatial vector is a second order tensor, written in matrix form as
\begin{equation}
\spacevec{\nabla}\spacevec{v} =\left[ {\begin{array}{*{20}{ccc}}
\pderivative{v_1}{x} & \pderivative{v_1}{y} & \pderivative{v_1}{z} \\[0.2cm]
\pderivative{v_2}{x} & \pderivative{v_2}{y} & \pderivative{v_2}{z} \\[0.2cm]
\pderivative{v_3}{x} & \pderivative{v_2}{y} & \pderivative{v_3}{z} \\[0.2cm]
\end{array}}
\right]
\end{equation}
and the divergence of a flux written as a block vector is defined as
\begin{equation}
\spacevec\nabla  \cdot \bigstatevec f = \left(\statevec{f}_1\right)_{\!x} + \left(\statevec{f}_2\right)_{\!y} + \left(\statevec{f}_3\right)_{\!z}.
\end{equation}
The dot product of two block vectors is defined by
\begin{equation}
\bigstatevec f \cdot \bigstatevec g = \sum\limits_{i = 1}^3 {{{\statevec f}_i}^T{{\statevec g}_i}},
\end{equation}
and the dot product of a block vector with a spatial vector is a state vector
\begin{equation} 
\spacevec g\cdot\bigstatevec f  = \sum\limits_{i = 1}^3 {{{ g}_i}{{\statevec f}_i}}. 
\end{equation}

\subsection{Resistive GLM-MHD equations}\label{Sec:ResistiveMHD}

The equations that govern the evolution of resistive, conducting fluids depend on the solution as well as its gradient \cite{dumbser2009,fambri2017,Warburton1999}. We outline the variant of the resistive GLM-MHD equations that is consistent with the continuous entropy analysis as developed in Derigs et al. \cite{Derigs2017}. The normalized resistive GLM-MHD equations defined in the domain $\Omega \subset \mathbb{R}^3$ read as
\begin{equation}
 \statevec{u}_t + \vec{\nabla} \cdot \bigstatevec{f}^a(\statevec{u}) - \vec{\nabla} \cdot \bigstatevec{f}^v(\statevec{u},\vec{\nabla} \statevec{u})  + \change{\noncon} = \statevec{r}
\label{resGLMMHD}
\end{equation} 
with the state vector $\statevec{u} = (\rho, \rho \spacevec{v}, E, \spacevec{B}, \psi )^T$ and the advective flux, which we split into three terms, the Euler, ideal MHD and GLM contributions,
\begin{equation}
\bigstatevec{f}^a(\statevec{u}) = \bigstatevec{f}^{a,\supEuler} +\bigstatevec{f}^{a,\supMHD}+\bigstatevec{f}^{a,\supGLM}=
\begin{pmatrix} 
\rho \spacevec{v} \\[0.15cm]
\rho (\spacevec{v}\, \spacevec{v}^{\,T}) + p\threeMatrix{I} \\[0.15cm]
\spacevec{v}\left(\frac{1}{2}\rho \left\|\spacevec{v}\right\|^2 + \frac{\gamma p}{\gamma -1}\right)  \\[0.15cm]
\threeMatrix{0}\\ \spacevec{0}\\[0.15cm]
\end{pmatrix} +
\begin{pmatrix} 
\spacevec{0} \\[0.15cm]
\frac{1}{2} \|\spacevec{B}\|^2 \threeMatrix{I} - \spacevec{B} \spacevec{B}^T \\[0.15cm]
\spacevec{v}\,\|\spacevec{B}\|^2 - \spacevec{B}\left(\spacevec{v}\cdot\spacevec{B}\right) \\[0.15cm]
\spacevec{v}\,\spacevec{B}^T - \spacevec{B}\,\spacevec{v}^{\,T} \\ \spacevec{0}\\[0.15cm]
\end{pmatrix} +
\begin{pmatrix} 
\spacevec{0} \\[0.15cm]
\threeMatrix{0} \\[0.15cm]
c_h \psi \spacevec{B} \\[0.15cm]
c_h \psi \threeMatrix{I} \\ c_h \spacevec{B}\\[0.15cm]
\end{pmatrix} 
\label{eq:advective_fluxes}
\end{equation}
and the viscous flux
\begin{equation}
\bigstatevec{f}^v(\statevec{u},\vec{\nabla} \statevec{u})  = 
\begin{pmatrix} 
\spacevec{0} \\[0.15cm] 
\threeMatrix{\tau} \\[0.15cm]
\threeMatrix{\tau}\spacevec{v} -\vec{\nabla} T - \resistivity \left( (\vec{\nabla} \times \spacevec{B} ) \times \spacevec{B}\right) \\[0.15cm]
\resistivity \left( (\vec{\nabla} \spacevec{B} )^T - \vec{\nabla} \spacevec{B}\right)\\[0.15cm]
\spacevec{0} \\[0.15cm]
\end{pmatrix}\,.
\label{eq:viscous_fluxes}
\end{equation}
Here, $\rho,\,\spacevec{v}=(v_1,v_2,v_3)^T,\,p,\,E$ are the mass density, fluid velocities, pressure and total energy, respectively, $\spacevec{B} = (B_1,B_2,B_3)^T$ denotes the magnetic field components and $\threeMatrix{I}$ the $3\times 3$ identity matrix. Furthermore, the viscous stress tensor reads \cite{kundu2008}
\begin{equation}
\threeMatrix{\tau} = \viscosity ((\vec{\nabla} \spacevec{v}\,)^T + \vec{\nabla} \spacevec{v}\,) - \frac{2}{3} \viscosity(\vec{\nabla} \cdot \spacevec{v}\,) \threeMatrix{I},
\label{stress}
\end{equation}
and the heat flux is defined as
\begin{equation}
\vec{\nabla} T = -\kappa \vec{\nabla} \left(\frac{p}{R \rho}\right).
\label{heat}
\end{equation}
The introduced constants $\viscosity,\resistivity,\kappa,R > 0$ describe the viscosity from the Navier-Stokes equations, resistivity of the plasma, thermal conductivity and the universal gas constant, respectively. In particular, the constants $\viscosity$ and $\resistivity$ are first-order transport coefficients that describe the kinematic viscosity and the diffusivity of the magnetic field \cite{woods1993}. We close the system with the ideal gas assumption, which relates the total energy and pressure 
\begin{equation}
p = (\gamma-1)\left(E - \frac{1}{2}\rho\left\|\spacevec{v}\right\|^2 - \frac{1}{2}\|\spacevec{B}\|^2 - \frac{1}{2}\psi^2\right),
\label{eqofstate}
\end{equation}
where $\gamma$ denotes the adiabatic coefficient. 

The system \eqref{resGLMMHD} contains the GLM extension indicated by the additional field variable $\psi$, which controls the divergence error by propagating it through the physical domain with the wave speed $c_h$, away from its source. \change{We also introduce a non-conservative term and a source term. The non-conservative term  $\noncon$ is related to the thermodynamic properties of \eqref{resGLMMHD} as it ensures the entropy conservation for the advective part of the system outlined in the next section. In particular, we split the non-conservative term into two parts $\noncon = \noncon^\supMHD + \noncon^\supGLM$ with
\begin{align}
\noncon^\supMHD &= (\spacevec{\nabla} \cdot \spacevec{B}) \JanC =  \left(\spacevec{\nabla} \cdot \spacevec{B}\right) 
\left( 0 \,,\, B_1\,,\,B_2\,,\,B_3\,,\, \spacevec{v}\cdot\spacevec{B} \,,\,  v_1\,,\,v_2\,,\,v_3 \,    ,\, 0 \right)^T \,, \label{Powell}\\ 
\noncon^\supGLM &= \GalC \cdot \spacevec{\nabla} \psi \quad =   \GalCs_1 \,\frac{\partial \psi}{\partial x} + \GalCs_2 \frac{\partial \psi}{\partial y} + \GalCs_3 \frac{\partial \psi}{\partial z} \,,\label{NC_GLM}
\end{align}
where $\GalC$ again is a $27$ block vector with
\begin{equation}\label{Galilean}
\GalCs_\ell = \left(0 \,,\, 0\,,\,0\,,\,0\,,\, v_\ell \psi \,,\, 0\,,\,0\,,\,0\,,\, v_\ell \right)^T, \quad \ell = 1,2,3.
\end{equation} 
As presented in \cite{Derigs2017}, the first non-conservative term $\noncon^\supMHD$ is the well-known Powell term \cite{Powell1999}, and the second term $\noncon^\supGLM$ results from Galilean invariance of the full GLM-MHD system \cite{Derigs2017}. We note, that both terms are zero in the continuous case and thus \eqref{resGLMMHD} reduces to the original resistive MHD equations.}

The second purely algebraic source term $\statevec{r}$ on the right hand side of the resistive GLM-MHD system solely provides additional damping of the divergence error \cite{Dedner2002,Derigs2017}, if desired, and is
\begin{equation}\label{damping}
\statevec{r} = \left(0 \,,\, 0\,,0\,,\,0\,,\, 0 \,,\,0\,,\,0\,,\,0\,,\, -\alpha \psi \right)^T\,,
\end{equation}
with $\alpha\geq 0$. 

\subsection{Thermodynamic properties of the system}\label{Sec:ContEntr}

In order to discuss the thermodynamic properties of the resistive GLM-MHD equations \eqref{resGLMMHD} we translate the concepts of the first and second law of thermodynamics into a mathematical context. To do so, we first exclusively examine the advective and non-conservative term proportional to \eqref{eq:divBIntro}. The ideal GLM-MHD equations satisfy the first law of thermodynamics, because the evolution of the total fluid energy is one of the conserved quantities. This is true for any choice of the vector $\noncon$ because \eqref{eq:divBIntro} is assumed to hold in the continuous analysis. But, on the discrete level, this is not the case as noted by many authors \cite{Derigs2017,Powell1999,Janhunen2000,Toth2000}. However, the mathematical description of the second law of thermodynamics is more subtle, because the entropy is not explicitly built into the system. Thus, we require more formalism and utilize the well-developed entropy analysis tools for hyperbolic systems, e.g. \cite{harten1983,tadmor1984,mock1980}. As such, we define a strongly convex entropy function that is then used to define an injective mapping between state space and entropy space. Note that we adopt the mathematical notation of a negative entropy function as is often done, e.g. \cite{harten1983,Tadmor2003}. Once the advective terms are accounted for in entropy space, we present the first main contribution of this work: The resistive terms are indeed consistent with the second law of thermodynamics. 

For the ideal and the resistive GLM-MHD equations, a suitable entropy function is the thermodynamical entropy density divided by the constant $(\gamma-1)$ for convenience
\begin{equation}
S(\statevec{u}) = - \frac{\rho s}{\gamma-1} ~~ \text{ with } ~~ s = \ln\left(p \rho^{-\gamma}\right),
\label{entropy}
\end{equation}
where $s$ is the thermodynamic entropy \cite{Landau1959} with the physical assumptions $\rho,p>0$.  From the entropy function we define the entropy variables to be
\begin{equation}
\statevec{w} = \frac{\partial S}{\partial \statevec{u}} = \left(\frac{\gamma-s}{\gamma-1} - \beta\left\|\spacevec{v}\right\|^2,~2\beta v_1,~2\beta v_2,~2\beta v_3,~-2\beta,~2\beta B_1,~2\beta B_2,~2\beta B_3,~2\beta \psi\right)^T
\label{entrvars}
\end{equation}
with $\beta = \frac{\rho}{2p}$, which is proportional to the inverse temperature.

For smooth solutions, it is known that when we contract the ideal GLM-MHD equations without viscous fluxes nor any source term $\statevec{r}$ on the right hand side by the entropy variables \eqref{entrvars} we obtain the entropy conservation law \cite{Derigs2017}
\begin{equation}\label{eq:PDEcontracted}
 \frac{\partial S}{\partial t} + \vec{\nabla}\cdot\spacevec{f}^{\,\ent} = 0,
\end{equation}
where the entropy fluxes are defined as
\begin{equation}
\spacevec{f}^{\,\ent} = \spacevec{v}S.
\end{equation}

Additionally, as it will be necessary in later derivations and the proof of discrete entropy stability, we define the entropy flux potential to be
\begin{equation}\label{eq:entPotential}
\spacevec{\Psi} := \statevec{w}^T\bigstatevec{f}^a - \spacevec{f}^{\,\ent}+\wDotJan \spacevec{B},
\end{equation}
where we introduce notation for the \change{Powell term from \eqref{Powell} contracted into entropy space, which is
\begin{equation}\label{eq:contractSource}
\wDotJan = \statevec{w}^T\JanC = 2\beta(\spacevec{v}\cdot\spacevec{B}). 
\end{equation}
We note, that the GLM part of the non-conservative term \change{in \eqref{NC_GLM} cancels} internally, when contracted in entropy space, i.e.
\begin{equation}\label{eq:contractSourceGLM}
\statevec{w}^T\noncon^{\supGLM} = \statevec{w}^T\GalC\cdot\spacevec{\nabla}\psi = \spacevec{0}\cdot\spacevec{\nabla}\psi = 0. 
\end{equation}}

\begin{rem}
As introduced in \eqref{eq:advective_fluxes}, we split the advective flux part into three terms to simplify the derivations and keep track of the individual contributions. This will be especially useful in the discrete entropy analysis, where we aim to mimic the continuous derivations. Hence, we analogously split the total entropy flux potential $\spacevec{\Psi}$ into Euler, ideal MHD and GLM components
\begin{equation}\label{eq:entPotential2}
\spacevec{\Psi} = \spacevec{\Psi}^{\supEuler} + \spacevec{\Psi}^{\supMHD} + \spacevec{\Psi}^{\supGLM},
\end{equation}
where
\begin{align}
\spacevec{\Psi}^{\supEuler} &= \statevec{w}^T\bigstatevec{f}^{a,\supEuler} - \spacevec{f}^{\,\ent},\label{eq:EulerEntFluxPot}\\[0.125cm]
\spacevec{\Psi}^{\supMHD} &= \statevec{w}^T\bigstatevec{f}^{a,\supMHD} + \wDotJan \spacevec{B},\label{eq:MHDEntFluxPot}\\[0.125cm]
\spacevec{\Psi}^{\supGLM} &= \statevec{w}^T\bigstatevec{f}^{a,\supGLM}.\label{eq:GLMEntFluxPot}
\end{align}
\end{rem}

Furthermore, e.g. in case of shock discontinuities, the solution satisfies the following entropy inequality
\begin{equation} 
  S_t + \vec{\nabla}\cdot(\spacevec{v}S) \leq 0,
\label{entrineq}
\end{equation} 
which is the mathematical description of the second law of thermodynamics for the ideal GLM-MHD equations. Next, we account for the resistive terms to demonstrate the entropy behavior for the resistive GLM-MHD equations. To do so, we require a suitable representation of the resistive terms to discuss how they affect \eqref{entrineq}. \\ 

\begin{lem}[Entropy representation of viscous and resistive fluxes]\label{lemma1}~\newline
The viscous and resistive fluxes of the resistive GLM-MHD equations in \eqref{eq:viscous_fluxes} can be expressed by gradients of the entropy variables as
\begin{equation}
 \bigstatevec{f}^v(\statevec{u},\vec{\nabla} \statevec{u}) = \twentysevenMatrix{K}  \vec{\nabla} \statevec{w}
\label{Kvisc}
\end{equation}
with a block matrix $\twentysevenMatrix{K}\in\mathbb{R}^{27\times 27}$ that is symmetric and positive semi-definite, i.e,
\begin{equation}
\mathsf{q}^T\twentysevenMatrix{K}\mathsf{q}\geq 0,\quad\forall\mathsf{q}\in\mathbb{R}^{27}.
\end{equation}
\end{lem}

\begin{proof}
We consider the viscous and resistive fluxes of the resistive GLM-MHD system in \eqref{eq:viscous_fluxes}
\begin{equation}
\bigstatevec{f}^v(\statevec{u},\vec{\nabla} \statevec{u}) = \left[\statevec{f}_1^v, \statevec{f}_2^v, \statevec{f}_3^v \right]^T\,.
\end{equation}
Using the vector of entropy variables from \eqref{entrvars} 
\begin{equation}
\statevec{w} = (w_1,\ldots,w_9)^T,
\end{equation}
we find the following relations:
\begin{align*}
\vec{\nabla} v_\ell  &= - \frac{1}{w_5} \vec{\nabla} w_{1+\ell} +  \frac{w_{1+\ell}}{w_5^2} \vec{\nabla} w_5 \,,\quad 
\vec{\nabla} B_\ell   = - \frac{1}{w_5} \vec{\nabla} w_{5+\ell} +  \frac{w_{5+\ell}}{w_5^2} \vec{\nabla} w_5 \,,\quad \ell=1,2,3\,,\\
\vec{\nabla} \left(\frac{p}{\rho}\right)  &= \frac{1}{w_5^2} \vec{\nabla} w_5 \,.
\end{align*}
With some algebraic effort we can determine the matrices $\nineMatrix{K}_{ij} \in \mathbb{R}^{9\times 9}, (i,j=1,2,3)$ to express the viscous fluxes in terms of matrices times the gradients of entropy variables:
\begin{align}
\statevec{f}_1^v & = \nineMatrix{K}_{11} \frac{\partial \statevec{w}}{\partial x} + \nineMatrix{K}_{12} \frac{\partial \statevec{w}}{\partial y} + \nineMatrix{K}_{13} \frac{\partial \statevec{w}}{\partial z} \label{fviscK1}\\
\statevec{f}_2^v & = \nineMatrix{K}_{21} \frac{\partial \statevec{w}}{\partial x} + \nineMatrix{K}_{22} \frac{\partial \statevec{w}}{\partial y} + \nineMatrix{K}_{23} \frac{\partial \statevec{w}}{\partial z} \\
\statevec{f}_3^v & = \nineMatrix{K}_{31} \frac{\partial \statevec{w}}{\partial x} + \nineMatrix{K}_{32} \frac{\partial \statevec{w}}{\partial y} + \nineMatrix{K}_{33} \frac{\partial \statevec{w}}{\partial z} \label{fviscK3}
\end{align}
We collect all these $9\times 9$ block matrices into the matrix $\twentysevenMatrix{K}\in\mathbb{R}^{27\times 27}$
\begin{equation}\label{eq:matrixK}
\twentysevenMatrix{K} = \begin{pmatrix}
\nineMatrix{K}_{11} & \nineMatrix{K}_{12} & \nineMatrix{K}_{13} \\[0.1cm]
\nineMatrix{K}_{21} & \nineMatrix{K}_{22} & \nineMatrix{K}_{23} \\[0.1cm]
\nineMatrix{K}_{31} & \nineMatrix{K}_{32} & \nineMatrix{K}_{33} \\[0.1cm]
\end{pmatrix},
\end{equation}
which clearly yields
\begin{equation}
 \bigstatevec{f}^v = \bigstatevec{f}^v(\statevec{u},\vec{\nabla} \statevec{u}) = \twentysevenMatrix{K}  \vec{\nabla} \statevec{w}.
\label{Kgradient}
\end{equation}
For clarification, we present the first matrix
\begin{equation}\label{eq:matK11}
\resizebox{0.925\hsize}{!}{$
\nineMatrix{K}_{11} = \frac{1}{w_5}\begin{pmatrix}
0 & 0 & 0 & 0 & 0 & 0 & 0 & 0 & 0 \\[0.15cm]
0 & -\frac{4\viscosity}{3} & 0 & 0 & \frac{4\viscosity w_2}{3w_5} & 0 & 0 & 0 & 0 \\[0.15cm]
0 & 0 & -\viscosity & 0 & \frac{\viscosity w_3}{w_5} & 0 & 0 & 0 & 0 \\[0.15cm]
0 & 0 & 0 & -\viscosity & \frac{\viscosity w_4}{w_5} & 0 & 0 & 0 & 0 \\[0.15cm]
0 & \frac{4\viscosity w_2}{3w_5} & \frac{\viscosity w_3}{w_5} & \frac{\viscosity w_4}{w_5} & -\frac{4\viscosity w_2^2}{3w_5^2}-\frac{\viscosity w_3^2}{w_5^2}-\frac{\viscosity w_4^2}{w_5^2}+\frac{\kappa}{Rw_5}-\frac{\resistivity w_7^2}{w_5^2}-\frac{\resistivity w_8^2}{w_5^2} & 0 & \frac{\resistivity w_7}{w_5} & \frac{\resistivity w_8}{w_5} & 0 \\[0.15cm]
0 & 0 & 0 & 0 & 0 & 0 & 0 & 0 & 0 \\[0.15cm]
0 & 0 & 0 & 0 & \frac{\resistivity w_7}{w_5} & 0 & -\resistivity & 0 & 0 \\[0.15cm]
0 & 0 & 0 & 0 & \frac{\resistivity w_8}{w_5} & 0 & 0 & -\resistivity & 0 \\[0.15cm]
0 & 0 & 0 & 0 & 0 & 0 & 0 & 0 & 0 \\[0.15cm]
\end{pmatrix}.$}
\end{equation}
The other matrices $\nineMatrix{K}_{12},\ldots,\nineMatrix{K}_{33}$ are explicitly stated in \ref{Sec:DisMatrix}. It is straightforward to verify that the matrix $\twentysevenMatrix{K}$ is symmetric by inspecting the block matrices listed in \eqref{eq:matK11} and \eqref{eq:matK12} - \eqref{eq:matK33} where the following relationships hold
\begin{equation}\label{eq:symmMatrix}
\nineMatrix{K}_{11} = \nineMatrix{K}_{11}^T,\;\nineMatrix{K}_{22} = \nineMatrix{K}_{22}^T,\;\nineMatrix{K}_{33} = \nineMatrix{K}_{33}^T,\;\nineMatrix{K}_{12} = \nineMatrix{K}_{21}^T,\;\nineMatrix{K}_{13} = \nineMatrix{K}_{31}^T,\;\nineMatrix{K}_{23} = \nineMatrix{K}_{32}^T.
\end{equation}

To show that the matrix $\twentysevenMatrix{K}$ is positive semi-definite is more involved. We first note that it is possible to split the matrix \eqref{eq:matrixK} into the viscous terms associated with the Navier-Stokes equations and the resistive terms of the magnetic fields that arise in the resistive GLM-MHD equations. We exploit this fact and rewrite the total diffusion matrix into two pieces
\begin{equation}
\twentysevenMatrix{K} = \twentysevenMatrix{K}^{\text{NS}} + \twentysevenMatrix{K}^{\text{RMHD}},
\end{equation}
where all terms with $\viscosity$ are put in $\twentysevenMatrix{K}^{\text{NS}}$ and all terms with $\resistivity$ are in $\twentysevenMatrix{K}^{\text{RMHD}}$. 
It is easy to verify that the NS and RMHD block matrices are symmetric, as both satisfy \eqref{eq:symmMatrix}. A further convenience is that the Navier-Stokes part, $\twentysevenMatrix{K}^{\text{NS}}$, is known to be positive semi-definite \cite{Dutt1988}
\begin{equation}
\mathsf{q}^T\twentysevenMatrix{K}^{\text{NS}}\mathsf{q}\geq 0,\quad \forall\mathsf{q}\in\mathbb{R}^{27}.
\end{equation}
Thus, all that remains is to demonstrate that the additional resistive dissipation matrix, $\twentysevenMatrix{K}^{\text{RMHD}}$, is positive semi-definite. To do so, we examine the eigenvalues of the system. We use the computer algebra system Maxima \cite{maxima} to find an explicit expression of the eigenvalues to be
\begin{equation}\label{eq:newEVs}
\lambda^{\text{RMHD}}_0 = 0,\;\lambda^{\text{RMHD}}_1 = \frac{2\resistivity p}{\rho},\;\lambda^{\text{RMHD}}_2 = \frac{\resistivity p\left(\|\spacevec{B}\|^2+2\right)}{\rho},\qquad\textrm{multiplicity:}\;\{24,1,2\}.
\end{equation}
Under the physical assumptions that $p,\rho>0$ and $\resistivity\geq 0$ we see that the eigenvalues \eqref{eq:newEVs} of the matrix $\twentysevenMatrix{K}^{\text{RMHD}}$ are all non-negative. Hence, the block matrix $\twentysevenMatrix{K}$ is symmetric and positive semi-definite. 
\end{proof}

With the ability to rewrite the viscous fluxes as a linear combination of the entropy variable gradients, we can summarize our first result:

\begin{thm}[Entropy inequality for the resistive GLM-MHD equations]\label{theorem:1}~\newline
Solutions of the resistive GLM-MHD equations \eqref{resGLMMHD} with the non-conservative \change{terms \eqref{Powell}, \eqref{NC_GLM} and} $\alpha\geq 0$ in \eqref{damping} satisfy the entropy inequality
\begin{equation}
\int\limits_{\Omega}S_t\,\mathrm{dV}+\int\limits_{\partial \Omega}(\spacevec{f}^{\,\ent} \cdot \spacevec{n})-\statevec{w}^T(\bigstatevec{f}^v \cdot \spacevec{n})\dS\leq 0.
\end{equation} 
\end{thm}
\begin{proof}
We start by contracting the resistive GLM-MHD system \eqref{resGLMMHD} with the entropy variables: 
\begin{equation}
 \statevec{w}^T\statevec{u}_t + \statevec{w}^T\left(\vec{\nabla} \cdot \bigstatevec{f}^a(\statevec{u}) + \noncon\right) = \statevec{w}^T\vec{\nabla} \cdot \bigstatevec{f}^v(\statevec{u},\vec{\nabla} \statevec{u}) + \statevec{w}^T\statevec{r}.
\label{weakformES}
\end{equation} 
From the definition of the entropy variables we have
\begin{equation}
\statevec{w}^T\!\statevec{u}_t = \left(\frac{\partial S}{\partial\statevec{u}}\right)^T\!\!\statevec{u}_t  = S_t.
\label{eq:wsContraction}
\end{equation}
Next, for clarity, we separate the advective flux into Euler, ideal MHD and GLM parts
\begin{equation}
\bigstatevec{f}^a(\statevec{u}) = \bigstatevec{f}^{a,\supEuler}+\bigstatevec{f}^{a,\supMHD}+\bigstatevec{f}^{a,\supGLM}.
\end{equation}
The Euler terms generate the divergence of the entropy flux, e.g., \cite{harten1983}
\begin{equation}\label{eq:fluxContractionEuler}
\statevec{w}^T\!\left( \spacevec{\nabla}\cdot\bigstatevec{f}^{a,\supEuler}\right)= \spacevec\nabla\cdot\spacevec{f}^{\,\ent},
\end{equation}
the ideal MHD and non-conservative term cancel, e.g., \cite{Barth1999,Liu2017}
\change{
\begin{equation}\label{eq:fluxContractionMHD}
\statevec{w}^T\!\left( \spacevec{\nabla}\cdot\bigstatevec{f}^{a,\supMHD} + \noncon^\supMHD\right)= 0
\end{equation}
and the GLM terms vanish as shown in \cite{Derigs2017}
\begin{equation}\label{eq:fluxContractionGLM}
\statevec{w}^T\!\left( \spacevec{\nabla}\cdot\bigstatevec{f}^{a,\supGLM} + \noncon^\supGLM\right)= 0.
\end{equation}}
The damping source term for the GLM divergence cleaning is zero in all but its ninth component, so we see
\begin{equation}
\statevec{w}^T\statevec{r} = -2\alpha\beta\psi^2.
\end{equation}
We have
\begin{equation}
S_t + \vec{\nabla} \cdot \spacevec{f}^{\,\ent} = \statevec{w}^T\vec{\nabla} \cdot \bigstatevec{f}^v-2\alpha \beta \psi^2.
\end{equation} 
Our summation-by-parts DG discretization introduced later is based on a variational formulation of the problem and our goal is to mimic the continuous derivations in the discrete sense. Thus, we seek an integral statement of the continuous entropy inequality. To do so, we integrate over the domain $\Omega$, apply Gauss' law to the entropy flux divergence, and apply integration-by-parts to the viscous and resistive flux contributions to obtain
\begin{equation}\label{eq:appliedGaussLaw}
\int\limits_{\Omega}S_t\,\textrm{dV}+\int\limits_{\partial \Omega}(\spacevec{f}^{\,\ent} \cdot \spacevec{n})-\statevec{w}^T(\bigstatevec{f}^v \cdot \spacevec{n})\dS = -\int\limits_{\Omega}(\vec{\nabla}\statevec{w})^T\bigstatevec{f}^v\,\textrm{dV} -\int\limits_{\Omega}2\alpha \beta \psi^2\,\textrm{dV} .
\end{equation} 
Using the representation of the viscous flux in entropy variable gradients from Lemma \ref{lemma1}, the viscous and resistive flux contribution in the domain become
\begin{equation}
-\int\limits_{\Omega}(\vec{\nabla}\statevec{w})^T\bigstatevec{f}^v\,\textrm{dV} = - \int\limits_{\Omega}(\vec{\nabla}\statevec{w})^T\twentysevenMatrix{K}\,\vec{\nabla}\statevec{w}\,\textrm{dV}\leq 0.
\end{equation}
Assuming the damping parameter $\alpha\geq 0$ and a positive temperature, i.e. $\beta>0$, the contribution of the damping term to the total entropy evolution is guaranteed negative $-2\alpha \beta \psi^2\leq 0$, which finalizes the proof of Theorem \ref{theorem:1}. 
\end{proof}

\begin{cor}
If we consider a closed system, e.g. periodic boundary conditions, Theorem \ref{theorem:1} shows that the total entropy is a decreasing function
\begin{equation}
\int\limits_{\Omega}S_t\,\textrm{dV}\leq 0.
\end{equation}
\end{cor}

In summary, we have demonstrated that the resistive GLM-MHD equations satisfy an entropy inequality. To do so, we separated the advective contributions into Euler, ideal MHD and GLM pieces and considered the viscous contributions separately, which served to clarify how each term contributed to the entropy analysis. A major result is that it is possible to rewrite the resistive terms of the three-dimensional system in an entropy consistent way to demonstrate that those terms are entropy dissipative. We will use an identical splitting of the advective and diffusive terms in the discrete entropy stability proofs in Sec. \ref{Sec:ESDGres} to directly mimic the continuous analysis.

\section{Curved split form discontinuous Galerkin approximation}\label{Sec:DG}
%
In this section, we briefly introduce the building blocks of our entropy stable DGSEM with LGL nodes on three-dimensional curvilinear hexahedral meshes. 
%
\subsection{Mapping the Equations}
%
First, we subdivide the physical domain, $\Omega$, into $N_{\mathrm{el}}$ non-overlapping and conforming hexahedral elements, $E_{\el}$, $\el=1,2,\ldots,N_{\mathrm{el}}$, which can have curved faces if necessary to accurately approximate the geometry \cite{Hindenlang2015}. We create a transformation $\spacevec x = \spacevec X(\spacevec \xi)$ to map computational coordinates, $\spacevec\xi=(\xi,\eta,\zeta)^T$, in the reference element $E=[-1,1]^{3}$ to physical coordinates $\spacevec{x}=(x,y,z)^T$, for each element, e.g. \cite{Farrashkhalvat2003,Kopriva:2009nx} and use this mapping to transform the governing equations \eqref{resGLMMHD} into reference space. From the element mapping we define the three covariant basis vectors
\begin{equation}
\spacevec{a}_{i}=\pderivative{\spacevec{X}}{\xi^{i}}\,,\quad i=1,2,3,
\end{equation}
and (volume weighted) contra-variant vectors
\begin{equation}
J\spacevec{a}^{\,i}=\spacevec{a}_{j}\times\spacevec{a}_{k}\,, \quad (i,j,k)\;\text{cyclic}\,,
\end{equation}
where the Jacobian of the transformation is given by
\begin{equation}
J=\spacevec{a}_{i}\cdot\left(\spacevec{a}_{j}\times \spacevec{a}_{k}\right), \quad (i,j,k)\;\text{cyclic}.
\end{equation}
The basis vectors and Jacobian values will vary on curved elements, but still, the divergence of a constant flux vanishes in the reference cube, i.e., the contra-variant vectors satisfy the metric identities \cite{Kopriva:2006er}
\begin{equation}
\sum\limits_{i = 1}^3 {\frac{{\partial\!\left( J{a^i_n} \right)}}{{\partial {\xi ^i}}} = 0}\,,\quad n=1,2,3\,. 
\label{eq:MetricIdentities}
\end{equation}
We note, that our discrete entropy analysis reveals that the proper discretization of the metric identities is crucial. 

\change{
In order to express the transformation of the gradient and divergence operators, we define two different matrices dependent on the metric terms
\begin{equation}
\twentysevenMatrix{M}  = 
\left[ {\begin{array}{*{20}{c}}
  {Ja_1^1\,{\nineMatrix I_9}}&{Ja_1^2\,{\nineMatrix I_9}}&{Ja_1^3\,{\nineMatrix I_9}} \\ 
  {Ja_2^1\,{\nineMatrix I_9}}&{Ja_2^2\,{\nineMatrix I_9}}&{Ja_2^3\,{\nineMatrix I_9}} \\ 
  {Ja_3^1\,{\nineMatrix I_9}}&{Ja_3^2\,{\nineMatrix I_9}}&{Ja_3^3\,{\nineMatrix I_9}} 
\end{array}} \right],
\quad\quad
\threeMatrix{M} = 
\left[ {\begin{array}{*{20}{c}}
  {Ja_1^1\,}&{Ja_1^2\,}&{Ja_1^3\,} \\ 
  {Ja_2^1\,}&{Ja_2^2\,}&{Ja_2^3\,} \\ 
  {Ja_3^1\,}&{Ja_3^2\,}&{Ja_3^3\,} 
\end{array}} \right]
\end{equation}
with the $9\times9$ unit matrix ${\nineMatrix I_9}$, as $9$ is the size of the GLM-MHD system.
Applying these matrices, the transformation of the gradient of a state or a scalar is
\begin{equation}
\spacevec \nabla_{x}\statevec u=\left[ 
{\begin{array}{*{20}{c}}
  {{{\statevec u}_x}} \\ 
  {{{\statevec u}_y}} \\ 
  {{{\statevec u}_z}} 
\end{array}} \right] = \frac{1}{J}\twentysevenMatrix{M}\left[ {\begin{array}{*{20}{c}}
  {{{\statevec u}_\xi }} \\ 
  {{{\statevec u}_\eta }} \\ 
  {{{\statevec u}_\zeta }} 
\end{array}} \right] = \frac{1}{J} \twentysevenMatrix{M}{\spacevec\nabla _\xi }\statevec u\,, \quad 
\spacevec\nabla_{x} h =\frac{1}{J} \threeMatrix{M}{\spacevec\nabla _\xi }h\,,
\end{equation}
and the transformation of the divergence  is 
\begin{equation}\label{eq:biggDiv}
\spacevec\nabla_{x}  \cdot \bigstatevec g = \frac{1}{J}\spacevec\nabla_{\xi}\cdot\left({\twentysevenMatrix{M}^T}\bigstatevec g\right)\,,\quad 
\spacevec\nabla_{x}  \cdot \spacevec h = \frac{1}{J}\spacevec\nabla_{\xi}\cdot\left({\threeMatrix{M}^T}\spacevec h\right)\,.
\end{equation} 
It is common to define contravariant block vectors and contravariant spatial vectors, e.g. 
\begin{equation}\label{eq:contra9D}
\bigcontravec{\!g} = \twentysevenMatrix{M}^{T}\bigstatevec g=\bigstatevec g \,\threeMatrix{M}\,,\quad \contraspacevec{h} = \threeMatrix{M}^{T}\spacevec h\,.
\end{equation} }
For the discretization of the viscous and resistive terms we introduce the gradient of the entropy variables as an additional unknown  $\bigstatevec q$. Applying the transformations to the divergence and the gradient, we get the transformed resistive GLM-MHD equations 
\begin{equation}
\begin{split}
 J {{\statevec u}_t} + {\spacevec\nabla _\xi } \cdot {\,\bigcontravec{\!f}^a} + \change{\left(\spacevec{\nabla}_{\xi}\cdot \contraspacevec{B} \,\right)\JanC +  \GalCT\cdot\spacevec{\nabla}_{\xi}{\psi}} &= {\spacevec\nabla _\xi } \cdot {{\,\bigcontravec{\!f}^v}}\left( {\statevec u,{\bigstatevec q}} \right) + J \statevec{r} \\
 J \bigstatevec q &= \twentysevenMatrix{M}  {\spacevec\nabla _\xi }\statevec w.
\end{split} 
\label{eq:xFormedNSEquations}
\end{equation}

The next step is the weak formulation of the transformed equations \eqref{eq:xFormedNSEquations}, for which we multiply by test functions $\testfuncOne$ and $\testfuncTwo$. Next, we use integration by parts for the flux divergence as well as for the non-conservative term to arrive at
\begin{equation}\label{eq:DGWeakForm2}
 \begin{aligned}
  \iprod{J\statevec u_t,\testfuncOne} & + \isurfE {\testfuncOne^T\left\{ {\bigstatevec f^a- {{\bigstatevec f^v}}} \right\} \cdot \spacevec n\,\hat{s}\dS}  -  \iprod{\,\bigcontravec{\!f}^a,{\spacevec\nabla _\xi }\testfuncOne} + \change{\isurfE {\testfuncOne}^T\JanC(\spacevec{B}\cdot\spacevec{n})\hat{s}\dS - \iprod{\contraspacevec{B},\spacevec{\nabla}_{\xi}(\testfuncOne^T\JanC)}} \\
		& \change{+ \isurfE {\testfuncOne}^T(\GalC\cdot\spacevec{n})\psi\hat{s}\dS - \iprod{{\psi},\spacevec{\nabla}_{\xi}\cdot(\testfuncOne^T\GalCT)}} = -\iprod{{{\,\bigcontravec{\!f}^v}},\spacevec\nabla _\xi  \testfuncOne} + \iprod{J \statevec{r} , \testfuncOne} \hfill \\
  \iprod{J\bigstatevec q,\testfuncTwo} & = \isurfE { {\statevec w}^T \left(\testfuncTwo \cdot \spacevec{n}\right)\,\hat{s}\dS}  - \iprod{\statevec w,\spacevec\nabla _\xi   \cdot \left( {{\twentysevenMatrix M^T}\testfuncTwo } \right)}. \hfill \\ 
 \end{aligned} 
\end{equation}
Here, we introduced the inner product notation on the reference element for state and block vectors
\begin{equation}
{{{\iprod{\statevec u,\statevec v} }} = \int\limits_{E}{\statevec u^{T}\statevec v \, \mathrm{d}\xi \mathrm{d}\eta \mathrm{d}\zeta } }\quad\text{ and }\quad\iprod{\bigstatevec f,\bigstatevec g}   = \int\limits_{E} {\sum\limits_{i = 1}^3 {\statevec f_i^T{\statevec g_i}} \, \mathrm{d}\xi \mathrm{d}\eta \mathrm{d}\zeta }\,
\end{equation} 
as well as the surface element $\hat{s}$ and the outward pointing unit normal vector $\spacevec{n}$ in physical space. In particular, these are defined for all faces of the reference element as
\begin{equation}
\label{eq:surface_metrics}
\begin{aligned}
     \xi =\pm 1  \quad&:\quad            \hat{s}(\eta,\zeta) :=\left|J\spacevec{a}^{1}(\pm1, \eta,\zeta)\right|
                          \,,\quad  \spacevec{n}(\eta,\zeta) :=  \pm J\spacevec{a}^{1}(\pm1, \eta,\zeta)/\hat{s}(\eta,\zeta)
\\  \eta =\pm 1  \quad&:\quad            \hat{s}(\xi,\zeta)  :=\left|J\spacevec{a}^{2}( \xi, \pm1,\zeta)\right|
                          \,,\quad  \spacevec{n}(\xi,\zeta)  :=  \pm J\spacevec{a}^{2}( \xi, \pm1,\zeta)/\hat{s}(\xi,\zeta) 
\\ \zeta =\pm 1  \quad&:\quad            \hat{s}(\xi,\eta)   :=\left|J\spacevec{a}^{3}( \xi, \eta, \pm1)\right| 
                          \,,\quad  \spacevec{n}(\xi,\eta)   :=  \pm J\spacevec{a}^{3}( \xi, \eta, \pm1)/\hat{s}(\xi,\eta) \,.
\end{aligned}
\end{equation} 
It is important to note, that under the assumption of a conforming mesh, the surface element $\hat{s}$ is continuous across the element interface and the normal vector only changes sign. 

\subsection{Spectral element approximation}

Details on the general DGSEM algorithm are available in the literature, e.g. in the book by Kopriva \cite{Kopriva:2009nx} and in Hindenlang et al. \cite{Hindenlang201286}. 

For a local approximation with polynomial degree $N$, we define $N+1$ LGL nodes and weights, $\xi_i$ and $\omega_i$, $i=0,\ldots,N$ on the unit interval $[-1,1]$ and span the one-dimensional nodal Lagrange basis functions $\ell_i(\xi)$. The one-dimensional functions are extended to the three-dimensional reference element by a tensor product ansatz. We approximate the state vector, flux vectors, etc. with polynomial interpolation on the LGL nodes denoted by capital letters. Alternatively, we write the interpolation of a function $g$ through those nodes as $G=\interpolant{g}$. Moreover, these local polynomial approximations of degree $N$ are used to define the discrete derivative operator. In one spatial dimension, the derivative matrix reads as
\begin{equation}\label{eq:derMatrix}
\dmat_{ij}:=\frac{\partial\ell_j}{\partial\xi}\bigg|_{\xi=\xi_i},\qquad i,j=0,\ldots,N.
\end{equation}
Integrals in the variational form are approximated by the same LGL quadrature rule.  Due to the collocation of the LGL polynomial approximation ansatz and the quadrature \eqref{eq:quadAnsatz}, the mass matrix is discretely orthogonal. In one spatial dimension, the mass matrix is
\begin{equation}
\mmat = \text{diag}(\omega_0,\ldots,\omega_N).
\end{equation}
As mentioned above, this particular choice of the DG derivative operator yields the SBP property
\begin{equation}\label{eq:SBPProperty}
(\mmat\dmat) + (\mmat\dmat)^T = \qmat + \qmat^T = \bmat,
\end{equation}
where we introduce the notation of the SBP matrix, $\qmat$, as well as the boundary matrix
\begin{equation}\label{eq:boundaryMatrix}
\bmat = \text{diag}(-1,0,\ldots,0,1).
\end{equation}
We stress again, that the SBP property is crucial for the stability proofs presented in this work. Furthermore, we note, that it is also important for our proofs on curvilinear meshes, that the mass matrix, also known as the norm matrix in the language of SBP finite difference methods, is diagonal. Additionally, the proper computation of the metric terms is crucial to guarantee that the \textit{discrete metric identities} hold
\begin{equation}
\sum\limits_{i = 1}^3 {\frac{\partial \interpolant{Ja^i_n}}{\partial {\xi ^i}} = 0},\quad n=1,2,3.
\label{eq:DiscreteMetricIdentities}
\end{equation}
This is ensured if the metric terms are computed as curl, i.e.,
\begin{equation}
Ja_n^i =  - {{\hat x}_i} \cdot {\nabla _\xi } \times \left( {\interpolant{{{X_l}{\nabla _\xi }{X_m}} }} \right)\,,\quad i = 1,2,3,\;n = 1,2,3,\;(n,m,l)\;\text{cyclic}.
\end{equation}
This definition ensures free stream preservation discretely and has already been shown to be important for numerical stability, e.g. \cite{Kopriva2016274,wintermeyer2017,Gassner2017}. 

Tensor product extension is used to approximate integrals in multiple spatial dimensions. As such we express the discrete inner product between two functions  $f$ and $g$ in three space dimensions as 
\begin{equation}\label{eq:quadAnsatz}
\iprod{f,g}_N = \sum\limits_{n,m,l = 0}^N {{f_{nml}}{g_{nml}}{\omega_n}{\omega_m}{\omega_l}} \equiv \sum\limits_{n,m,l = 0}^N {{f_{nml}}{g_{nml}}{\omega_{nml}}}, 
\end{equation}
where  $f_{nml}=f\left(\xi_{n},\eta_{m},\zeta_{l}\right)$, etc.

A consequence of the SBP property is a discrete extended Gauss Law \cite{kopriva2017}. That is, for any $V=\interpolant{v}$
\begin{equation}
 \iprodN{\spacevec\nabla_\xi  \cdot \bigcontravec F,V} = \isurfEN {V^T\!\left(\bigcontravec F \cdot \hat n\right) \dS}  - \iprodN{\bigcontravec F,\spacevec\nabla_\xi V} = \isurfEN {V^T\!\left(\bigstatevec F \cdot \spacevec n\right)\hat{s}  \dS}  - \iprodN{\bigcontravec F,\spacevec\nabla_\xi V}\,,
 \label{eq:DiscreteGreens_DAK}
\end{equation}
which we can apply to mimic integration-by-parts in the continuous derivations. Here, the discrete surface integral is also defined via LGL quadrature, and the integrand is transformed with the collocated surface metrics from \eqref{eq:surface_metrics}, yielding
\begin{equation}
 \begin{split}
 \label{eq:discrete_surfint}
 \isurfEN {V^T\!\left(\bigcontravec F \cdot \hat n\right)\!\!\dS}  
 &=\isurfEN {V^T\left(\sum\limits_{i = 1}^3 \left(J\spacevec{a}^{\,i}\cdot \bigstatevec F\right)   \hat n_i\right)\dS}
%
\\&\equiv\sum\limits_{j,k = 0}^N\!\! {{\omega_{jk}}\,\Big[V^T\left(\bigstatevec F \cdot \spacevec{n}\right) \hat{s}\Big]_{\eta_j,\zeta_k}^{\xi=\pm 1}  } \!\!\!
       + \sum\limits_{i,k = 0}^N\!\! {{\omega_{ik}}\,\Big[V^T\left(\bigstatevec F \cdot \spacevec{n}\right) \hat{s}\Big]_{\xi_i,\zeta_k}^{\eta=\pm 1}  } \!\!\!
       + \sum\limits_{i,j = 0}^N\!\! {{\omega_{ij}}\,\Big[V^T\left(\bigstatevec F \cdot \spacevec{n}\right) \hat{s}\Big]_{\xi_i,\eta_j }^{\zeta=\pm 1}  }
\\& = \isurfEN {V^T\left(\bigstatevec F \cdot \spacevec{n}\right) \,\hat{s} \dS}\,.  
%
 \end{split}
\end{equation}

Next, we replace the Jacobian, the metric terms as well as the solution vector, its time derivative, the fluxes, the test functions and source term by polynomial interpolations in \eqref{eq:DGWeakForm2}, to obtain the \textit{discrete weak form} of the DGSEM
\begin{equation}
 \begin{aligned}
  \iprodN{\interpolant{J}\statevec U_t,\testfuncOne}
  &+ \isurfEN {\testfuncOne^{T}\left\{ {\statevec F_n^a - \statevec F_n^v} \right\} \hat s\dS}  
  - \iprodN{\bigcontravec F^a,{\spacevec\nabla _\xi }\testfuncOne}+\change{\isurfEN \testfuncOne^T\Jan B_n \hat{s} \dS - \iprodN{\contraspacevec{B},\spacevec{\nabla}_{\xi}(\testfuncOne^T\Jan)}}
	  \\&\change{+\isurfEN \testfuncOne^T\GalDs_n\psi\hat{s} \dS - \iprodN{{\psi},\spacevec{\nabla}_{\xi}\cdot\interpolant{\testfuncOne^T\GalDT}}} = -\iprodN{{\bigcontravec F^v},\spacevec\nabla_\xi \testfuncOne}\hfill + \iprodN{\interpolant{J}\statevec R,\testfuncOne}\\
  \iprodN{\interpolant{J}\bigstatevec Q, \testfuncTwo } &= 
  \isurfEN {{{\statevec W}^T}\left({\testfuncTwo } \cdot \spacevec{n}\right)\hat{s} \dS}  
  - \iprodN{\statevec W,\spacevec\nabla_\xi  \cdot \interpolant{{\twentysevenMatrix M^T}\testfuncTwo } }, 
 \end{aligned}
\end{equation}
where we introduced compact notation for normal quantities, e.g., the normal flux in physical space $\statevec F_n=\left(\bigstatevec F \cdot \spacevec{n}\right)$.

Due to the discontinuous nature of our local polynomial ansatz, values at the element interfaces are not uniquely defined. The elements are coupled through the boundary terms by way of \emph{numerical fluxes}, which we denote as $\statevec F_n^{a,*}$, $\statevec F_n^{v,*}$ and $\statevec W^{*}$. For now, we postpone the selection of these numerical flux functions to the next sections. 
\change{The non-conservative terms also couple elements through the boundary. However, for now, as the definition of the non-conservative terms at the boundary is not unique, we denote these unknowns at the element interface by $\Bstar$ and $\Psistar$ to obtain}
\begin{equation}\label{eq:weakFormDGSEM}
 \begin{aligned}
\iprodN{\interpolant{J}\statevec U_t,\testfuncOne}
&+ \isurfEN {\testfuncOne^{T}\left\{ {\statevec F_n^{a,*} - \statevec F_n^{v,*}} \right\} \hat{s} \dS}  
- \iprodN{\bigcontravec F^a,{\spacevec\nabla _\xi }\testfuncOne } + \change{\isurfEN \testfuncOne^T \Bstar \hat{s}\dS - \iprodN{\contraspacevec{B},\spacevec{\nabla}_{\xi}\left(\testfuncOne^T\Jan\right)}}\\
&\change{+ \isurfEN \testfuncOne^T \Psistar \hat{s}\dS - \iprodN{{\psi},\spacevec{\nabla}_{\xi}\cdot\interpolant{\testfuncOne^T\GalDT}}} = -\iprodN{\bigcontravec F^v,\spacevec\nabla_\xi \testfuncOne} + \iprodN{\interpolant{J}\statevec R,\testfuncOne} \hfill \\
\iprodN{\interpolant{J}\bigstatevec Q,\testfuncTwo} 
&= \isurfEN {{{\statevec W}^{*,T}}\left(\testfuncTwo  \cdot \spacevec n \right)\hat{s}\dS}  
- \iprodN{\statevec W,\spacevec\nabla_\xi  \cdot \interpolant{{\twentysevenMatrix M^T}\testfuncTwo}}. \hfill \\ 
 \end{aligned}
\end{equation}
Applying the discrete extended Gauss law \eqref{eq:DiscreteGreens_DAK} to the flux and non-conservative terms of the first equation in \eqref{eq:weakFormDGSEM} gives the \textit{strong form} of the DGSEM, which we use in this work to approximate solutions of the resistive GLM-MHD equations
\begin{equation}
 \begin{aligned}
\iprodN{\interpolant{J}\statevec U_t,\testfuncOne} 
&+\iprodN{\spacevec\nabla _\xi \cdot \interpolant{\bigcontravec F^a},\testfuncOne }
+ \isurfEN {\testfuncOne^{T}\left\{ \left(\statevec F_n^{a,*} - \statevec{F}_n^a\right) \right\} \hat s\dS} \\
&+\change{\iprodN{\Jan\spacevec\nabla _\xi\cdot\interpolant{\contraspacevec{B}},\testfuncOne}
+\isurfEN\testfuncOne^T \left\{\Bstar-\Jan \Bn \right\}\hat{s}\dS}\\
&\change{+\iprodN{\GalDT\cdot\spacevec\nabla_\xi \interpolant{{\psi}},\testfuncOne}
+\isurfEN\testfuncOne^T \left\{\Psistar-\GalDs_n \psi \right\}\hat{s}\dS}\\
&= \iprodN{\spacevec\nabla _\xi\cdot\interpolant{\bigcontravec F^v},\testfuncOne }
+\isurfEN\testfuncOne^T\left\{\statevec F_n^{v,*} - \statevec{F}_n^v \right\}\hat{s}\dS 
+ \iprodN{\interpolant{J}\statevec R,\testfuncOne} \hfill \\
\iprodN{\interpolant{J}\bigstatevec Q,\testfuncTwo } 
&= \isurfEN {{{\statevec W}^{*,T}}\left( \testfuncTwo  \cdot \spacevec{n}\right)\hat s\dS}  
- \iprodN{\statevec W,\spacevec\nabla_\xi  \cdot \interpolant{{\twentysevenMatrix M^T}\testfuncTwo } }. \hfill \\ 
 \end{aligned}
 \label{eq:StrongDGSEM}
\end{equation}

\subsection{Split form approximation}

The entropy stable discretization of the flux divergence term is based on the works of Fisher et al. and Carpenter et al. \cite{fisher2013_2,gassner_skew_burgers,carpenter_esdg}. We follow the notation introduced in \cite{Gassner:2016ye,Gassner2017} and present a split form DG approximation, where we have two numerical fluxes, one at the surface and one inside the volume. \change{The special split form volume integral reads
\begin{equation}
\label{eq:entropy-cons_volint}
\begin{split}
\spacevec\nabla _\xi \cdot \interpolant{\bigcontravec F^a}\approx \DD\bigcontravec F^{a,\#} = 
&\quad 2\sum_{m=0}^N \dmat_{im}\,\left(\bigstatevec{F}^{a,\#}(\statevec{U}_{ijk}, \statevec{U}_{mjk})\cdot\avg{J\spacevec{a}^{\,1}}_{(i,m)jk}\right)\\ 
&+     2\sum_{m=0}^N \dmat_{jm}\,\left(\bigstatevec{F}^{a,\#}(\statevec{U}_{ijk}, \statevec{U}_{imk})\cdot\avg{J\spacevec{a}^{\,2}}_{i(j,m)k}\right)\\ 
&+     2\sum_{m=0}^N \dmat_{km}\,\left(\bigstatevec{F}^{a,\#}(\statevec{U}_{ijk}, \statevec{U}_{ijm})\cdot\avg{J\spacevec{a}^{\,3}}_{ij(k,m)}\right)\,
\end{split}
\end{equation}
for each point $i,j,k$ of an element}. Here we introduce the two-point, symmetric volume flux $\bigstatevec{F}^{a,\#}$ and the arithmetic mean of the metric terms. The arithmetic mean in each spatial direction is written compactly, e.g., using the notation in the $\xi-$direction gives
\begin{equation}
\avg{\cdot}_{(i,m)jk} = \frac{1}{2}\left(\left(\cdot\right)_{ijk}+\left(\cdot\right)_{mjk}\right).
\end{equation}
\change{In a similar fashion, the volume contributions of the non-conservative terms are approximated by
\begin{equation}\label{eq:nonConsDerivative}
\begin{split}
\Jan\spacevec\nabla _\xi\cdot\interpolant{\contraspacevec{B}} \approx \Jan\DDncdiv\contraspacevec{B}=
&\quad \sum_{m=0}^N \dmat_{im}\,\left(\Jan_{ijk}\left(\spacevec{B}_{mjk}\cdot\avg{J\spacevec{a}^{\,1}}_{(i,m)jk}\right)\right)\\ 
&+     \sum_{m=0}^N \dmat_{jm}\,\left(\Jan_{ijk}\left(\spacevec{B}_{imk}\cdot\avg{J\spacevec{a}^{\,2}}_{i(j,m)k}\right)\right)\\ 
&+     \sum_{m=0}^N \dmat_{km}\,\left(\Jan_{ijk}\left(\spacevec{B}_{ijm}\cdot\avg{J\spacevec{a}^{\,3}}_{ij(k,m)}\right)\right),
\end{split}
\end{equation}
and 
\begin{equation}\label{eq:nonConsGradient}
\begin{split}
\GalDT\cdot\spacevec\nabla_\xi \interpolant{{\psi}} \approx \GalDT\cdot\DDncgrad{\psi} =
& \quad\sum_{m=0}^N \dmat_{im}\,\left(\left(J\spacevec{a}^{\,1}_{ijk}\cdot\GalD_{ijk} \right)\psi_{mjk}\right)\\ 
&+      \sum_{m=0}^N \dmat_{jm}\,\left(\left( J\spacevec{a}^{\,2}_{ijk}\cdot\GalD_{ijk}\right)\psi_{imk}\right)\\ 
&+      \sum_{m=0}^N \dmat_{km}\,\left(\left( J\spacevec{a}^{\,3}_{ijk}\cdot\GalD_{ijk}\right)\psi_{ijm}\right),
\end{split}
\end{equation}
which introduces compact notation for the discrete divergence and gradient on the non-conservative terms.} We will show in the discrete entropy proofs, that it is important to separate the derivative on the magnetic field components and the metric terms.
Next, we address the discretization of the viscous terms in the resistive MHD equations. The volume contributions are computed in a standard DGSEM way, e.g. \cite{Hindenlang201286}, as
\begin{equation}\label{eq:viscVolume}
\begin{split}
\spacevec\nabla _\xi \cdot \interpolant{\bigcontravec F^v}\approx \DDs\bigcontravec{F}^{v} = \sum_{m=0}^N \dmat_{im}\,\left(\contrastatevec{F}^{v}_1\right)_{mjk} + \sum_{m=0}^N \dmat_{jm}\,\left(\contrastatevec{F}^{v}_2\right)_{imk} + \sum_{m=0}^N \dmat_{km}\,\left(\contrastatevec{F}^{v}_3\right)_{ijm},
\end{split}
\end{equation}
where the metric terms are included in the transformed viscous fluxes $\contrastatevec{F}^v_{l}$, $l=1,2,3$. Inserting the volume discretizations \eqref{eq:entropy-cons_volint}, \eqref{eq:nonConsDerivative}, \eqref{eq:nonConsGradient} and \eqref{eq:viscVolume} into \eqref{eq:StrongDGSEM} we obtain the split form DGSEM:
\begin{equation}
 \begin{aligned}
\iprodN{\interpolant{J}\statevec U_t,\testfuncOne} 
&+\iprodN{\DD\bigcontravec F^{a,\#},\testfuncOne }
+ \isurfEN {\testfuncOne^{T}\left\{ \statevec F_n^{a,*} - \statevec{F}_n^a \right\} \hat{s}\dS} \\
&\change{+\iprodN{\Jan\DDncdiv\contraspacevec{B},\testfuncOne}
+\isurfEN\testfuncOne^T \left\{\Bstar-\Jan \Bn \right\}\hat{s}\dS}\\
&\change{+\iprodN{\GalDT\cdot\DDncgrad\psi,\testfuncOne}
+\isurfEN\testfuncOne^T \left\{\Psistar-\GalDs_n \psi \right\}\hat{s}\dS}\\
&= \iprodN{\DDs\bigcontravec{F}^{v},\testfuncOne }
+\isurfEN\testfuncOne^T\left\{\statevec F_n^{v,*} - \statevec{F}_n^v \right\}\hat{s}\dS   
+ \iprodN{\interpolant{J}\statevec R,\testfuncOne} \hfill \\
\iprodN{\interpolant{J}\bigstatevec Q,\testfuncTwo } 
&= \isurfEN {{{\statevec W}^{*,T}}\left(\testfuncTwo \cdot \spacevec n\right)\hat{s}\dS}  
- \iprodN{\statevec W,\spacevec\nabla_\xi  \cdot \interpolant{{\twentysevenMatrix M^T}\testfuncTwo } }. \hfill \\ 
 \end{aligned}
 \label{eq:schemeFinal}
\end{equation}

\section{Entropy stable DG scheme for resistive GLM-MHD} \label{Sec:ESDGres}

Much work in the numerics community has been invested over the years to develop approximations of non-linear hyperbolic PDE systems that remain thermodynamically consistent, e.g. \cite{Chandrashekar2015,Fjordholm2011,gassner_skew_burgers,tadmor2016,wintermeyer2017,Winters2016}. This began with the pioneering work of Tadmor \cite{tadmor1984,Tadmor1987_2} to develop low-order finite volume approximations. Extension to higher spatial order was recently achieved in the context of DG methods for the compressible Navier-Stokes equations \cite{carpenter_esdg,Gassner:2016ye} as well as the ideal MHD equations \cite{Liu2017,Valencia2017}. Remarkably, Carpenter et al. showed that the conditions to develop entropy stable approximations at low-order immediately apply to high-order methods provided the derivative approximation satisfies the SBP property \cite{carpenter_esdg,fisher2013,fisher2013_2}. 

In order to retain entropy stability, we start with the derived split form DG approximation \eqref{eq:schemeFinal} and contract into entropy space by replacing the first test function with the interpolant of the entropy variables and the second one with the interpolant of the viscous fluxes to obtain:
\begin{equation}\label{eq:schemeFinalcontr}
\begin{aligned}
&\iprodN{\interpolant{J}\statevec{U}_t, \statevec{W}}  = 
- \iprodN{\DD \bigcontravec{F}^{a,\#}, \statevec{W}} 
\change{- \iprodN{\JanD \DDncdiv \contraspacevec{B}, \statevec{W}} 
-\iprodN{\GalDT\cdot\DDncgrad\psi,\statevec{W}}}
+ \iprodN{\interpolant{J} \statevec{R}, \statevec{W}} \\[0.125cm]
&  -  \isurfEN \statevec{W}^T \left[\statevec{F}_n^{a,\ast}-\statevec{F}_n^a \right]\hat{s} \dS 
\change{- \isurfEN  \statevec{W}^T \left[\Bstar - \JanD \Bn \right]\hat{s} \dS  -\isurfEN\statevec{W}^T \left\{\Psistar-\GalDs_n \psi \right\}\hat{s}\dS}\\[0.125cm]
&+ \iprodN{\DDs \bigcontravec{F}^v, \statevec{W}} 
       + \isurfEN \statevec{W}^T \left[\statevec{F}_n^{v,\ast}-\statevec{F}_n^v \right] \hat{s} \dS \\[0.125cm]
&\iprodN{\interpolant{J} \bigstatevec{Q},\bigcontravec{F}^v} 
 = \isurfEN \statevec{W}^{\ast,T}\left(\bigstatevec{F}^v\cdot\spacevec{n}\right)\hat{s} \dS - \iprodN{\statevec{W},\DDs \bigcontravec{F}^v} 
\end{aligned}
\end{equation}
Here, we have intentionally arranged the advective plus non-conservative volume parts, the advective plus non-conservative surface parts and the viscous parts of the first equation into separate rows.

The time derivative term in \eqref{eq:schemeFinalcontr} is the rate of change of the entropy in the element. Assuming that the chain rule with respect to differentiation in time holds (time continuity), we use the contraction property of the entropy variable \eqref{eq:wsContraction} at each LGL node within the element to see that on each element $\el=1,\ldots,N_\mathrm{el}$ we have
\begin{equation}\label{totalEntr}
\iprodN{\interpolant{J}\statevec{U}_t,\statevec{W}} = \sum\limits_{i,j,k=0}^N J_{ijk}\omega_i \omega_j\omega_k \statevec{W}^T_{ijk}\frac{d \statevec{U}_{ijk}}{d t} = \sum\limits_{i,j,k=0}^N J_{ijk}\omega_{ijk} \frac{d S_{ijk}}{d t} = \iprodN{\interpolant{J}S_t,1}.
\end{equation}
To obtain the time derivative of the total discrete entropy we sum over all elements
\begin{equation}
\frac{d \overline{S}}{d t}\equiv\sum\limits_{\el=1}^{N_{\mathrm{el}}}\iprodN{J^\el S^\el_t,1}.
\end{equation}
The final goal of this section is to demonstrate the entropy stability of the contracted DG approximation \eqref{eq:schemeFinalcontr} for the resistive GLM-MHD system.  That is, we want the discrete total entropy in a closed system (periodic boundary conditions) to be a decreasing function
\begin{equation}\label{eq:whatWeWant}
\frac{d \overline{S}}{d t} \leq 0.
\end{equation}

To get the result \eqref{eq:whatWeWant} we examine each row in the first equation of \eqref{eq:schemeFinalcontr} incrementally. In Sec.~\ref{Sec:AdvParts}, we demonstrate the behavior of the advective and non-conservative volume as well as interface contributions. Throughout this section, we highlight how the metric terms and the GLM divergence cleaning parts affect the approximation. Then, in Sec.~\ref{Sec:ResParts}, we assess the contribution of the viscous and resistive terms by using the results of Lemma~\ref{lemma1} and a proof of entropy stability for the BR1 scheme presented in Gassner et al. \cite{Gassner2017}. 

\subsection{Analysis of the advective parts}\label{Sec:AdvParts}

This section focuses on the advective parts in the contracted DG approximation \eqref{eq:schemeFinalcontr}. First, we select the specific form of the advective interface and volume numerical fluxes in Sec.~\ref{Sec:Fluxes}. In the next section, \ref{Sec:Vol}, we show that the volume contributions of the entropy conservative flux of the Euler terms become the entropy flux at the surfaces, the ideal MHD terms cancel and the GLM terms vanish. By splitting the entropy conservative flux into three terms we explicitly see how the discrete contraction into entropy space mimics the results of the continuous analysis, i.e., \eqref{eq:fluxContractionEuler}, \eqref{eq:fluxContractionMHD} and \eqref{eq:fluxContractionGLM}. Next, with the knowledge that the volume contributions move to the interfaces, Sec.~\ref{Sec:Surf} addresses all the surface contributions and we select the form of the coupling for the non-conservative terms. By summing over all the elements and applying the definition of the entropy conservative fluxes we cancel all the remaining advective and non-conservative terms for a closed system (periodic boundary conditions). 

\subsubsection{Numerical advective fluxes}\label{Sec:Fluxes}

A consistent, symmetric numerical flux function, which is entropy conservative for the ideal GLM-MHD equations, is derived in the finite volume context \cite{Derigs2017} and serves as the backbone for the high-order entropy stable DGSEM considered in this work. First, we define the notation for the jump operator, arithmetic and logarithmic means between a left and right state, $a_L$ and $a_R$, respectively
\begin{equation}
\jump{a} := a_R-a_L, ~~~~~~~~~ \avg{a} := \frac{1}{2}(a_L+a_R), ~~~~~~~~ a^{\ln} := \jump{a}/\jump{\ln(a)},
\label{means}
\end{equation}
where a numerically stable procedure to evaluate the logarithmic mean is given in \cite{IsmailRoe2009}. We present the entropy conserving (EC) numerical flux in the first spatial direction to be
\begin{equation}\label{ECFlux}
\statevec{f}_1^{\ec}(\statevec u_L,\statevec u_R) =
\begin{pmatrix} 
\rho^{\ln} \avg{v_1} \\
\rho^{\ln} \avg{v_1}^2 - \avg{B_1}^2 + \overline{p} + \frac{1}{2} \Big(\avg{B_1 B_1} + \avg{B_2 B_2} + \avg{B_3 B_3}\Big)\\ 
\rho^{\ln} \avg{v_1} \avg{v_2} - \avg{B_1} \avg{B_2} \\
\rho^{\ln} \avg{v_1} \avg{v_3} - \avg{B_1} \avg{B_3} \\
f_{1,5}^{\ec} \\
c_h \avg{\psi} \\
\avg{v_1}\avg{B_2} - \avg{v_2}\avg{B_1}\\
\avg{v_1}\avg{B_3} - \avg{v_3}\avg{B_1}\\
c_h \avg{B_1} 
\end{pmatrix}
\end{equation}
with 
\begin{equation}
\begin{split}
f_{1,5}^{\ec} = & f_{1,1}^{\ec}\bigg[\frac{1}{2 (\gamma-1) \beta^{\ln}} - \frac{1}{2} \left(\avg{v_1^2} + \avg{v_2^2} + \avg{v_3^2}\right) \bigg] + f_{1,2}^{\ec} \avg{v_1} + f_{1,3}^{\ec} \avg{v_2} + f_{1,4}^{\ec} \avg{v_3} \\
 & + f_{1,6}^{\ec} \avg{B_1} + f_{1,7}^{\ec} \avg{B_2} + f_{1,8}^{\ec} \avg{B_3} + f_{1,9}^{\ec} \avg{\psi} - \frac{1}{2} \big(\avg{v_1 B_1^2}+\avg{v_1 B_2^2}+\avg{v_1 B_3^2}\big)\\
 & + \avg{v_1 B_1} \avg{B_1}+\avg{v_2 B_2} \avg{B_1}+\avg{v_3 B_3} \avg{B_1} - c_h \avg{B_1 \psi}
\end{split}
\end{equation}
and
\begin{equation*}
\overline{p} = \frac{\avg{\rho}}{2\avg{\beta}}.
\end{equation*}
This particular choice of flux satisfies the discrete entropy conservation condition \cite{Chandrashekar2015,Derigs2017,Liu2017,Valencia2017}
\begin{equation}\label{discECcond}
\jump{\statevec{w}}^T\statevec{f}_\ell^\ec(\statevec u_L,\statevec u_R) = \jump{\Psi_\ell} - \avg{B_\ell}\jump{\wDotJan} \,,\quad \ell=1,2,3
\end{equation}
with the entropy flux potential $\spacevec{\Psi}$ \eqref{eq:entPotential} and the contracted non-conservative state vector $\wDotJan$ \eqref{eq:contractSource}.
 
Moreover, in the presence of shocks or discontinuities, we must add dissipation to the interface fluxes in terms of the entropy variables  to ensure we do not violate the entropy inequality \eqref{entrineq}. In order to create such an entropy stable scheme, we use the EC flux in \eqref{ECFlux} as a baseline flux and add a general form of numerical dissipation at the interfaces to get an entropy stable (ES) numerical flux that is applicable to arbitrary flows
\begin{equation}\label{ES_flux}
	\bigstatevec{f}^{\,\es}\cdot\spacevec{n} =  \bigstatevec{f}^{\,\ec}\cdot \spacevec{n} - \frac{1}{2} \nineMatrix{\Lambda} \nineMatrix{H} \jump{\statevec{w}},
\end{equation}
where $\nineMatrix{H}$ is the entropy Jacobian and $\nineMatrix{\Lambda}$ the dissipation matrix. Note that for strong shocks or high Mach number flows a careful evaluation of the dissipation matrix at the interface is needed, e.g. \cite{Derigs2016_2,Winters2017}. It is straightforward to derive similar EC and ES fluxes for the $y-$ and $z-$directions. Full details can be found in \cite{Derigs2017}.

In the contracted DG approximation \eqref{eq:schemeFinalcontr} we select both the two point volume fluxes $\bigstatevec{F}^{a,\#}$ and the advective surface fluxes $\statevec{F}_n^{a,\ast}$ to be the EC fluxes
\begin{equation}\label{ECVolSurfluxes}
\bigstatevec{F}^{a,\#}(\statevec{U}_{ijk},\statevec{U}_{mjk}) = \bigstatevec{f}^{\,\ec}(\statevec{U}_{ijk},\statevec{U}_{mjk})\,,\quad\statevec{F}_n^{a,\ast} = \bigstatevec{F}^{\ec}\cdot\spacevec{n}=\bigstatevec{f}^{\,\ec}(\statevec U_L,\statevec U_R) \cdot\spacevec{n}\,,
\end{equation}
whereas the latter can also include stabilization terms as in \eqref{ES_flux}.

Again, as in the continuous analysis, we split the advective EC numerical flux function into three components
\begin{equation}\label{ECfluxsplit}
\bigstatevec{F}^\ec = \bigstatevec{F}^{\ec,\supEuler} + \bigstatevec{F}^{\ec,\supMHD} + \bigstatevec{F}^{\ec,\supGLM},
\end{equation}
as well as the appropriate entropy conservation conditions for the numerical flux functions at each surface LGL node,
\begin{align}
\jump{\statevec{W}}^T\bigstatevec{F}^{\ec,\supEuler} &= \jump{\spacevec{\Psi}^{\supEuler}}\label{eq:EulerPartDiscrete},\\[0.1cm]
\jump{\statevec{W}}^T\bigstatevec{F}^{\ec,\supMHD} &= \jump{\spacevec{\Psi}^{\supMHD}} - \avg{\spacevec{B}}\jump{\wDotJan}\label{eq:MHDPartDiscrete},\\[0.1cm]
\jump{\statevec{W}}^T\bigstatevec{F}^{\ec,\supGLM} &= \jump{\spacevec{\Psi}^{\supGLM}}\label{eq:GLMPartDiscrete},
\end{align}
where we use the previously defined split entropy flux potential \eqref{eq:entPotential2}. Respectively, \eqref{eq:EulerPartDiscrete} contains the hydrodynamic contributions, \eqref{eq:MHDPartDiscrete} contains the magnetic field parts and \eqref{eq:GLMPartDiscrete} contains the GLM components. 

\subsubsection{Volume contributions}\label{Sec:Vol}

Next, we focus on the advective volume discretizations as well as on the non-conservative volume terms in the first row of \eqref{eq:schemeFinalcontr}. Using \eqref{ECfluxsplit} we split the advective fluxes to determine the contributions from the Euler, MHD and GLM parts, separately. Since the contribution of the curvilinear Euler components has been investigated in the DG context, see e.g. \cite{carpenter_esdg,Gassner2017}, we address the curvilinear ideal MHD and GLM parts first. Due to the inclusion of the complete three-dimensional curved framework, both proofs are quite lengthy and, thus, can be found in the appendices.  
\change{
\begin{lem}[Entropy contribution of the curvilinear ideal MHD volume terms]\label{lem:MHD_EC_vol}~\newline
The curvilinear volume contributions of the ideal MHD equations in \eqref{eq:schemeFinalcontr} cancel in entropy space. That is,
\begin{equation}\label{eq:MHDVolTerms}
\iprodN{\DD\bigcontravec F^{\ec,\supMHD},\statevec W} + \iprodN{\Jan\DDncdiv\contraspacevec{B},\statevec{W}} = 0.
\end{equation}
\end{lem}}
\begin{proof}
See \ref{Sec:MHD_vol}.
\end{proof}

\begin{rem}\label{MHD_rem}
In the proof we use the SBP property \eqref{eq:SBPProperty} repeatedly as well as the discrete version of the MHD entropy potential condition \eqref{eq:MHDPartDiscrete}. Moreover, a crucial condition to obtain the desired result on curved elements is that the discrete metric identities \eqref{eq:DiscreteMetricIdentities} are satisfied.
\end{rem}

\change{
\begin{lem}[Entropy contribution of the curvilinear GLM volume terms]\label{lem:GLM_EC_vol}~\newline
The curvilinear GLM volume contributions of \eqref{eq:schemeFinalcontr} reduce to zero in entropy space. That is,
\begin{equation}\label{eq:volInProof}
\iprodN{\DD\bigcontravec F^{\ec,\supGLM},\statevec W} + \iprodN{\GalDT\cdot\DDncgrad\psi,\statevec{W}}= 0.
\end{equation}
\end{lem}}
\begin{proof}
See \ref{Sec:GLM_vol}.
\end{proof}

\begin{rem}\label{GLM_rem}
Again, the most important requirements for the proof are the SBP property \eqref{eq:SBPProperty}, the discrete GLM entropy flux condition \eqref{eq:GLMPartDiscrete} and the discrete metric identities \eqref{eq:DiscreteMetricIdentities}.
\end{rem}

\begin{rem}\label{ES_damp}
If we also take the damping source term of the GLM divergence cleaning into account, the statement of Lemma \ref{lem:GLM_EC_vol} becomes an inequality, i.e.
\begin{equation}\label{eq:volInProofDamp}
-\iprodN{\DD\bigcontravec F^{\ec,\supGLM},\statevec W} \change{-\iprodN{\GalDT\cdot\DDncgrad\psi,\statevec{W}}}+ \iprodN{ \interpolant{J} \statevec{R}, \statevec{W}} \leq 0,
\end{equation}
since
\begin{equation}
 \iprodN{ \interpolant{J} \statevec{R},\statevec{W} }=  - \sum\limits_{i,j,k=0}^N J_{ijk} \omega_{ijk} \left(2 \alpha \beta_{ijk} \psi_{ijk}^2 \right)\leq  0,
\end{equation}
for $\alpha,\beta_{ijk} \geq 0$. This result corresponds to discrete entropy stability instead of conservation and will be excluded for the following discussion of the remaining advective parts.
\end{rem}

All together, this leads us to the following result:
\begin{cor}[Entropy contribution of the curvilinear advective volume terms]\label{cor:EC_vol}~\newline
For each element the sum of all curvilinear advective volume contributions plus the non-conservative volume terms in \eqref{eq:schemeFinalcontr} yields
\begin{equation}
\change{\iprodN{\DD \bigcontravec{F}^\ec, \statevec{W}} + \iprodN{\Jan\DDncdiv\contraspacevec{B},\statevec{W}} + \iprodN{\GalDT\cdot\DDncgrad \psi,\statevec{W}}} = \isurfEN \left(\spacevec{F}^\ent\cdot\spacevec{n}\right) \hat{s}\dS.
\end{equation}
\end{cor}
\begin{proof}
Again we first split the volume flux in Euler, MHD and GLM parts according to \eqref{ECfluxsplit}. From Lemmas \ref{lem:MHD_EC_vol} and \ref{lem:GLM_EC_vol} we know, that the curvilinear MHD and GLM volume terms together with the non-conservative terms vanish. Moreover, we know from \cite{fisher2013,Gassner2017}, that the volume contributions of the Euler components become the entropy flux evaluated at the boundary 
\begin{equation}\label{eq:volSurfTerms}
\iprodN{\DD \bigcontravec F^{\ec,\supEuler},\statevec W} 
= \isurfEN \left(\spacevec{F}^\ent\cdot\spacevec{n}\right) \hat{s} \dS,
\end{equation}
which is equivalent to the steps \eqref{eq:fluxContractionEuler} and \eqref{eq:appliedGaussLaw} in the continuous analysis with $\spacevec{F}^\ent$ being the discrete evaluation of the entropy flux.
\end{proof}

The results of Lemmas \ref{lem:MHD_EC_vol}, \ref{lem:GLM_EC_vol} and Corollary \ref{cor:EC_vol} demonstrate that many of the volume contributions cancel in entropy space and the remaining terms move to the interfaces of the contracted DG approximation. Thus, in the next section we include this additional interface contribution containing the entropy fluxes.

\subsubsection{Surface contributions}\label{Sec:Surf}

We are now prepared to examine the advective surface terms of the contracted DG approximation \eqref{eq:schemeFinalcontr} incorporating the now known additional surface part that comes from the volume terms due to the result of Corollary \ref{cor:EC_vol}. On each element the surface terms are given in compact notation as
\begin{equation}\label{eq:surfaceTermsOnK}
\begin{aligned}
\Gamma_\nu = 
\isurfEnuN  \statevec{W}^T \left[\statevec{F}_n^{\ec}-\statevec{F}_n^a \right]&\hat{s} \dS \change{+ \isurfEnuN   \statevec{W}^T \left[\Bstar - \JanD \Bn\right]\hat{s} \dS} 
\\& \change{+ \isurfEnuN   \statevec{W}^T \left[\Psistar - \GalDn\psi\right]\hat{s} \dS}  +  \isurfEnuN  \left(\spacevec{F}^\ent\cdot\spacevec{n}\right) \hat{s}\dS.
\end{aligned}
\end{equation}
To determine the total surface contributions from the advective and non-conservative terms in the contracted DG approximation \eqref{eq:schemeFinalcontr} we sum over all elements, $\nu=1,\ldots,N_{\mathrm{el}}$ similar to Gassner et al. \cite{Gassner2017}. We introduce notation for states at the LGL node of the one side of the  interface between two elements to be a primary ``$\ma$'' and complement the notation with a secondary ``$\sl$'' to denote the value at the LGL nodes on the opposite side. This allows us to define the orientated jump and the arithmetic mean at the interfaces to be
\begin{equation}\label{eq:jumpNotation}
\jump{\cdot} = (\cdot)^{\sl} - (\cdot)^{\ma},\quad \avg{\cdot} = \frac{1}{2}\left((\cdot)^{\sl}+(\cdot)^{\ma}\right).
\end{equation}
When applied to vectors, the average and jump operators are evaluated separately for each vector component. The physical normal vector $\spacevec{n}$ is then defined uniquely to point from the ``$\ma$'' to the ``$\sl$'' side, so that $\spacevec{n}=(\spacevec{n})^\ma=-(\spacevec{n})^\sl$. 

We consider the discrete total entropy evolution in a closed system and thus focus on fully periodic domains, so that all interfaces in the domain have two adjacent elements. We investigate the total surface contributions from \eqref{eq:surfaceTermsOnK} term by term. The sum over all elements for the first term generates jumps in the fluxes and entropy variables, where we also use the uniqueness of the numerical surface flux function, $\statevec F_n^{\ec}=\bigstatevec{F}^{\ec}\cdot\spacevec{n}$, yielding
\begin{equation}\label{eq:surfaceTermsOnK2}
\sum\limits_{\nu=1}^{N_{\mathrm{el}}}\;\;\; \isurfEnuN  \statevec{W}^T \left(\bigstatevec{F}^{\ec}-\bigstatevec{F}^a \right)\cdot \spacevec{n} \,\hat{s} \dS 
=-\sum\limits_{\interiorfaces}\isurfN  \left(\jump{\statevec{W}}^T\left(\bigstatevec{F}^{\ec}\cdot\spacevec n\right) - \jump{\statevec{W}^T\bigstatevec{F}^a}\cdot\spacevec{n}\right)\,\hat{s}\dS.
\end{equation}

Next, we examine the behavior of the GLM part of the entropy conservative flux at the interfaces that come from \eqref{eq:surfaceTermsOnK}. \change{Also, we account for the surface contribution of the GLM non-conservative  term \eqref{NC_GLM}. 
\begin{lem}[Entropy contribution of GLM surface terms]\label{lem:GLM_EC_surf}~\newline
The contribution from the GLM part of the entropy conservative scheme vanishes at element interfaces, i.e.,
\begin{equation}\label{eq:GLMsurfFlux}
\isurfN  \left(\jump{\statevec W}^T \bigstatevec{F}^{\ec,\supGLM} 
- \jump{\statevec{W}^T\,\bigstatevec{F}^{a,\supGLM}}\right)\cdot\spacevec{n}\,\hat{s}\dS = 0.
\end{equation}
For the non-conservative term surface contribution we define the interface coupling as
\begin{equation}\label{eq:GalInterface}
\Psistar = \left(\left(\GalD\right)^{\ma}\cdot\spacevec{n}\right)\avg{\psi}
\end{equation}
to ensure that the associated non-conservative terms from \eqref{eq:surfaceTermsOnK} vanish locally at each element face.
\end{lem}
\begin{proof}
The proof of \eqref{eq:GLMsurfFlux} follows directly from the definition of the GLM components of the entropy conservative flux \eqref{eq:GLMPartDiscrete} 
\begin{equation}
\left(\jump{\statevec W}^T \bigstatevec{F}^{\ec,\supGLM}  
- \jump{\statevec{W}^T\,\bigstatevec{F}^{a,\supGLM}}\right)\cdot\spacevec{n}
= \left(\jump{\statevec W}^T \bigstatevec{F}^{\ec,\supGLM} - \jump{\spacevec{\Psi}^{\supGLM}}\right)\cdot\spacevec{n} = 0.
\end{equation}
To demonstrate the behavior of the non-conservative term on each element face we examine the appropriate part from \eqref{eq:surfaceTermsOnK} and substitute the coupling term \eqref{eq:GalInterface} at a single interface to find
\begin{equation}
\begin{aligned}
\statevec{W}^T \left[\Psistar - \GalDn\psi\right] &= \left(\statevec{W}^{\ma}\right)^T\left[\left(\left(\GalD\right)^{\ma}\cdot\spacevec{n}\right)\avg{\psi} - \left(\left(\GalD\right)^{\ma}\cdot\spacevec{n}\right)\psi^{\ma}\right] \\
&= \left(\left[\left(\statevec{W}^{\ma}\right)^T\left(\GalD\right)^{\ma}\right]\cdot\spacevec{n}\right)\frac{1}{2}\left(\psi^{\sl} - \psi{\ma}\right)\\
&= \left(\left[\left(\statevec{W}^{\ma}\right)^T\left(\GalD\right)^{\ma}\right]\cdot\spacevec{n}\right)\frac{1}{2}\jump{\psi}.
\end{aligned}
\end{equation}
It is straightforward to verify that each part of $\left(\GalD\right)^{\ma}$ contracts to zero in entropy space, i.e.,
\begin{equation}
\left(\statevec{W}^{\ma}\right)^T\left(\GalDs_\ell\right)^{\ma}=0,\quad \ell=1,2,3,
\end{equation}
such that
\begin{equation}
\begin{aligned}
\statevec{W}^T \left[\Psistar - \GalDn\psi\right] &= \left(\left[\left(\statevec{W}^{\ma}\right)^T\left(\GalD\right)^{\ma}\right]\cdot\spacevec{n}\right)\frac{1}{2}\jump{\psi}\\
&= \left(\spacevec{0}\cdot\spacevec{n}\right)\frac{1}{2}\jump{\psi}\\
&=0.
\end{aligned}
\end{equation}
Thus, the surface contribution of the GLM non-conservative terms directly vanish at each element face. 
\end{proof}}

\change{Before we investigate the remaining contributions of the Euler and ideal MHD components, we define} $\Bstar$ and examine the contribution of the second term from \eqref{eq:surfaceTermsOnK}. What we will find is that the surface contribution of the MHD non-conservative terms generates an additional boundary term that cancels an extraneous term left over from the analysis of the ideal MHD part of the advective fluxes.
\begin{lem}[Discretization of the non-conservative \change{ideal MHD} surface term]\label{lem:nonConsSurf}~\newline
For the second term in \eqref{eq:surfaceTermsOnK} we define 
\begin{equation}\label{eq:surfNonCons}
\Bstar = \left(\JanD\right)^{\ma}\avg{\spacevec{B}}\cdot\spacevec{n},
\end{equation}
to obtain the total contribution of the non-conservative \change{ideal MHD} surface terms
\begin{equation}
 \sum\limits_{\nu=1}^{N_{\mathrm{el}}}\;\;\; \isurfEnuN   \statevec{W}^T \left(\Bstar - \JanD\Bn\right) \hat{s} \dS 
 =  \sum\limits_{\interiorfaces}\isurfN  \avg{\wDotJan}\jump{\spacevec{B}}\cdot\spacevec{n}\,\hat{s}\dS.
\end{equation}
\end{lem} 
\begin{proof}
We first substitute the definition \eqref{eq:surfNonCons} into the second term of \eqref{eq:surfaceTermsOnK}, where, for clarity, we explicitly state that values from the current element $E_\el$ to be primary (``$\ma$''), since $\spacevec{n}$ is outward pointing
\begin{equation}\label{eq:firstStepNonCons}
\isurfEnuN   \statevec{W}^T \left(\Bstar - \JanD\Bn\right)\hat{s} \dS 
= \isurfEnuN   \left(\statevec{W}^{\ma}\right)^T \left(\left(\JanD\right)^{\ma}\avg{\spacevec{B}} - \left(\JanD\right)^{\ma} \spacevec{B}^{\ma} \right)\cdot\spacevec{n}\, \hat{s}\dS.
\end{equation}
Note that the values of $\statevec{W}$ and $\JanD$ in the contribution \eqref{eq:firstStepNonCons} are evaluated from the current element, so we have a discrete version of the property \eqref{eq:contractSource}
\begin{equation}
\left(\statevec{W}^{\ma}\right)^T\left(\JanD\right)^{\ma} = \wDotJan^{\ma}.
\end{equation}
Thus,
\begin{equation}\label{eq:secondStepNonCons}
\isurfEnuN   \left(\statevec{W}^{\ma}\right)^T \left(\left(\JanD\right)^{\ma}\avg{\spacevec{B}} - \left(\JanD\right)^{\ma} \spacevec{B}^{\ma}\right)\cdot\spacevec{n}\, \hat{s} \dS
= \isurfEnuN   \wDotJan^{\ma} \left(\avg{\spacevec{B}} - \spacevec{B}^{\ma}\right)\cdot\spacevec{n}\, \hat{s}\dS.
\end{equation}
Next, we expand the arithmetic mean to get
\begin{equation}\label{eq:thirdStepNonCons}
\isurfEnuN   \wDotJan^{\ma} \left(\avg{\spacevec{B}} -  \spacevec{B}^{\ma}\right)\cdot\spacevec{n}\, \hat{s} \dS
=\frac{1}{2}\isurfEnuN  \wDotJan^{\ma} \left(\spacevec{B}^{\sl} - \spacevec{B}^{\ma}\right)\cdot\spacevec{n} \,\hat{s}\dS.
\end{equation}
The total surface contribution of \eqref{eq:thirdStepNonCons} requires delicate consideration due to the inherent non-uniqueness of the non-conservative term at the interface. Each interface actually contributes twice to the contracted DG approximation and it is important to choose, again, a unique normal vector for each interface $\spacevec{n}$. The sum over all elements gives for an arbitrary interface contribution of the integrand 
\begin{equation}\label{eq:singleFaceTerm}
\frac{1}{2}\left(\wDotJan^{\ma}\left(\spacevec{B}^{\sl}-\spacevec{B}^{\ma}\right)\cdot\spacevec{n}\right)
+\frac{1}{2}\left(\wDotJan^{\sl}\left(\spacevec{B}^{\ma}-\spacevec{B}^{\sl}\right)\cdot(-\spacevec{n})\right)
=\avg{\wDotJan}\jump{\spacevec{B}}\cdot\spacevec{n}\,,
\end{equation}
yielding the desired result
\begin{equation}
 \sum\limits_{\nu=1}^{N_{\mathrm{el}}}\;\;\;\isurfEnuN   \statevec{W}^T \left(\Bstar - \JanD\Bn\right)\hat{s} \dS 
 =  \sum\limits_{\interiorfaces}\isurfN  \avg{\wDotJan}\jump{\spacevec{B}}\cdot\spacevec{n}\,\hat{s}\dS.
\end{equation}

\end{proof}

\begin{rem}
The prescription of non-conservative surface contributions for high-order DG methods have been previously investigated by Cheng and Shu \cite{Cheng2007} in the context of the Hamilton-Jacobi equations. Recently, these methods from the Hamilton-Jacobi community have been applied to approximate the solution of the ideal MHD equations at high-order on two-dimensional Cartesian meshes \cite{Liu2017,Valencia2017}. The current work built upon these previous results to fully generalize the extension of the non-conservative surface contributions into the three-dimensional, unstructured, curvilinear hexahedral mesh framework and re-contextualize the non-conservative surface term discretization based on specific non-conservative numerical surface approximations.
\end{rem}

The sum over all elements on the third term in \eqref{eq:surfaceTermsOnK} generates a jump in the entropy fluxes
\begin{equation}\label{eq:surfEntFluxes}
 \sum\limits_{\nu=1}^{N_{\mathrm{el}}}\;\;\;\isurfEnuN  \left(\spacevec{F}^\ent\cdot\spacevec{n}\right) \hat{s} \dS = -\sum\limits_{\interiorfaces}\isurfN \jump{\spacevec{F}^{\ent}}\cdot\spacevec{n}\,\hat{s}\dS.
\end{equation}
Now, with the results of Lemmas~\ref{lem:GLM_EC_surf} and \ref{lem:nonConsSurf} as well as the results \eqref{eq:surfaceTermsOnK2} and \eqref{eq:surfEntFluxes} we can address the remaining contributions of the Euler and ideal MHD components at the surface:
\begin{cor}[Entropy contributions of total advective surface terms]\label{cor:EC_surf}~\newline
Summing over all elements in \eqref{eq:surfaceTermsOnK} shows that the contribution of the curvilinear advective and non-conservative terms on the surface cancel, meaning
\begin{equation}
\sum_{\nu=1}^{N_{\mathrm{el}}} \Gamma_\nu = 0.
\end{equation}
\end{cor}
\begin{proof}
We note that from Lemma~\ref{lem:GLM_EC_surf} we have accounted for the cancellation of the GLM terms. Similar to the volume term analysis in Corollary~\ref{cor:EC_vol} we again separate the contributions of the Euler and ideal MHD terms. It is immediate that the Euler terms drop out from the definition of the entropy flux potential for the Euler part \eqref{eq:EulerEntFluxPot} and the separation of the entropy conserving flux condition \eqref{eq:EulerPartDiscrete}
\begin{equation}
\left(\jump{\statevec W}^T \bigstatevec{F}^{\ec,\supEuler} - \jump{\left(\statevec{W}\right)^T\,\bigstatevec{F}^{a,\supEuler}} + \jump{\spacevec{F}^\ent}\right)\cdot\spacevec{n} = 
\left(\jump{\statevec W}^T \bigstatevec{F}^{\ec,\supEuler}  - \jump{\spacevec{\Psi}^{\supEuler}}\right)\cdot\spacevec{n} = 0.
\end{equation}
For the ideal MHD contributions we make use of the entropy flux potential \eqref{eq:MHDEntFluxPot} to write
\begin{equation}\label{eq:manipulateMHD}
\resizebox{0.925\hsize}{!}{$
\begin{aligned}
\left(\jump{\statevec W}^T \bigstatevec{F}^{\ec,\supMHD}  
- \jump{\statevec{W}^T\,\bigstatevec{F}^{a,\supMHD}}\right)\cdot\spacevec{n} &=
\left(\jump{\statevec W}^T\left(\bigstatevec{F}^{\ec,\supMHD}\right) 
- \jump{\spacevec{\Psi}^{\supMHD}} + \jump{\wDotJan \spacevec{B} }\right)\cdot\spacevec{n}\\
&= \left(\jump{\statevec W}^T\left(\bigstatevec{F}^{\ec,\supMHD}\right) 
- \jump{\spacevec{\Psi}^{\supMHD}} + \jump{\wDotJan}\avg{\spacevec{B}} + \avg{\wDotJan}\jump{\spacevec{B}}\right)\cdot\spacevec{n},
\end{aligned} $}
\end{equation}
where we use a property of the jump operator
\begin{equation}\label{eq:jumpProperty}
\jump{ab} = \avg{a}\jump{b} + \avg{b}\jump{a}.
\end{equation}
We see that the first three terms on the last line of \eqref{eq:manipulateMHD} are the entropy conservative flux condition of the magnetic field components \eqref{eq:MHDPartDiscrete} and cancel. This leaves the remainder term
\begin{equation}\label{eq:remainderTermMHD}
-\sum\limits_{\interiorfaces}\isurfN \left(\jump{\statevec W}^T\bigstatevec{F}^{\ec,\supMHD}  
- \jump{\statevec{W}^T\,\bigstatevec{F}^{a,\supMHD}}\right)\cdot\spacevec{n} \hat{s}\dS
= -\sum\limits_{\interiorfaces}\isurfN \avg{\wDotJan}\jump{\spacevec{B}}\cdot\spacevec{n}\hat{s}\dS.
\end{equation}
This term is identical to the surface contribution of the non-conservative term from Lemma \ref{lem:nonConsSurf} but with opposite sign. Thus the final two terms cancel and we get the desired result
\begin{equation}
\sum_{\nu=1}^{N_{\mathrm{el}}} \Gamma_\nu = 0.
\end{equation}
\end{proof}

\subsection{Analysis of the viscous and resistive parts}\label{Sec:ResParts}

Lastly, since the discussion of the curvilinear advective and non-conservative parts is now complete, we focus on the resistive parts, namely the last row of the first equation in \eqref{eq:schemeFinalcontr}. Again, we first have to select appropriate numerical fluxes at the interfaces. Thus, we use the computationally simple Bassi-Rebay (BR1) type approximation \cite{Bassi&Rebay:1997:B&F97} in terms of the discrete entropy variables and gradients \cite{Gassner2017}
\begin{equation}\label{eq:BR1Fluxes}
\statevec F_n^{v,*} = \avg{\bigstatevec{F}^v}\cdot\spacevec{n},\quad \statevec{W}^*=\avg{\statevec{W}}.
\end{equation}
With the results from the previous section we are able to prove the main result of this work,

\begin{thm}[Discrete entropy stability of the curvilinear DGSEM for the resistive GLM-MHD equations]\label{thm_discrete}
The curvilinear DGSEM for the resistive GLM-MHD equations \eqref{eq:schemeFinal} with 
\begin{equation}
\change{\Psistar = \left(\left(\GalD\right)^{\ma}\cdot\spacevec{n}\right)\avg{\psi}},\quad 
\Bstar = \left(\JanD\right)^{\ma}\avg{\spacevec{B}}\cdot\spacevec{n}, \quad
\statevec{F}_n^{a,\ast} = \bigstatevec{F}^{a,\#}\cdot\spacevec{n} = \bigstatevec{F}^\ec\cdot\spacevec{n}\quad
\end{equation}
and the viscous interface fluxes \eqref{eq:BR1Fluxes} is entropy stable, i.e. for a closed system (periodic boundary conditions) the discrete total entropy is a decreasing function in time
\begin{equation}
\frac{d\bar{S}}{dt}\leq 0.
\end{equation}
\end{thm}
\begin{proof}
From Corollaries~\ref{cor:EC_vol} and \ref{cor:EC_surf} we know that the volume, surface and non-conservative terms of the advective portions of the resistive GLM-MHD equations cancel in entropy space. The remaining parts of the contracted DG approximation are
\begin{equation}\label{eq:schemeFinalInProof}
 \begin{aligned}
\iprodN{\interpolant J\statevec U_t,\statevec{W} } 
&= \iprodN{\DDs\bigcontravec F^v,\statevec{W} }
+\isurfEN\statevec{W}^T\left(\statevec F_n^{v,*} - \statevec F_n^{v} \right) \hat{s}\dS + \iprodN{\interpolant J \statevec{R},\statevec{W} }\hfill, \\
\iprodN{\interpolant J\bigstatevec Q,\bigcontravec F^v  }
&= \isurfEN {{{\statevec W}^{*,T}}\left( {\bigstatevec F^v }  \cdot \spacevec{n}\right)\hat{s}\dS}  
- \iprodN{\statevec W, {\DDs \bigcontravec F^v } }. \hfill \\ 
 \end{aligned}
\end{equation}
We consider the first term of the second equation and insert the alternate form of the viscous flux rewritten in terms of the gradient of the entropy variables as in the continuous analysis \eqref{Kgradient}. We use the known property that the viscous and resistive coefficient matrix $\twentysevenMatrix{K}$ is symmetric positive semi-definite for the resistive MHD equations to see that
\begin{equation}
\label{eq:disc_visc_volint_estimate}
\iprodN{\interpolant{J}\bigstatevec{Q},\bigcontravec{F}_v} = \iprodN{\interpolant{J}\bigstatevec{Q},\twentysevenMatrix{K}\,\bigstatevec{Q}}\geqslant \min\limits_{E,N}(\interpolant{J}) \iprodN{\bigstatevec{Q},\twentysevenMatrix{K}\,\bigstatevec{Q}}\geq 0,
\end{equation}
Next, we insert the second equation of \eqref{eq:schemeFinalInProof} into the first and use the estimate \eqref{eq:disc_visc_volint_estimate} to get
\begin{equation}\label{eq:schemeFinalInProof2}
\iprodN{\interpolant J\statevec U_t,\statevec{W} } 
\leq \isurfEN\left(\statevec{W}^T\left(\statevec F_n^{v,*} - \statevec F_n^{v} \right) + {\statevec W^{*,T}}{\statevec F_n^v }  \right)\hat{s}\dS + \iprodN{\interpolant J \statevec{R},\statevec{W} }
\end{equation}
From Remark \ref{ES_damp} we know, that we can also ignore the discrete damping source term without violating the inequality. After summing over all elements we
can replace the left hand side by the total entropy derivative according to \eqref{totalEntr} and obtain
\begin{equation}\label{eq:schemeFinalInProof3}
\frac{d \overline{S}}{d t}
\leq -\sum\limits_\interiorfaces \isurfN \left(\jump{\statevec{W}}^T\,\statevec F_n^{v,*} 
- \jump{\statevec{W}^T\bigstatevec F^{v}}\cdot\spacevec{n}  
+ \left(\statevec{W}^{*}\right)^T\jump{\bigstatevec F^v}\cdot\spacevec{n}\right)\hat{s}\dS.
\end{equation}

In Gassner et al. \cite{Gassner2017} it is shown that for the BR1 choice \eqref{eq:BR1Fluxes} all the surface terms in \eqref{eq:schemeFinalInProof3} vanish identically to zero if periodic boundary conditions are considered, yielding our desired result
\begin{equation}\label{eq:entr_stab_bounds}
\frac{d \overline{S}}{d t} \leq 0,
\end{equation}
which shows that the discrete total entropy is decreasing in time, i.e. that the DGSEM for the resistive GLM-MHD equations is entropy stable.
\end{proof}

\begin{rem}
It is desirable to introduce additional upwind-type dissipation for advection dominated problems through the choice of the numerical advection fluxes by replacing the EC fluxes at element interfaces with e.g. the ES fluxes \eqref{ES_flux}.
\end{rem}

\section{Numerical verification}\label{Sec:NumVer}
In this section, we present numerical tests to validate the theoretical findings of the previous sections. We start with a demonstration of the high-order accuracy for the resistive GLM-MHD system by a manufactured solution and the entropy conservation for the ideal GLM-MHD system, both on curved meshes. \change{Next, we verify the GLM divergence cleaning capability of the scheme.} Finally, we provide an example, in which every piece of the presented numerical solver is exercised, to demonstrate the increased robustness and applicability of the entropy stable DG approximation for the resistive GLM-MHD equations. Specifically, we use a three-dimensional, viscous version of the well-known Orszag-Tang vortex, e.g. \cite{Elizarova2015} to show the value of the entropy stable framework in conjunction with GLM hyperbolic divergence cleaning in providing numerical stability.

All numerical computations in this work are performed with the open source, three-dimensional curvilinear split form DG framework for the resistive MHD equations FLUXO (\url{www.github.com/project-fluxo}). FLUXO is written in modern Fortran and its non-blocking pure MPI parallelization shows excellent strong and weak scaling on modern HPC architectures. The three-dimensional high order meshes for the simulations are generated with the open source tool HOPR (\url{www.hopr-project.org}). All simulation results are obtained with an explicit five stage, fourth order accurate low storage Runge-Kutta scheme \cite{Carpenter&Kennedy:1994}, where a stable time step is computed according to the adjustable coefficient CFL$\in(0,1]$, the local maximum wave speed, the viscous eigenvalues, and the relative grid size, e.g. \cite{gassner2011}. Unless otherwise stated, we set the damping parameter $\alpha=0$ and the GLM propagation speed $c_h$ to be proportional to the maximum advective wave speed.

\subsection{Curved periodic meshes}\label{Sec:periodic_meshes}
In this work, all numerical results are obtained on fully periodic curved meshes. The discussion on entropy stable boundary conditions for the resistive MHD equations is ongoing research. To generate the curved mesh, we first define the high order element nodes on a standard cartesian mesh with periodic boundary conditions, in the space variables $\spacevec{\chi}=(\chi_1,\chi_2,\chi_3)^T$. 
The curved mesh  is then generated by applying a transformation function to all high order nodes mapping them to physical space $\spacevec{x}$ defined by
\begin{equation}
 \label{eq:meshtransform}
 x_\ell = \chi_\ell+0.1\sin(\pi\chi_1)\sin(\pi\chi_2)\sin(\pi\chi_3)\,,\quad \ell=1,2,3\,.
\end{equation} 
As shown in Figure~\ref{fig:meshes}, two mesh types will be considered throughout this section. Type (a) with flat periodic boundaries and curved element faces inside, using $\vec{\chi}\!\in\![0,1]^3$ and type (b) with curved element interfaces, using $\spacevec{\chi}\!\in\![-0.6,1.4]\!\times\![-0.8,1.2]\!\times\![-0.7,1.3]$, being still fully periodic and conforming.
\begin{figure}[h!]
\begin{center}
\begin{minipage}[b]{0.48\linewidth}\centering
   \includegraphics[height=0.25\textheight]{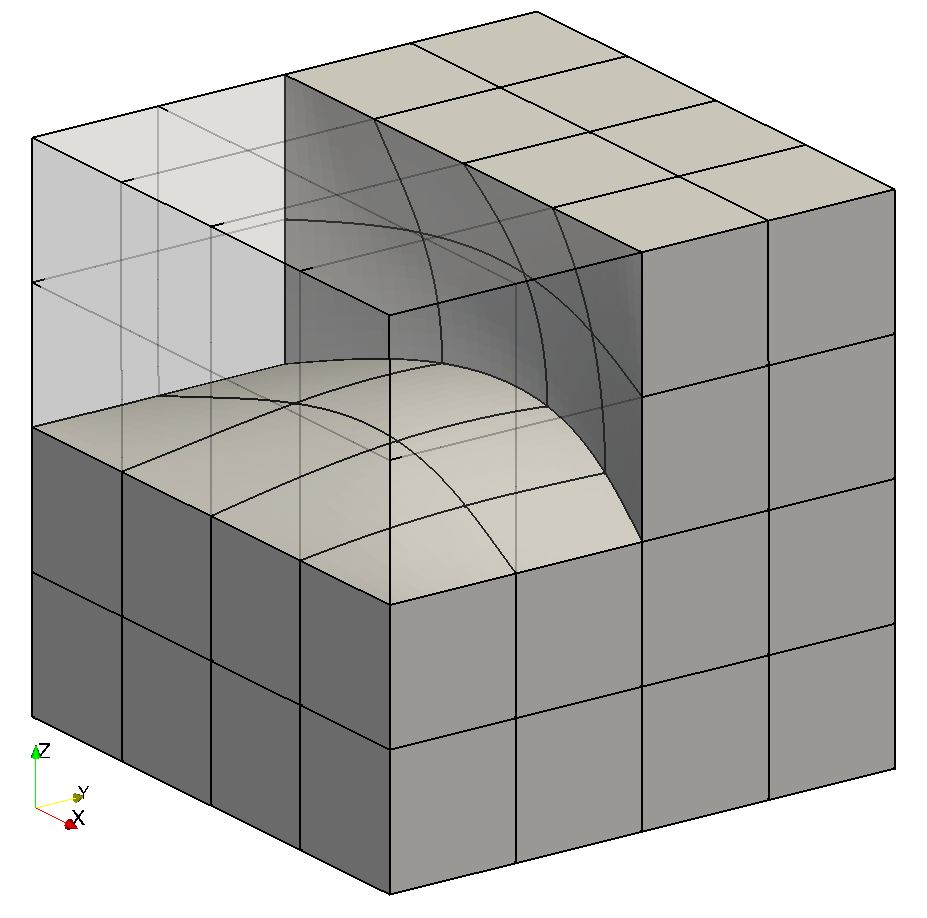}\\[-1ex]
   \textbf{(a) } \\ $\vec{\chi}\!\in\![0,1]^3$
 \end{minipage}
 \begin{minipage}[b]{0.48\textwidth}   \centering
   \includegraphics[height=0.25\textheight]{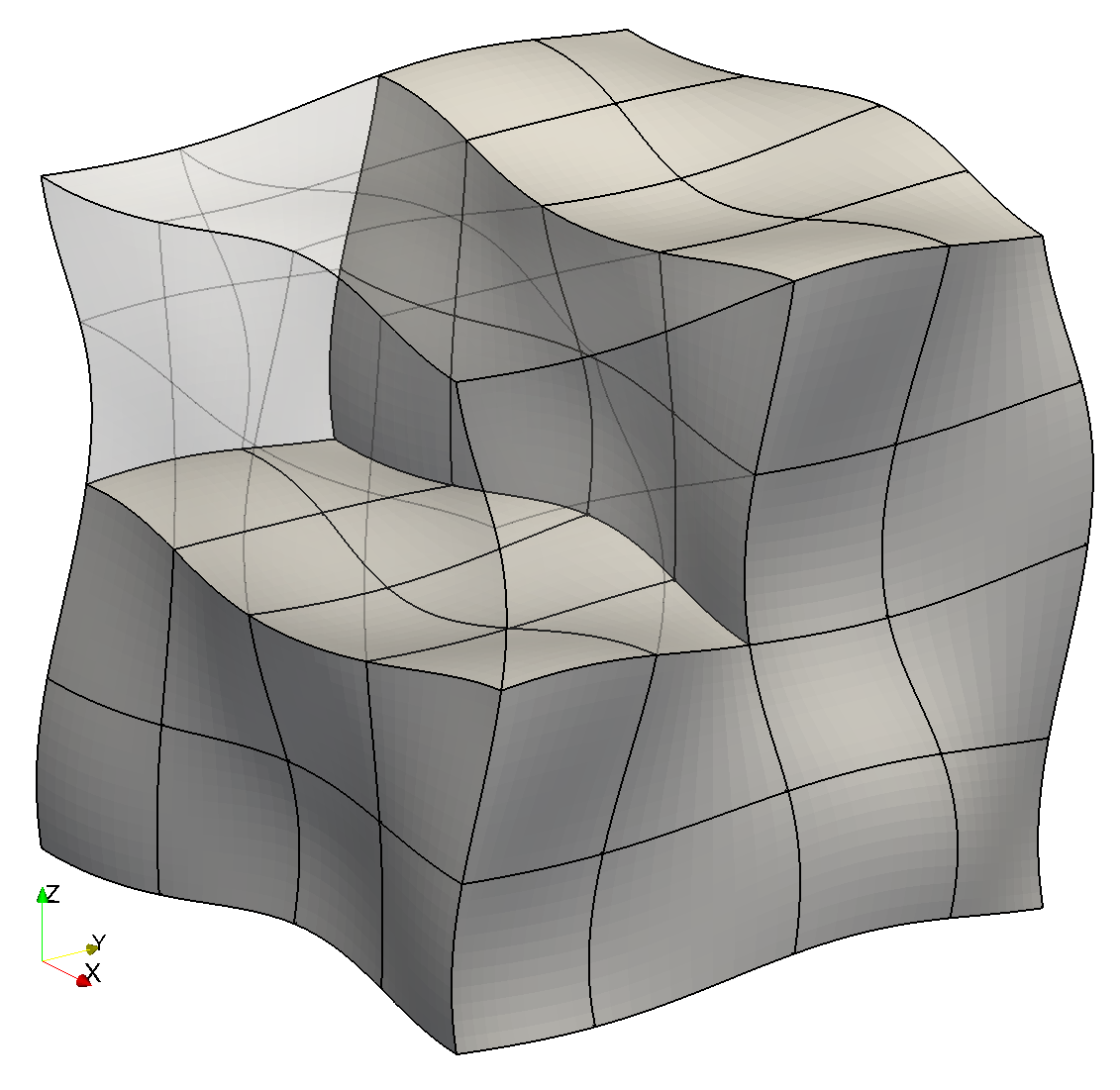}\\[-1ex]
   \textbf{(b)} \\$\spacevec{\chi}\!\in\![-0.6,1.4]\!\times\![-0.8,1.2]\!\times\![-0.7,1.3]$
 \end{minipage}
\caption{Two types of fully periodic curved meshes, shown for $4^3$ elements, generated via the transformation function \eqref{eq:meshtransform}. }
\label{fig:meshes}
\end{center}
\end{figure}

\subsection{High-Order Convergence on 3D Curvilinear Meshes}\label{Sec:Conv_MS}

In order to verify the high-order approximation of the entropy stable DG discretization \eqref{eq:schemeFinal} for the resistive GLM-MHD system \eqref{resGLMMHD}, we run a convergence test with the method of manufactured solutions. To do so, we assume a solution of the form
\begin{equation}
\statevec{u} = \left[h,h,h,0,2h^2+h,h,-h,0,0\right]^T \text{ with } h = h(x,y,z,t) = 0.5 \sin(2\pi(x+y+z-t))+2.
\label{manufac}
\end{equation}
The main advantage of this choice is, that it is symmetric and spatial derivatives cancel with temporal derivatives, i.e.
\begin{equation}
h_x=h_y=h_z= -h_t. 
\end{equation} 
Hence, the additional residual for the resistive GLM-MHD system reads
\begin{equation}
\statevec{u}_t + \vec{\nabla} \cdot \bigstatevec{f}^a(\statevec{u}) - \vec{\nabla} \cdot \bigstatevec{f}^v(\statevec{u},\vec{\nabla} \statevec{u}) = \begin{pmatrix} h_x \\ h_x + 4hh_x \\ h_x + 4hh_x \\ 4hh_x \\ h_x + 12hh_x - 6\resistivity(h_x^2+h h_{xx})-6\viscosity h_{xx}/\text{Pr} \\ h_x - 3\resistivity h_{xx} \\ -h_x + 3\resistivity h_{xx} \\ 0 \\ 0 \end{pmatrix}
\label{residual}
\end{equation}
for $\gamma = 2$ and $\text{Pr}\!=\!0.72$. In order to solve the inhomogeneous problem, we subtract the residual from the approximate solution in each Runge-Kutta step. Moreover, we run the test case up to the final time $T\!=\!1.0$ and we set $\viscosity=\resistivity\!=\!0.005$. For all computations we use mesh type (b) from Figure~\ref{fig:meshes} with a varying number of elements.

Finally, we obtain the convergence results illustrated in Tables \ref{EOCres3} and \ref{EOCres4} that confirm the high-order accuracy of the scheme on curvilinear meshes in three spatial dimensions. The errors of $v_2$ and $B_2$ are similar to the ones of $v_1$ and $B_1$ and, thus, not presented in the tables.

\newcommand\stent{\rule[-1ex]{0pt}{3.2ex}}
\begin{table}[h!]
	\centering
		\begin{tabular}{|c|c|c|c|c|c|c|c|}\hline
			 \small $N_{\text{el}}$ &  $L^2(\rho)$ & $L^2(v_1)$ & $L^2(v_3)$ & $L^2(p)$ & $L^2(B_1)$ & $L^2(B_3)$ & $L^2(\psi)$ \stent\\
			\hline 
			$4^3$ & 1.62E-01  & 1.74E-01 &  1.42E-01 &  3.42E-01 &  1.19E-01 &  1.65E-02 &  2.03E-02\stent\\
			\hline
			$8^3$ & 6.11E-03 &  8.38E-03  &  6.13E-03 &  1.59E-02 &  3.51E-03 &  2.18E-03 & 1.19E-03\stent\\
			\hline
			$16^3$ & 2.40E-04 &  5.02E-04 & 3.40E-04 &  1.18E-03 &  1.39E-04 &  1.09E-04 &  6.06E-05\stent\\
			\hline
			$32^3$ & 1.93E-05 &  2.51E-05 & 1.54E-05 &  7.42E-05 &  7.56E-06 &  2.80E-06 &  3.40E-06\stent\\
			\hline
			\hline 	
			avg EOC & 4.34	& 4.25	&	4.39	& 4.06 &	4.65 &	4.18	& 4.18 \stent\\
			\hline 
		\end{tabular}
		\caption{$L^2$-errors and EOC of manufactured solution test for resistive GLM-MHD and $N\!=\!3$.}\label{EOCres3}
\end{table}
\begin{table}[h!]
	\centering
		\begin{tabular}{|c|c|c|c|c|c|c|c|}\hline
			 \small $N_{\text{el}}$ & $L^2(\rho)$ & $L^2(v_1)$ & $L^2(v_3)$ & $L^2(p)$ & $L^2(B_1)$ & $L^2(B_3)$ & $L^2(\psi)$ \stent\\
			\hline 
			$4^3$ & 1.11E-01 &  1.66E-01 &  1.06E-01 &  2.76E-01 &  7.50E-02 & 5.95E-02 &  1.87E-02\stent\\
			\hline
			$8^3$ & 2.49E-03  & 5.01E-03  &  3.96E-03 &  1.17E-02 &  2.39E-03 &  2.02E-03 &  8.44E-04\stent\\
			\hline
			$16^3$ & 7.88E-05 &  1.45E-04 & 1.04E-04 &  3.92E-04 &  5.84E-05 & 4.52E-05 &  3.03E-05\stent\\
			\hline
			$32^3$ & 2.75E-06  & 3.51E-06 &  2.44E-06 &  1.17E-05 &  1.46E-06 & 7.18E-07 &  1.50E-06\stent\\
			\hline
			\hline 	
			avg EOC & 5.10	& 5.18	& 5.13	& 4.84	& 5.21 &	5.45	& 4.54 \stent\\
			\hline 
		\end{tabular}
		\caption{$L^2$-errors and EOC of manufactured solution test for resistive GLM-MHD and $N\!=\!4$.}\label{EOCres4}
\end{table}

We note, if we apply an entropy conservative approximation to this test case, the convergence order would exhibit an odd/even effect depending on the polynomial degree. This phenomenon of optimal convergence order for even and sub-optimal convergence order for odd polynomial degrees has been previously observed for entropy conservative DG approximations, e.g. \cite{gassner_skew_burgers,Gassner:2016ye}.

\subsection{Entropy Conservation on 3D Curvilinear Meshes}\label{Sec:EC}
Next, we demonstrate the discrete entropy conservation of the newly proposed method, both with and without the GLM terms. As such, we deactivate the numerical dissipation introduced by the interface stabilization terms \eqref{ES_flux} and set $\viscosity=\resistivity\!=\!0$ to remove the resistive terms, because entropy conservation only applies to the \textit{ideal} equations.

We choose mesh type (b) from Figure~\ref{fig:meshes} with a resolution of $7^3\!=\!343$ elements. As well-resolved smooth solutions would result in very small changes of total entropy anyway, we purposely select a more challenging test case initializing a moving spherical blast wave in a constant magnetic field. The inner and outer states are given in Table~\ref{tab:ECtest_states} and are blended over a distance of approximately $\delta_0$ with the blending function 
\begin{equation}
 \label{eq:ECtest_exfunc}
 \statevec{u}=\frac{\statevec{u}_\text{inner}+\lambda \statevec{u}_\text{outer}}{1+\lambda}\,, \quad \lambda=\exp\left[\frac{5}{\delta_0}(r-r_0)\right]\,,\quad r = \|\spacevec{x}-\spacevec{x}_c\|\,.
\end{equation} 
The parameters are $\spacevec{x}_c=(0.3,0.4,0.2),r_0=0.3,\delta_0=0.1$ and we set $\gamma\!=\! 5/3$.  A visualization of the initial state is shown in Figure~\ref{fig:ECtest_visu}.
\begin{table}[htbp!]
	\centering
		\begin{tabular}{|c|c|c|c|c|c|c|c|c|c|}\hline
			     & $\rho$ 	& $v_1$	& $v_2$	& $v_3$	& $p$ 	& $B_1$	& $B_2$	& $B_3$	& $\psi$ \stent\\\hline
			inner & $1.2$	& $0.1$	& $0.0$	& $0.1$	&$0.9$	& $1.0$	& $1.0$	& $1.0$	& $0.0$ \stent\\\hline
			outer & $1.0$	& $0.2$	& $-0.4$& $0.2$	&$0.3$	& $1.0$	& $1.0$	& $1.0$	& $0.0$ \stent\\\hline 
		\end{tabular}
		\caption{Inner and outer primitive states for the entropy conservation test.}
\label{tab:ECtest_states}
\end{table}

We know, that the entropy conservative scheme is essentially dissipation-free by construction. Therefore, even in the case of discontinuous solutions, the only dissipation introduced into the approximation is through the time integration scheme. Hence, we can use the error in the total entropy as a measure of the temporal convergence of the method. In particular, the total entropy at time $t$ is computed in the discrete approximation by \eqref{totalEntr}. We know that by reducing the CFL number, the dissipation of the time integration scheme decreases and entropy conservation is captured more accurately, e.g. \cite{Fjordholm2011}. Figure \ref{fig:ECtest_dt_convergence} confirms this precise behavior, for $N\!=\!4$ and $N\!=\!5$ using a fourth order Runge-Kutta time integrator. For small enough time steps, the EC scheme conserves entropy discretely up to numerical roundoff even for this discontinuous test case. The same behavior is observed for the EC scheme without the GLM terms. We also demonstrate, that the entropy change for the entropy stable scheme including additional dissipation through the numerical surface fluxes \change{is independent of the time step.}
\begin{figure}[htbp!]
\begin{center}
\includegraphics[width=0.75\textwidth]{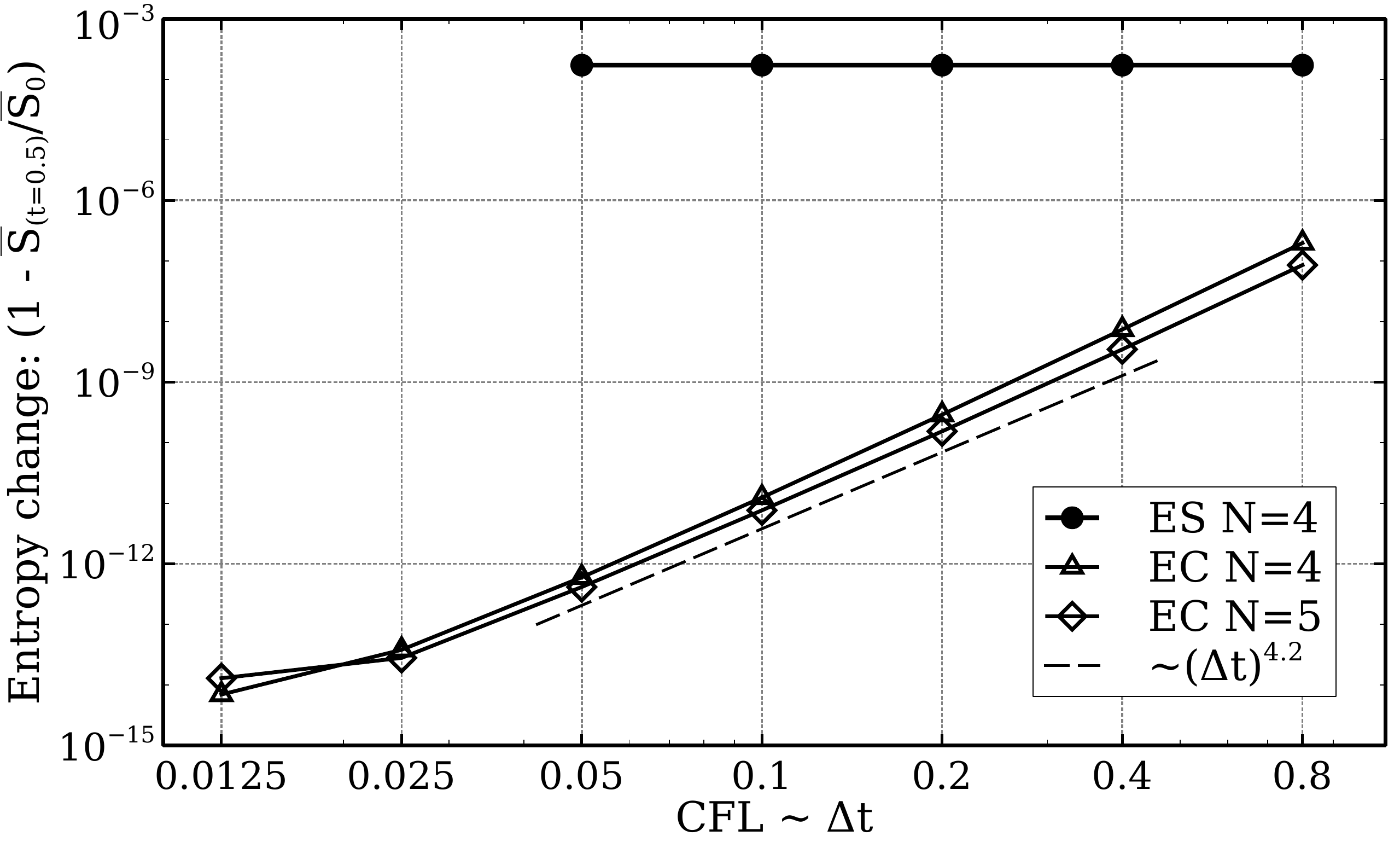}
\caption{Log-log plot of entropy change from the initial entropy $\overline{S}_0$ to $\overline{S}(t=0.5)$,  over the timestep.}
\label{fig:ECtest_dt_convergence}
\end{center}
\end{figure}

As the EC scheme has virtually no built-in dissipation, all EC simulations crash due to negative pressure, independently of the CFL number, after $T\!\approx \!0.72,0.61$ for $N\!=\!4,5$ respectively. In Figure~\ref{fig:ECtest_visu}, the $N\!=\!4$ simulation results at $T\!=\!0.5$ of the EC and ES scheme are visualized and clearly underline the difference between the EC and the ES results, although dissipation is only added at the element interfaces. 
\begin{figure}[htbp!]
\begin{center}
   \includegraphics[width=0.98\textwidth]{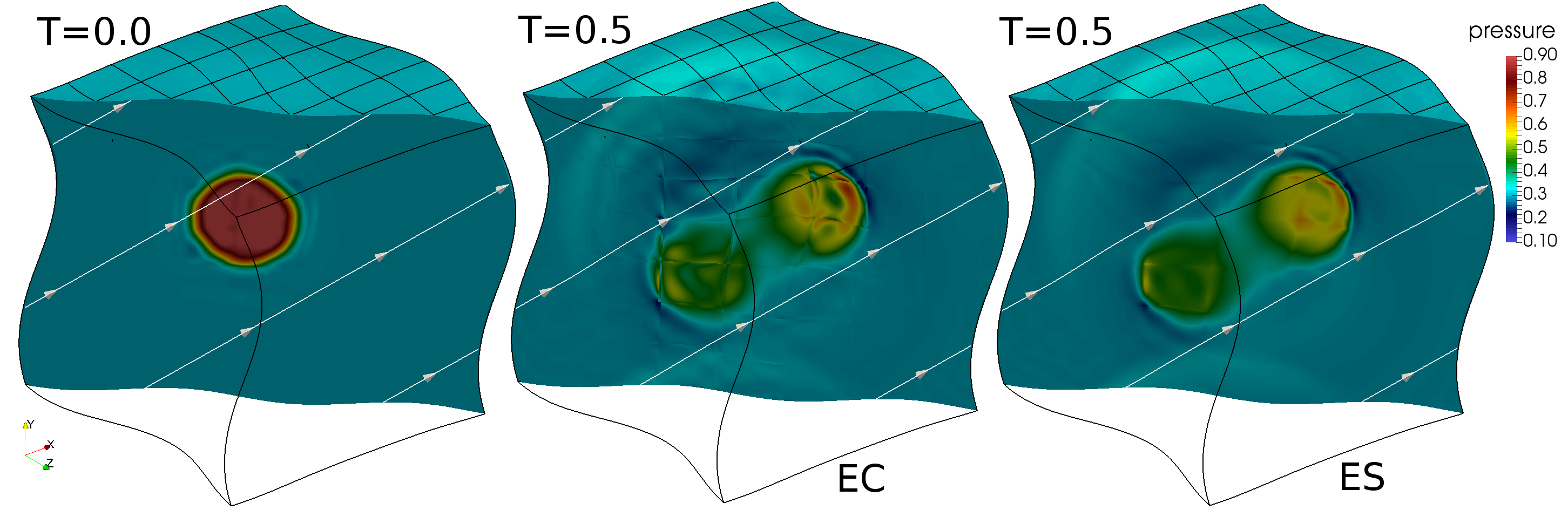}
\caption{Entropy conserving test of ideal GLM-MHD equations for $N\!=\!4$: Pressure distribution, left at initialization, the EC scheme in the middle  at $T=0.5$  and the ES scheme on the right. }
\label{fig:ECtest_visu}
\end{center}
\end{figure}

In all simulations, we use the discretely divergence-free metric terms computed via the discrete curl form \cite{Kopriva:2006er}, a necessary condition pointed out in the entropy conservation proofs. We note that, when using the $N\!=\!4$ EC scheme with cross-product metrics instead, the discrete entropy conservation property is broken with an absolute entropy change $>10^{-8}$ and even with the wrong sign for CFL numbers $\leq 0.4$.

\change{\subsection{GLM Divergence Cleaning}}\label{Sec:Gauss}

\change{In order to demonstrate the reduction of the divergence error in the magnetic field, we use a maliciously chosen non-divergence-free initialization in $\Omega = \left[0,1\right]^3$ defined by a Gaussian pulse in the $x-$component of the magnetic field proposed in \cite{altmann2012}:
\begin{equation*}
 \rho(x,y,0) = 1, ~~~ E(x,y,0) = 6, ~~~ B_1(x,y,0) = \exp\left(-\frac{1}{8} ~ \frac{(x-0.5)^2+(y-0.5)^2+(z-0.5)^2}{0.0275^2}\right).
\end{equation*}
The other initial values are set to zero and the boundaries are periodic. Again we set $\gamma = 5/3$ and turn off physical viscosity. In Figure \ref{DivError} we illustrate the time evolution of the normalized discrete divergence error measured in terms of $\|\vec{\nabla} \cdot \spacevec{B}\|_{L^2(\Omega)}$ for $N=3$ and mesh type (a) with $20^3$ elements. Here, we show the simulation results of the ideal GLM-MHD approximation, in which the divergence error is solely propagated through the domain. However, due to the periodic nature of the boundaries, it is known that the divergence errors will simply advect back into the domain \cite{Dedner2002} with only minimal damping due to the high-order nature of the DG scheme. As such, we provide a study comparing the no divergence cleaning case ($c_h=0$) against the GLM with additional damping and varying the value of $\alpha$ in \eqref{damping}. We see in Figure \ref{DivError} that without GLM divergence cleaning the simulation even crashes at $t\approx 3.4$ and for $\alpha>0$ the divergence error decays over time.\\
}

\begin{figure}[h!]
\begin{center}
\includegraphics[width=9cm]{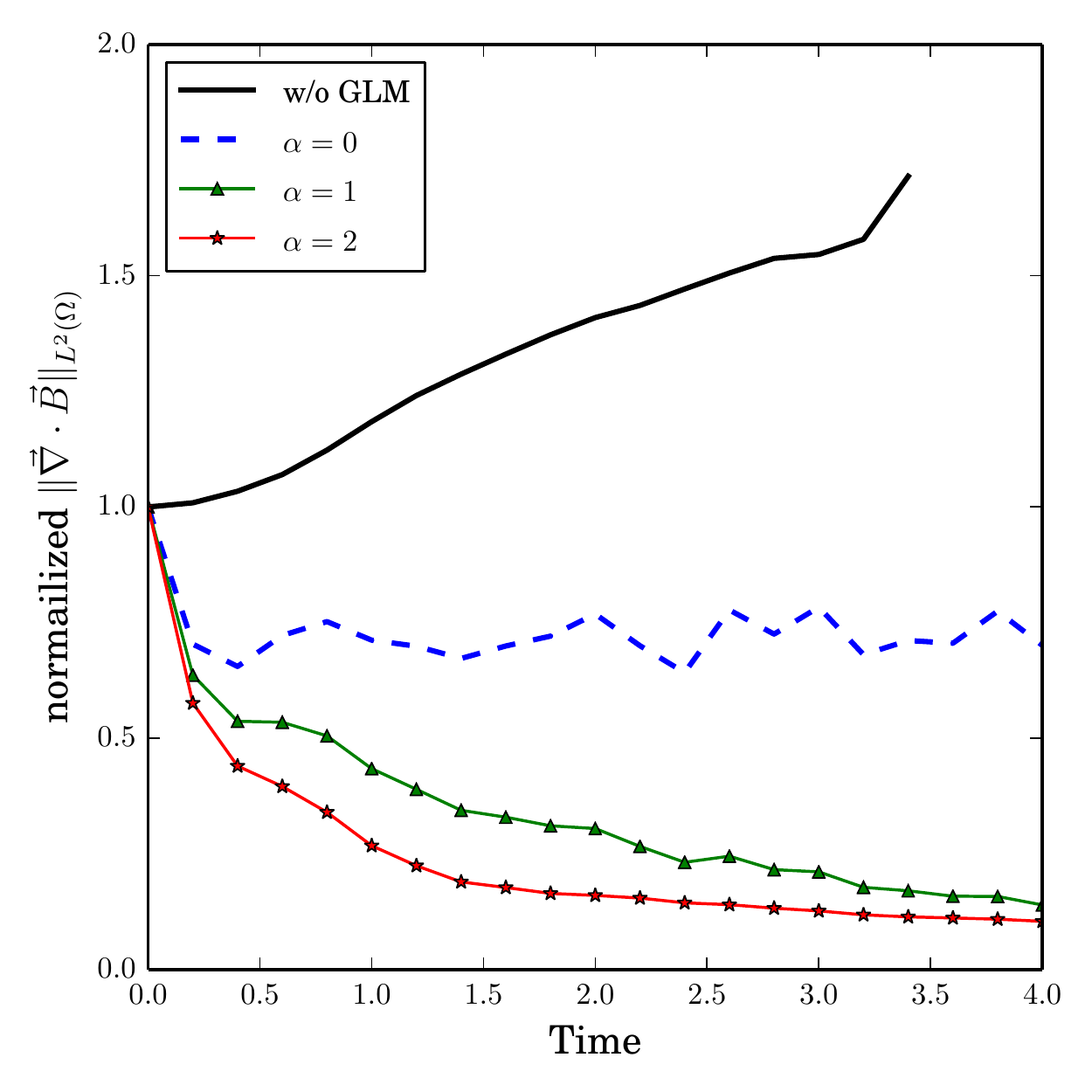}
\caption{Temporal evolution of the normalized discrete $L^2$ error in the divergence-free condition for $N=3$ in each spatial direction on $20^3$ curved elements.}
\label{DivError}
\end{center}
\end{figure}

\subsection{Robustness test}\label{Sec:Robustness}

Lastly, we use the Orszag-Tang vortex to demonstrate the increased robustness of the entropy stable approximation including GLM divergence cleaning. To do so, we use a three-dimensional extension of this test case proposed by Elizarova and Popov \cite{Elizarova2015} as well as consider a viscous version of the standard test problem, similar to Altmann \cite{altmann2012}. This test is particularly challenging because the initial conditions are simple and smooth, but evolve to contain complex structures and energy exchanges between the velocity and magnetic fields. 

\begin{table}[htbp!]
	\change{\centering
		\begin{tabular}{|c|c|c|c|c|c|c|c|c|}\hline
		 $\rho$ 	& $v_1$	& $v_2$	& $v_3$	& $p$ 	& $B_1$	& $B_2$	& $B_3$	& $\psi$ \stent\\\hline
	$\frac{25}{36\pi}$	& $-\sin(2\pi z)$	& $\sin(2\pi x)$	& $\sin(2\pi y)$	&$\frac{5}{12\pi}$	& $-\frac{1}{4\pi}\sin(2\pi z)$	& $\frac{1}{4\pi}\sin(4\pi x)$	& $\frac{1}{4\pi}\sin(4\pi y)$	& $0.$ \stent\\\hline
		\end{tabular}
		\caption{Initial condition the the 3D Orzag-Tang vortex in in primivtive variables.}
\label{tab:OTVtest_prim}}
\end{table} 

The domain is $\Omega = [0,1]^3$ with periodic boundary conditions \change{ and the initial data is given in Table \ref{tab:OTVtest_prim}. We choose $\gamma = 5/3$ and a Prandtl number of $\mathrm{Pr}\!=\!0.72$, and set the viscosity and resistivity parameters to
\begin{equation}\label{eq:resParameters}
\viscosity = 10^{-3},\quad\resistivity = 6\times 10^{-4}\,
\end{equation}
which correspond} to a kinematic Reynolds number (Re) of $1000$ and a magnetic Reynolds number ($\mathrm{Re_m}$) of approximately $1667$. The initial conditions evolve to a final time of $T\!=\!0.5$. The simulation uses a polynomial order $N\!=\!7$ in each spatial direction on a $10\!\times\!10\!\times\! 10$ element internally curved hexahedral mesh, Fig. \ref{fig:meshes}(a). 

We successfully run the entropy stable DGSEM for the resistive GLM-MHD equations up to the final time. The magnetic energy, $\frac{1}{2}\|\spacevec{B}\|^2$, is visualized in Fig. \ref{fig:OTV} for the initialization, time $T\!=\!0.25$ as well as the final time $T\!=\!0.5$.
\begin{figure}[!ht]
   \centering
   \includegraphics[width=1.0\textwidth]{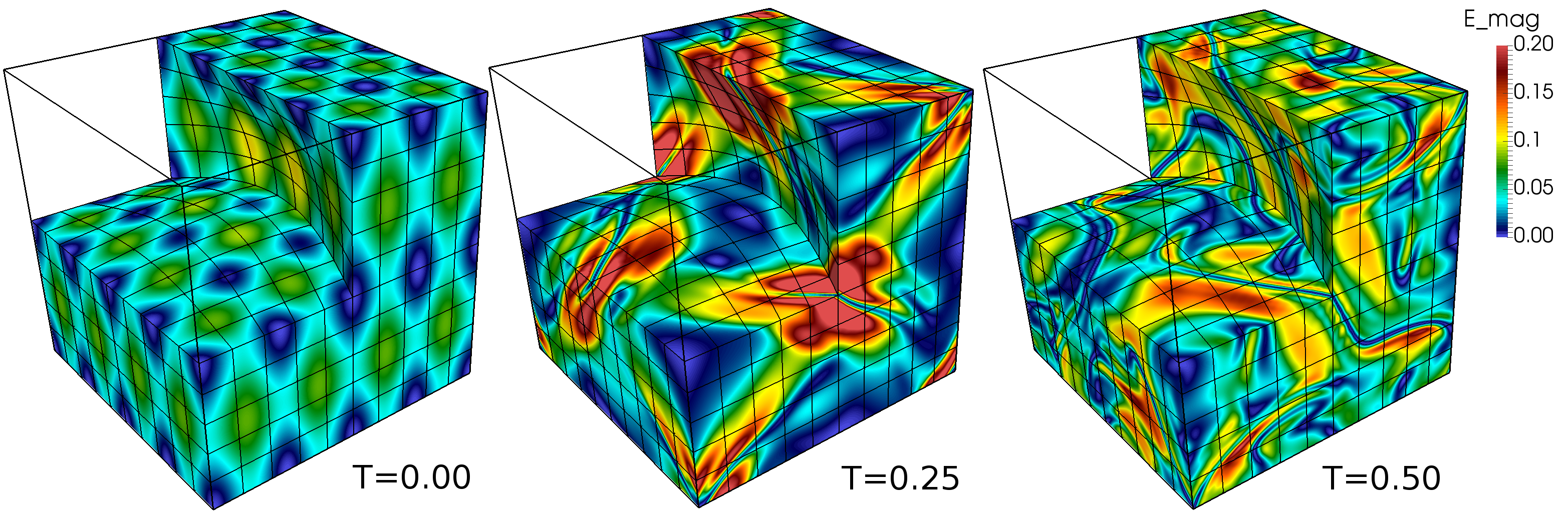}
    \caption{Visualization of the time evolution of the magnetic energy for the three-dimensional viscous Orszag-Tang vortex with polynomial order $N\!=\!7$ in each spatial direction on a $10\!\times\!10\!\times\! 10$ internally curved hexahedral mesh, e.g. Fig. \ref{fig:meshes}(a). The upper left part of the mesh is not shown to visualize interior element boundaries.}
    \label{fig:OTV}
\end{figure}
Further, we find that the standard DGSEM, e.g. \cite{Hindenlang201286}, for this resistive GLM-MHD model crashes at $T\!\approx\! 0.42$ due to the generation of negative pressure values. Also, we shrank the time step for the entropy stable as well as the standard DG runs and find the same numerical stability results. This reinforces that the numerical instabilities in the approximate flow are caused by errors other than those introduced by the time integration scheme. These results demonstrate the strong benefits of using an entropy stable formulation for such a complex configuration.  


\section{Conclusions}\label{Sec:Concl}

In this work, we have presented a novel entropy stable nodal DG scheme for the resistive MHD equations including GLM divergence cleaning on curvilinear unstructured hexahedral meshes. First, we have analyzed the continuous entropic properties of the underlying system in order to demonstrate that the resistive GLM-MHD equations satisfy the entropy inequality. This also provided guidance for the semi-discrete analysis. We carefully have split the different terms in the DGSEM and investigated the individual discrete entropy contributions step by step. 

We have validated our theoretical analysis with several numerical results. In doing so, we have introduced the open source framework FLUXO (\url{www.github.com/project-fluxo}) that implements a high-order, entropy stable, nodal DGSEM for the resistive GLM-MHD equations. In particular, we have shown with the method of manufactured solutions, that the entropy stable DGSEM solver described in this work is high-order accurate on curvilinear meshes. Next, we have verified the entropy conservative nature \change{as well as the GLM divergence cleaning} of the underlying scheme for the ideal part of the equations, also on three-dimensional curvilinear meshes. Finally, the last numerical test of a three-dimensional viscous Orszag-Tang vortex reveals that the entropy stable discretization with hyperbolic divergence cleaning significantly improves the robustness compared to the standard DGSEM.

\section*{Acknowledgements}
Gregor Gassner thanks the European Research Council for funding through the ERC Starting Grant ``An Exascale aware and Un-crashable Space-Time-Adaptive Discontinuous Spectral Element Solver for Non-Linear Conservation Laws'' (\texttt{Extreme}), ERC grant agreement no. 714487.
Dominik Derigs acknowledges the support of the Bonn-Cologne Graduate School for Physics and Astronomy (BCGS), which is funded through the Excellence Initiative, as well as the Sonderforschungsbereich (SFB) 956 on the ``Conditions and impact of star formation''.
Florian Hindenlang thanks Eric Sonnendr\"ucker and the Max-Planck Institute for Plasma Physics in Garching for their constant support.
This work was partially performed on the Cologne High Efficiency Operating Platform for Sciences (CHEOPS) at the Regionales Rechenzentrum K\"oln (RRZK).

\section*{References}

\bibliography{Paper_ESDGSEM_GLMMHD_Theory}

\appendix

\section{Dissipation matrices for entropy variables}\label{Sec:DisMatrix}

In this section we explicitly state the missing block matrices necessary to define the diffusion terms for the entropy stable approximation of the resistive GLM-MHD equations from Lemma \ref{lemma1}:
\begin{equation}\label{eq:matK12}
\nineMatrix{K}_{12} = \frac{1}{w_5}\begin{pmatrix}
0 & 0 & 0 & 0 & 0 & 0 & 0 & 0 & 0 \\[0.15cm]
0 & 0 & \frac{2\viscosity}{3} & 0 & -\frac{2\viscosity w_3}{3w_5} & 0 & 0 & 0 & 0 \\[0.15cm]
0 & -\viscosity & 0 & 0 & \frac{\viscosity w_2}{w_5} & 0 & 0 & 0 & 0 \\[0.15cm]
0 & 0 & 0 & 0 & 0 & 0 & 0 & 0 & 0 \\[0.15cm]
0 & \frac{\viscosity w_3}{w_5} & -\frac{2\viscosity w_2}{3w_5} & 0 & -\frac{\viscosity w_2 w_3}{3w_5^2} & 0 & 0 & 0 & 0 \\[0.15cm]
0 & 0 & 0 & 0 & \frac{\resistivity w_6 w_7}{w_5^2} & -\frac{\resistivity w_7}{w_5} & 0 & 0 & 0 \\[0.15cm]
0 & 0 & 0 & 0 & 0 & 0 & 0 & 0 & 0 \\[0.15cm]
0 & 0 & 0 & 0 & -\frac{\resistivity w_6}{w_5} & \resistivity & 0 & 0& 0 \\[0.15cm]
0 & 0 & 0 & 0 & 0 & 0 & 0 & 0 & 0 \\[0.15cm]
\end{pmatrix}
\end{equation}

\begin{equation}
\nineMatrix{K}_{13} = \frac{1}{w_5}\begin{pmatrix}
0 & 0 & 0 & 0 & 0 & 0 & 0 & 0 & 0 \\[0.15cm]
0 & 0 & 0 & \frac{2\viscosity}{3} & -\frac{2\viscosity w_4}{3w_5} & 0 & 0 & 0 & 0 \\[0.15cm]
0 & 0 & 0 & 0 & 0 & 0 & 0 & 0 & 0 \\[0.15cm]
0 & -\viscosity & 0 & 0 & \frac{\viscosity w_2}{w_5} & 0 & 0 & 0 & 0 \\[0.15cm]
0 & \frac{\viscosity w_4}{w_5} & 0 & -\frac{2\viscosity w_2}{3 w_5} & -\frac{\viscosity w_2 w_4}{3w_5^2}+\frac{\resistivity w_6 w_8}{w_5^2} & -\frac{\resistivity w_8}{w_5} & 0 & 0 & 0 \\[0.15cm]
0 & 0 & 0 & 0 & 0 & 0 & 0 & 0 & 0 \\[0.15cm]
0 & 0 & 0 & 0 & 0 & 0 & 0 & 0 & 0 \\[0.15cm]
0 & 0 & 0 & 0 & -\frac{\resistivity w_6}{w_5} & \resistivity & 0 & 0 & 0 \\[0.15cm]
0 & 0 & 0 & 0 & 0 & 0 & 0 & 0 & 0 \\[0.15cm]
\end{pmatrix}
\end{equation}

\begin{equation}
\nineMatrix{K}_{21} = \frac{1}{w_5}\begin{pmatrix}
0 & 0 & 0 & 0 & 0 & 0 & 0 & 0 & 0 \\[0.15cm]
0 & 0 & -\viscosity & 0 & \frac{\viscosity w_3}{w_5} & 0 & 0 & 0 & 0 \\[0.15cm]
0 & \frac{2\viscosity}{3} & 0 & 0 & -\frac{2\viscosity w_2}{3 w_5} & 0 & 0 & 0 & 0 \\[0.15cm]
0 & 0 & 0 & 0 & 0 & 0 & 0 & 0 & 0 \\[0.15cm]
0 & -\frac{2\viscosity w_3}{3w_5} & \frac{\viscosity w_2}{w_5} & 0 & -\frac{\viscosity w_2 w_3}{3w_5^2} & 0 & 0 & 0 & 0 \\[0.15cm]
0 & 0 & 0 & 0 & \frac{\resistivity w_6 w_7}{w_5^2} & 0 & -\frac{\resistivity w_6}{w_5} & 0 & 0 \\[0.15cm]
0 & 0 & 0 & 0 & -\frac{\resistivity w_7}{w_5} & 0 & \resistivity & 0 & 0 \\[0.15cm]
0 & 0 & 0 & 0 & 0 & 0 & 0 & 0 & 0 \\[0.15cm]
0 & 0 & 0 & 0 & 0 & 0 & 0 & 0 & 0 \\[0.15cm]
\end{pmatrix}
\end{equation}

\begin{equation}
\resizebox{0.925\hsize}{!}{$
\nineMatrix{K}_{22} = \frac{1}{w_5}\begin{pmatrix}
0 & 0 & 0 & 0 & 0 & 0 & 0 & 0 & 0 \\[0.15cm]
0 & -\viscosity & 0 & 0 & \frac{\viscosity w_2}{w_5} & 0 & 0 & 0 & 0 \\[0.15cm]
0 & 0 & -\frac{4\viscosity}{3} & 0 & \frac{4\viscosity w_3}{3 w_5} & 0 & 0 & 0 & 0 \\[0.15cm]
0 & 0 & 0 & -\viscosity & \frac{\viscosity w_4}{w_5} & 0 & 0 & 0 & 0 \\[0.15cm]
0 & \frac{\viscosity w_2}{w_5} & \frac{4\viscosity w_3}{3 w_5} & \frac{\viscosity w_4}{w_5} & -\frac{\viscosity w_2^2}{w_5^2}-\frac{4\viscosity w_3^2}{3 w_5^2}-\frac{\viscosity w_4^2}{w_5^2}+\frac{\kappa}{R w_5}-\frac{\resistivity w_6^2}{w_5^2}-\frac{\resistivity w_8^2}{w_5^2} & \frac{\resistivity w_6}{w_5} & 0 & \frac{\resistivity w_8}{w_5} & 0 \\[0.15cm]
0 & 0 & 0 & 0 & \frac{\resistivity w_6}{w_5} & -\resistivity & 0 & 0 & 0 \\[0.15cm]
0 & 0 & 0 & 0 & 0 & 0 & 0 & 0 & 0 \\[0.15cm]
0 & 0 & 0 & 0 & \frac{\resistivity w_8}{w_5} & 0 & 0 & -\resistivity & 0 \\[0.15cm]
0 & 0 & 0 & 0 & 0 & 0 & 0 & 0 & 0 \\[0.15cm]
\end{pmatrix}$}
\end{equation}

\begin{equation}
\nineMatrix{K}_{23} = \frac{1}{w_5}\begin{pmatrix}
0 & 0 & 0 & 0 & 0 & 0 & 0 & 0 & 0 \\[0.15cm]
0 & 0 & 0 & 0 & 0 & 0 & 0 & 0 & 0 \\[0.15cm]
0 & 0 & 0 & \frac{2\viscosity}{3} & -\frac{2\viscosity w_4}{3 w_5} & 0 & 0 & 0 & 0 \\[0.15cm]
0 & 0 & -\viscosity & 0 & \frac{\viscosity w_3}{w_5} & 0 & 0 & 0 & 0 \\[0.15cm]
0 & 0 & \frac{\viscosity w_4}{w_5} & -\frac{2\viscosity w_3}{3 w_5} & -\frac{\viscosity w_3 w_4}{3 w_5^2}+\frac{\resistivity w_7 w_8}{w_5^2} & 0 & -\frac{\resistivity w_8}{w_5} & 0 & 0 \\[0.15cm]
0 & 0 & 0 & 0 & 0 & 0 & 0 & 0 & 0 \\[0.15cm]
0 & 0 & 0 & 0 & 0 & 0 & 0 & 0 & 0 \\[0.15cm]
0 & 0 & 0 & 0 & -\frac{\resistivity w_7}{w_5} & 0 & \resistivity & 0 & 0 \\[0.15cm]
0 & 0 & 0 & 0 & 0 & 0 & 0 & 0 & 0 \\[0.15cm]
\end{pmatrix}
\end{equation}

\begin{equation}
\nineMatrix{K}_{31} = \frac{1}{w_5}\begin{pmatrix}
0 & 0 & 0 & 0 & 0 & 0 & 0 & 0 & 0 \\[0.15cm]
0 & 0 & 0 & -\viscosity & \frac{\viscosity w_4}{w_5} & 0 & 0 & 0 & 0 \\[0.15cm]
0 & 0 & 0 & 0 & 0 & 0 & 0 & 0 & 0 \\[0.15cm]
0 & \frac{2\viscosity}{3} & 0 & 0 & -\frac{2\viscosity w_2}{3 w_5} & 0 & 0 & 0 & 0 \\[0.15cm]
0 & -\frac{2\viscosity w_4}{3 w_5} & 0 & \frac{\viscosity w_2}{w_5} & -\frac{\viscosity w_2 w_4}{3w_5^2}+\frac{\resistivity w_6 w_8}{w_5^2} & 0 & 0 & -\frac{\resistivity w_6}{w_5} & 0 \\[0.15cm]
0 & 0 & 0 & 0 & -\frac{\resistivity w_8}{w_5} & 0 & 0 & \resistivity & 0 \\[0.15cm]
0 & 0 & 0 & 0 & 0 & 0 & 0 & 0 & 0 \\[0.15cm]
0 & 0 & 0 & 0 & 0 & 0 & 0 & 0 & 0 \\[0.15cm]
0 & 0 & 0 & 0 & 0 & 0 & 0 & 0 & 0 \\[0.15cm]
\end{pmatrix}
\end{equation}

\begin{equation}
\nineMatrix{K}_{32} = \frac{1}{w_5}\begin{pmatrix}
0 & 0 & 0 & 0 & 0 & 0 & 0 & 0 & 0 \\[0.15cm]
0 & 0 & 0 & 0 & 0 & 0 & 0 & 0 & 0 \\[0.15cm]
0 & 0 & 0 & -\viscosity & \frac{\viscosity w_4}{w_5} & 0 & 0 & 0 & 0 \\[0.15cm]
0 & 0 & \frac{2\viscosity}{3} & 0 & -\frac{2\viscosity w_3}{3 w_5} & 0 & 0 & 0 & 0 \\[0.15cm]
0 & 0 & -\frac{2\viscosity w_4}{3 w_5} & \frac{\viscosity w_3}{w_5} & -\frac{\viscosity w_3 w_4}{3 w_5^5}+\frac{\resistivity w_7 w_8}{w_5^2} & 0 & 0 & -\frac{\resistivity w_7}{w_5} & 0 \\[0.15cm]
0 & 0 & 0 & 0 & 0 & 0 & 0 & 0 & 0 \\[0.15cm]
0 & 0 & 0 & 0 & -\frac{\resistivity w_8}{w_5} & 0 & 0 & \resistivity & 0 \\[0.15cm]
0 & 0 & 0 & 0 & 0 & 0 & 0 & 0 & 0 \\[0.15cm]
0 & 0 & 0 & 0 & 0 & 0 & 0 & 0 & 0 \\[0.15cm]
\end{pmatrix}
\end{equation}

\begin{equation}\label{eq:matK33}
\resizebox{0.925\hsize}{!}{$
\nineMatrix{K}_{33} = \frac{1}{w_5} \begin{pmatrix}
0 & 0 & 0 & 0 & 0 & 0 & 0 & 0 & 0 \\[0.15cm]
0 & -\viscosity & 0 & 0 & \frac{\viscosity w_2}{w_5} & 0 & 0 & 0 & 0 \\[0.15cm]
0 & 0 & -\viscosity & 0 & \frac{\viscosity w_3}{w_5} & 0 & 0 & 0 & 0 \\[0.15cm]
0 & 0 & 0 & -\frac{4\viscosity}{3} & \frac{4\viscosity w_4}{3 w_5} & 0 & 0 & 0 & 0 \\[0.15cm]
0 & \frac{\viscosity w_2}{w_5} & \frac{\viscosity w_3}{w_5} & \frac{4\viscosity w_4}{3 w_5} & -\frac{\viscosity w_2^2}{w_5^2}-\frac{\viscosity w_3^2}{w_5^2}-\frac{4\viscosity w_4^2}{3w_5^2}+\frac{\kappa}{Rw_5}-\frac{\resistivity w_6^2}{w_5^2}-\frac{\resistivity w_6^2}{w_5^2} & \frac{\resistivity w_6}{w_5} & \frac{\resistivity w_7}{w_5} & 0 & 0 \\[0.15cm]
0 & 0 & 0 & 0 & \frac{\resistivity w_6}{w_5} & -\resistivity & 0 & 0 & 0 \\[0.15cm]
0 & 0 & 0 & 0 & \frac{\resistivity w_7}{w_5} & 0 & -\resistivity & 0 & 0 \\[0.15cm]
0 & 0 & 0 & 0 & 0 & 0 & 0 & 0 & 0 \\[0.15cm]
0 & 0 & 0 & 0 & 0 & 0 & 0 & 0 & 0 \\[0.15cm]
\end{pmatrix}$}
\end{equation}

\section{Proof of Lemma \ref{lem:MHD_EC_vol} (Entropy contribution of the curvilinear ideal MHD volume terms)}\label{Sec:MHD_vol}
In this section we show that the volume contributions of the ideal MHD equations and the non-conservative terms cancel in entropy space. In order to do so, we first expand each of the volume contribution from the advective terms
\begin{equation}\label{eq:F_EC_vol}
\begin{aligned}
\iprodN{\DD\bigcontravec F^{\ec,\supMHD},\statevec{W}} = 
 \sum\limits_{i,j,k=0}^{N}\omega_{ijk}\statevec{W}^T_{ijk}&\left[2\sum_{m=0}^N \dmat_{im}\!\left(\bigstatevec{F}^{\ec,\supMHD}(\statevec{U}_{ijk}, \statevec{U}_{mjk})\cdot\avg{J\spacevec{a}^{\,1}}_{(i,m)jk}\right)\right.\\[0.1cm]
&+     2\sum_{m=0}^N \dmat_{jm}\!\left(\bigstatevec{F}^{\ec,\supMHD}(\statevec{U}_{ijk}, \statevec{U}_{imk})\cdot\avg{J\spacevec{a}^{\,2}}_{i(j,m)k}\right)\\[0.1cm]
&+     \left.2\sum_{m=0}^N \dmat_{km}\!\left(\bigstatevec{F}^{\ec,\supMHD}(\statevec{U}_{ijk}, \statevec{U}_{ijm})\cdot\avg{J\spacevec{a}^{\,3}}_{ij(k,m)}\right)\right]
\end{aligned}
\end{equation}
and the non-conservative term
\begin{equation}\label{eq:nonConsProof}
\begin{aligned}
\iprodN{\Jan\DDncdiv\contraspacevec{B},\statevec{W}}=
\sum\limits_{i,j,k=0}^{N}\omega_{ijk}\statevec{W}^T_{ijk}&\left[ \sum_{m=0}^N \dmat_{im}\!\Jan_{ijk}\!\left(\spacevec{B}_{mjk}\cdot\avg{J\spacevec{a}^{\,1}}_{(i,m)jk}\right)\right.\\[0.1cm]
&+     \sum_{m=0}^N \dmat_{jm}\!\Jan_{ijk}\!\left(\spacevec{B}_{imk}\cdot\avg{J\spacevec{a}^{\,2}}_{i(j,m)k}\right)\\[0.1cm]
&+     \left.\sum_{m=0}^N \dmat_{km}\!\Jan_{ijk}\!\left(\spacevec{B}_{ijm}\cdot\avg{J\spacevec{a}^{\,3}}_{ij(k,m)}\right)\right].
\end{aligned}
\end{equation}
We then focus on the $\xi-$direction term of the volume integral approximations, which greatly simplifies the analysis (as the other directions are done in an analogous manner). The sum of \eqref{eq:F_EC_vol} can be written in terms of the SBP matrix \eqref{eq:SBPProperty}, $\qmat_{im}=\omega_{i}\dmat_{im}$,
\begin{equation}
\begin{aligned}
\sum\limits_{i,j,k=0}^{N}\omega_{ijk}\statevec{W}^T_{ijk}&2\sum_{m=0}^N \dmat_{im}\!\left(\bigstatevec{F}^{\ec,\supMHD}(\statevec{U}_{ijk}, \statevec{U}_{mjk})\cdot\avg{J\spacevec{a}^{\,1}}_{(i,m)jk}\right)\\[0.05cm]
=&\sum\limits_{j,k=0}^{N}\omega_{jk}\sum\limits_{i=0}^N\statevec{W}^T_{ijk}\sum_{m=0}^N 2\omega_i\dmat_{im}\!\left(\bigstatevec{F}^{\ec,\supMHD}(\statevec{U}_{ijk}, \statevec{U}_{mjk})\cdot\avg{J\spacevec{a}^{\,1}}_{(i,m)jk}\right)\\[0.05cm]
=&\sum\limits_{j,k=0}^{N}\omega_{jk}\sum\limits_{i=0}^N\statevec{W}^T_{ijk}\sum_{m=0}^N 2\qmat_{im}\!\left(\bigstatevec{F}^{\ec,\supMHD}(\statevec{U}_{ijk}, \statevec{U}_{mjk})\cdot\avg{J\spacevec{a}^{\,1}}_{(i,m)jk}\right).
\end{aligned}
\end{equation}
We use the summation-by-parts property $2\qmat_{im} = \qmat_{im} - \qmat_{mi} + \bmat_{im}$, perform a reindexing of $i$ and $m$ to subsume the $\qmat_{mi}$ term and use the facts that $\bigstatevec{F}^{\ec,\supMHD}(\statevec{U}_{ijk}, \statevec{U}_{mjk})$ and the average operator of the metric term $\avg{Ja^1}_{(i,m)jk}$ are symmetric with respect to the index $i$ and $m$ to rewrite the $\xi-$direction contribution to the volume integral approximation as
\begin{equation}\label{eq:summ1}
\begin{split}
\sum\limits_{i=0}^N\statevec W^T_{ijk}\sum_{m=0}^N& 2\qmat_{im}\!\left(\bigstatevec F^{\ec,\supMHD}(\statevec{U}_{ijk}, \statevec{U}_{mjk})\cdot\avg{J\spacevec{a}^{\,1}}_{(i,m)jk}\right) \\
= \sum_{i,m=0}^N &\statevec W^T_{ijk}\,(\qmat_{im} - \qmat_{mi} + \bmat_{im})\left( 
\bigstatevec F^{\ec,\supMHD}(\statevec{U}_{ijk},\statevec{U}_{mjk})\cdot\avg{J\spacevec{a}^{\,1}}_{(i,m)jk}\right)\\
= \sum_{i,m=0}^N &\qmat_{im}\!\left(\statevec W_{ijk} - \statevec W_{mjk}\right)^T\!\left(
\bigstatevec F^{\ec,\supMHD}(\statevec{U}_{ijk},\statevec{U}_{mjk})\cdot\avg{J\spacevec{a}^{\,1}}_{(i,m)jk}\right)\, \\
&+ \bmat_{im}\!\statevec W^T_{ijk}\!\left(\bigstatevec F^{\ec,\supMHD}(\statevec{U}_{ijk},\statevec{U}_{mjk})\cdot\avg{J\spacevec{a}^{\,1}}_{(i,m)jk}\right).
\end{split}
\end{equation}

Because the proof at hand only concerns ideal MHD terms, we are only concerned with the entropy conservation condition \eqref{eq:MHDPartDiscrete}, which we use to replace the first terms in \eqref{eq:summ1} with
\begin{equation}\label{eq:moreBs}
\left(\statevec W_{ijk} - \statevec W_{mjk}\right)^T\statevec F_\ell^{\ec,\supMHD}(\statevec{U}_{ijk},\statevec{U}_{mjk}) = \left(\Psi^{\supMHD}_\ell\right)_{ijk}-\left(\Psi^{\supMHD}_\ell\right)_{mjk}
-\frac{1}{2}\left(\left(B_\ell\right)_{ijk}+\left(B_\ell\right)_{mjk}\right)\left(\wDotJan_{ijk}-\wDotJan_{mjk}\right)
\end{equation}
for $\ell = 1,2,3$.

Furthermore, note that the entries of the boundary matrix $\bmat$ are only non-zero when $i=m=0$ or $i=m=N$, so
\begin{equation}\label{eq:boundaryBs}
\bmat_{im}\statevec W^T_{ijk}\statevec F_\ell^{\ec,\supMHD}(\statevec{U}_{ijk},\statevec{U}_{mjk}) = \bmat_{im}\!\left(\left(\Psi_l^{\supMHD}\right)_{ijk} - \wDotJan_{ijk}\left(B_\ell\right)_{ijk}\right)
\end{equation}
for $\ell = 1,2,3$.

We substitute \eqref{eq:moreBs} and \eqref{eq:boundaryBs} into the final line of \eqref{eq:summ1} to find
\begin{equation}\label{eq:awfulSum}
\begin{aligned}
\sum\limits_{i=0}^N&\statevec W^T_{ijk}\sum_{m=0}^N 2\qmat_{im}\!\left(\bigstatevec F^{\ec,\supMHD}(\statevec{U}_{ijk}, \statevec{U}_{mjk})\cdot\avg{J\spacevec{a}^{\,1}}_{(i,m)jk}\right)\\
&=\sum\limits_{i,m=0}^N\qmat_{im}\!\left(\left[\spacevec{\Psi}^{\supMHD}_{ijk}-\spacevec{\Psi}^{\supMHD}_{mjk}-\frac{1}{2}\left(\spacevec{B}_{ijk}+\spacevec{B}_{mjk}\right)\left(\wDotJan_{ijk}-\wDotJan_{mjk}\right)\right]\cdot\avg{J\spacevec{a}^{\,1}}_{(i,m)jk}\right)\\
&\qquad\qquad\qquad \qquad+  \bmat_{im}\!\left(\left[\spacevec{\Psi}^{\supMHD}_{ijk} - \wDotJan_{ijk}\spacevec{B}_{ijk}\right]\cdot\avg{J\spacevec{a}^{\,1}}_{(i,m)jk}\right).
\end{aligned}
\end{equation}
We next examine the terms of the sum \eqref{eq:awfulSum} systematically from left to right. Now, because the derivative of a constant is zero (i.e. the rows of $\qmat$ sum to zero),
\begin{equation}\label{eq:B11}
\begin{aligned}
\sum\limits_{i,m=0}^N\qmat_{im}\!\left(\spacevec{\Psi}^{\supMHD}_{ijk}\cdot\avg{J\spacevec{a}^{\,1}}_{(i,m)jk}\right) &= \frac{1}{2}\sum\limits_{i=0}^N\left(\spacevec{\Psi}^{\supMHD}_{ijk}\cdot\left(J\spacevec{a}^{\,1}\right)_{ijk}\right)\sum\limits_{m=0}^N\qmat_{im} + \frac{1}{2}\sum\limits_{i,m=0}^N\qmat_{im}\!\left(\spacevec{\Psi}^{\supMHD}_{ijk}\cdot\left(J\spacevec{a}^{\,1}\right)_{mjk}\right)\\[0.05cm]
&=\frac{1}{2}\sum\limits_{i,m=0}^N\qmat_{im}\!\left(\spacevec{\Psi}^{\supMHD}_{ijk}\cdot\left(J\spacevec{a}^{\,1}\right)_{mjk}\right).
\end{aligned}
\end{equation}
Next, on the second term, we use the summation by parts property \eqref{eq:SBPProperty}, reindex on the $\qmat_{mi}$ term, and the symmetric property of the arithmetic mean to rewrite
\begin{equation}\label{eq:B12}
\begin{aligned}
-\!\!\!\sum\limits_{i,m=0}^N\qmat_{im}\!\left(\spacevec{\Psi}^{\supMHD}_{mjk}\cdot\avg{J\spacevec{a}^{\,1}}_{(i,m)jk}\right) &= -\!\!\!\sum\limits_{i,m=0}^N(\bmat_{im}-\qmat_{mi})\left(\spacevec{\Psi}^{\supMHD}_{mjk}\cdot\avg{J\spacevec{a}^{\,1}}_{(i,m)jk}\right)\\
&=-\!\!\!\sum\limits_{i,m=0}^N(\bmat_{im}-\qmat_{im})\left(\spacevec{\Psi}^{\supMHD}_{ijk}\cdot\avg{J\spacevec{a}^{\,1}}_{(i,m)jk}\right)\\
&=-\!\!\!\sum\limits_{i,m=0}^N\bmat_{im}\!\left(\spacevec{\Psi}^{\supMHD}_{ijk}\cdot\avg{J\spacevec{a}^{\,1}}_{(i,m)jk}\right) + \sum\limits_{i,m=0}^N\qmat_{im}\!\left(\spacevec{\Psi}^{\supMHD}_{ijk}\cdot\avg{J\spacevec{a}^{\,1}}_{(i,m)jk}\right)\\
&=-\!\!\!\sum\limits_{i,m=0}^N\bmat_{im}\!\left(\spacevec{\Psi}^{\supMHD}_{ijk}\cdot\avg{J\spacevec{a}^{\,1}}_{(i,m)jk}\right) + \frac{1}{2}\sum\limits_{i,m=0}^N\qmat_{im}\!\left(\spacevec{\Psi}^{\supMHD}_{ijk}\cdot\left(J\spacevec{a}^{\,1}\right)_{mjk}\right),
\end{aligned}
\end{equation}
where, again, one term in the second arithmetic mean drops out due to consistency of the matrix $\qmat$. 

We come next to the terms involving $\spacevec{B}$ and $\wDotJan$ in \eqref{eq:awfulSum}. We leave these terms grouped for convenience and first expand to find
\begin{equation}\label{eq:expandGross}
\begin{aligned}
\sum\limits_{i,m=0}^N\qmat_{im}\!&\left(\left[-\frac{1}{2}\left(\spacevec{B}_{ijk}+\spacevec{B}_{mjk}\right)\left(\wDotJan_{ijk}-\wDotJan_{mjk}\right)\right]\cdot\avg{J\spacevec{a}^{\,1}}_{(i,m)jk}\right)\\
&=-\frac{1}{2}\sum\limits_{i,m=0}^N \qmat_{im}\!\left(\left[\wDotJan_{ijk}\spacevec{B}_{ijk} + \wDotJan_{ijk}\spacevec{B}_{mjk} - \wDotJan_{mjk}\spacevec{B}_{ijk} - \wDotJan_{mjk}\spacevec{B}_{mjk}\right]\cdot\avg{J\spacevec{a}^{\,1}}_{(i,m)jk}\right).
\end{aligned}
\end{equation}
We examine each term from \eqref{eq:expandGross}: for the first term we use the consistency of $\qmat$, the second term is left alone, the third term makes a reindexing of $i$ and $m$, and the fourth term applies the SBP property to obtain
\begin{equation}\label{eq:expandGross2}
\begin{aligned}
-\frac{1}{2}\sum\limits_{i,m=0}^N& \qmat_{im}\!\left(\left[\wDotJan_{ijk}\spacevec{B}_{ijk} + \wDotJan_{ijk}\spacevec{B}_{mjk} - \wDotJan_{mjk}\spacevec{B}_{ijk} - \wDotJan_{mjk}\spacevec{B}_{mjk}\right]\cdot\avg{J\spacevec{a}^{\,1}}_{(i,m)jk}\right)\\
=-\frac{1}{4}&\sum\limits_{i,m=0}^N\qmat_{im}\wDotJan_{ijk}\!\left(\spacevec{B}_{ijk}\cdot\left(J\spacevec{a}^{\,1}\right)_{mjk}\right)-\frac{1}{2}\sum\limits_{i,m=0}^N(\qmat_{im}-\qmat_{mi})\wDotJan_{ijk}\!\left(\spacevec{B}_{mjk}\cdot\avg{J\spacevec{a}^{\,1}}_{(i,m)jk}\right)\\[0.05cm]
&\qquad\qquad\qquad\qquad\qquad\qquad\quad + \frac{1}{2}\sum\limits_{i,m=0}^N(\bmat_{im}-\qmat_{mi})\wDotJan_{mjk}\!\left(\spacevec{B}_{mjk}\cdot\avg{J\spacevec{a}^{\,1}}_{(i,m)jk}\right).
\end{aligned}
\end{equation}
Next, we use the SBP property on the $\qmat_{mi}$ term in the second sum of \eqref{eq:expandGross2} and reindex $i$ and $m$ in the third term to get
\begin{equation}\label{eq:expandGross3}
\begin{aligned}
-\frac{1}{2}\sum\limits_{i,m=0}^N& \qmat_{im}\left(\left[\wDotJan_{ijk}\spacevec{B}_{ijk} + \wDotJan_{ijk}\spacevec{B}_{mjk} - \wDotJan_{mjk}\spacevec{B}_{ijk} - \wDotJan_{mjk}\spacevec{B}_{mjk}\right]\cdot\avg{J\spacevec{a}^{\,1}}_{(i,m)jk}\right)\\
=-\frac{1}{4}&\sum\limits_{i,m=0}^N\qmat_{im}\wDotJan_{ijk}\!\left(\spacevec{B}_{ijk}\cdot\left(J\spacevec{a}^{\,1}\right)_{mjk}\right)-\frac{1}{2}\sum\limits_{i,m=0}^N(2\qmat_{im}-\bmat_{im})\wDotJan_{ijk}\!\left(\spacevec{B}_{mjk}\cdot\avg{J\spacevec{a}^{\,1}}_{(i,m)jk}\right)\\[0.05cm]
&\qquad\qquad\qquad\qquad\qquad\qquad\quad + \frac{1}{2}\sum\limits_{i,m=0}^N(\bmat_{im}-\qmat_{im})\wDotJan_{ijk}\!\left(\spacevec{B}_{ijk}\cdot\avg{J\spacevec{a}^{\,1}}_{(i,m)jk}\right).
\end{aligned}
\end{equation}
It is clear now that the terms with the boundary matrix, $\bmat$, combine from the second and third terms of \eqref{eq:expandGross3}. Also, the $i$ term of the $\qmat_{im}$ piece of the third term cancels due to consistency (similar to \eqref{eq:B11}). The remaining part of the third term then combines with the first term arriving at
\begin{equation}\label{eq:expandGross4}
\begin{aligned}
-\frac{1}{2}\sum\limits_{i,m=0}^N& \qmat_{im}\!\left(\left[\wDotJan_{ijk}\spacevec{B}_{ijk} + \wDotJan_{ijk}\spacevec{B}_{mjk} - \wDotJan_{mjk}\spacevec{B}_{ijk} - \wDotJan_{mjk}\spacevec{B}_{mjk}\right]\cdot\avg{J\spacevec{a}^{\,1}}_{(i,m)jk}\right)\\
=&-\frac{1}{2}\sum\limits_{i,m=0}^N\qmat_{im}\wDotJan_{ijk}\!\left(\spacevec{B}_{ijk}\cdot\left(J\spacevec{a}^{\,1}\right)_{mjk}\right)-\!\!\!\sum\limits_{i,m=0}^N\qmat_{im}\wDotJan_{ijk}\!\left(\spacevec{B}_{mjk}\cdot\avg{J\spacevec{a}^{\,1}}_{(i,m)jk}\right)\\[0.05cm]
&\qquad\qquad\qquad\qquad\qquad\qquad\quad +\sum\limits_{i,m=0}^N\bmat_{im}\wDotJan_{ijk}\!\left(\spacevec{B}_{ijk}\cdot\avg{J\spacevec{a}^{\,1}}_{(i,m)jk}\right).
\end{aligned}
\end{equation}

Combining the results of \eqref{eq:B11}, \eqref{eq:B12}, and \eqref{eq:expandGross4}, we rewrite \eqref{eq:awfulSum} to have 
\begin{equation}\label{eq:awfulSumRedux}
\begin{aligned}
\sum\limits_{i=0}^N&\statevec W^T_{ijk}\sum_{m=0}^N 2\qmat_{im}\!\left(\bigstatevec F^{\ec,\supMHD}(\statevec{U}_{ijk}, \statevec{U}_{mjk})\cdot\avg{J\spacevec{a}^{\,1}}_{(i,m)jk}\right)\\
&=-\!\!\!\sum\limits_{i,m=0}^N\bmat_{im}\!\left(\spacevec{\Psi}^{\supMHD}_{ijk}\cdot\avg{J\spacevec{a}^{\,1}}_{(i,m)jk}\right)
+ \sum\limits_{i,m=0}^N\qmat_{im}\!\left(\spacevec{\Psi}^{\supMHD}_{ijk}\cdot\left(J\spacevec{a}^{\,1}\right)_{mjk}\right)
-\frac{1}{2}\sum\limits_{i,m=0}^N\qmat_{im}\wDotJan_{ijk}\!\left(\spacevec{B}_{ijk}\cdot\left(J\spacevec{a}^{\,1}\right)_{mjk}\right)\\[0.05cm]
&\quad-\!\!\!\sum\limits_{i,m=0}^N\qmat_{im}\wDotJan_{ijk}\!\left(\spacevec{B}_{mjk}\cdot\avg{J\spacevec{a}^{\,1}}_{(i,m)jk}\right)
+ \sum\limits_{i,m=0}^N\bmat_{im}\wDotJan_{ijk}\!\left(\spacevec{B}_{ijk}\cdot\avg{J\spacevec{a}^{\,1}}_{(i,m)jk}\right)\\[0.05cm]
&\quad+ \sum\limits_{i,m=0}^N\bmat_{im}\!\left(\left[\spacevec{\Psi}^{\supMHD}_{ijk} - \wDotJan_{ijk}\spacevec{B}_{ijk}\right]\cdot\avg{J\spacevec{a}^{\,1}}_{(i,m)jk}\right).
\end{aligned}
\end{equation}
Conveniently, several terms in \eqref{eq:awfulSumRedux} cancel to leave
\begin{equation}\label{eq:awfulSumRedux2}
\begin{aligned}
\sum\limits_{i=0}^N&\statevec W^T_{ijk}\sum_{m=0}^N 2\qmat_{im}\!\left(\bigstatevec F^{\ec,\supMHD}(\statevec{U}_{ijk}, \statevec{U}_{mjk})\cdot\avg{J\spacevec{a}^{\,1}}_{(i,m)jk}\right)\\
&=\sum\limits_{i,m=0}^N\qmat_{im}\!\left(\spacevec{\Psi}^{\supMHD}_{ijk}\cdot\left(J\spacevec{a}^{\,1}\right)_{mjk}\right)
-\frac{1}{2}\sum\limits_{i,m=0}^N\qmat_{im}\wDotJan_{ijk}\!\left(\spacevec{B}_{ijk}\cdot\left(J\spacevec{a}^{\,1}\right)_{mjk}\right)
-\!\!\!\sum\limits_{i,m=0}^N\qmat_{im}\wDotJan_{ijk}\!\left(\spacevec{B}_{mjk}\cdot\avg{J\spacevec{a}^{\,1}}_{(i,m)jk}\right).
\end{aligned}
\end{equation}

We are now prepared to revisit the contributions from the non-conservative volume terms \eqref{eq:nonConsProof}, which read in the $\xi-$direction as
\begin{equation}\label{eq:nonConsProof1}
\begin{aligned}
\sum\limits_{i,j,k=0}^{N}\omega_{ijk}\statevec{W}^T_{ijk} \sum_{m=0}^N \dmat_{im}\Jan_{ijk}\!\left(\spacevec{B}_{mjk}\cdot\avg{J\spacevec{a}^{\,1}}_{(i,m)jk}\right)
=&\sum\limits_{j,k=0}^N\omega_{jk}\sum_{i,m=0}^N\omega_i\dmat_{im}\statevec{W}^T_{ijk}\Jan_{ijk}\!\left(\spacevec{B}_{mjk}\cdot\avg{J\spacevec{a}^{\,1}}_{(i,m)jk}\right)\\
=&\sum\limits_{j,k=0}^N\omega_{jk}\sum_{i,m=0}^N\qmat_{im}\statevec{W}^T_{ijk}\Jan_{ijk}\!\left(\spacevec{B}_{mjk}\cdot\avg{J\spacevec{a}^{\,1}}_{(i,m)jk}\right)\\
=&\sum\limits_{j,k=0}^N\omega_{jk}\sum_{i,m=0}^N\qmat_{im}\wDotJan_{ijk}\!\left(\spacevec{B}_{mjk}\cdot\avg{J\spacevec{a}^{\,1}}_{(i,m)jk}\right),
\end{aligned}
\end{equation}
where we use the definition of the SBP matrix and the property \eqref{eq:contractSource} contracting the non-conservative term into entropy space.
Comparing the result \eqref{eq:nonConsProof1} and the last term of \eqref{eq:awfulSumRedux2} we see that they cancel when added together. Thus, the contribution in the $\xi-$direction is
\begin{equation}\label{eq:awfulSumRedux3}
\begin{aligned}
\sum\limits_{j,k=0}^N\omega_{jk}\sum\limits_{i=0}^N&\statevec W^T_{ijk}\left[\sum_{m=0}^N 2\qmat_{im}\!\left(\bigstatevec F^{\ec,\supMHD}(\statevec{U}_{ijk}, \statevec{U}_{mjk})\cdot\avg{J\spacevec{a}^{\,1}}_{(i,m)jk}\right)
+ \sum_{m=0}^N\qmat_{im}\wDotJan_{ijk}\!\left(\spacevec{B}_{mjk}\cdot\avg{J\spacevec{a}^{\,1}}_{(i,m)jk}\right)\right]\\
&\qquad\qquad=\sum\limits_{j,k=0}^N\omega_{jk}\left[\sum\limits_{i,m=0}^N\qmat_{im}\!\left(\spacevec{\Psi}^{\supMHD}_{ijk}\cdot\left(J\spacevec{a}^{\,1}\right)_{mjk}\right)
-\frac{1}{2}\sum\limits_{i,m=0}^N\qmat_{im}\wDotJan_{ijk}\!\left(\spacevec{B}_{ijk}\cdot\left(J\spacevec{a}^{\,1}\right)_{mjk}\right)\right].
\end{aligned}
\end{equation}

Summarized, the total contribution of the $\xi-$direction of \eqref{eq:F_EC_vol} and \eqref{eq:nonConsProof} in the volume term is
\begin{equation}
\begin{aligned}
\sum\limits_{i,j,k=0}^N\omega_{ijk}&\statevec{W}^T_{ijk}2\sum\limits_{m=0}^N\dmat_{im}\!\left(\left[\bigstatevec{F}^{\ec,\supMHD}(\statevec{U}_{ijk},\statevec{U}_{mjk})+\Jan_{ijk}\spacevec{B}_{mjk}\right]\cdot\avg{J\spacevec{a}^{\,1}}_{(i,m)jk}\right)\\
&=\sum\limits_{i,j,k=0}^N\omega_{ijk}\sum\limits_{m=0}^N\dmat_{im}\!\left(\left[\spacevec{\Psi}^{\supMHD}_{ijk}
-\frac{1}{2}\wDotJan_{ijk}\spacevec{B}_{ijk}\right]\cdot\left(J\spacevec{a}^{\, 1}\right)_{mjk}\right),
\end{aligned}
\end{equation}
where we returned the polynomial derivative matrix due to the property $\qmat_{im}=\omega_i\dmat_{im}$.

Similar results hold for the $\eta-$ and $\zeta-$directions of the volume integral approximations leading to
\begin{equation}
\begin{aligned}
\iprodN{\DD\bigcontravec F^{\ec,\supMHD},\statevec W} + \iprodN{\Jan\DDncdiv\contraspacevec{B},\statevec{W}}& =\sum\limits_{i,j,k=0}^N\omega_{ijk}\sum\limits_{m=0}^N\dmat_{im}\!\left(\left[\spacevec{\Psi}^{\supMHD}_{ijk}
-\frac{1}{2}\wDotJan_{ijk}\spacevec{B}_{ijk}\right]\cdot\left(J\spacevec{a}^{\, 1}\right)_{mjk}\right)\\[0.05cm]
& + \sum\limits_{i,j,k=0}^N\omega_{ijk}\sum\limits_{m=0}^N\dmat_{jm}\!\left(\left[\spacevec{\Psi}^{\supMHD}_{ijk}
-\frac{1}{2}\wDotJan_{ijk}\spacevec{B}_{ijk}\right]\cdot\left(J\spacevec{a}^{\, 2}\right)_{imk}\right)\\[0.05cm]
& + \sum\limits_{i,j,k=0}^N\omega_{ijk}\sum\limits_{m=0}^N\dmat_{km}\!\left(\left[\spacevec{\Psi}^{\supMHD}_{ijk}
-\frac{1}{2}\wDotJan_{ijk}\spacevec{B}_{ijk}\right]\cdot\left(J\spacevec{a}^{\, 3}\right)_{ijm}\right).
\end{aligned}
\end{equation}
Regrouping these volume terms we have shown that
\begin{equation}
\begin{aligned}
&\iprodN{\DD\bigcontravec F^{\ec,\supMHD},\statevec W} + \iprodN{\Jan\DDncdiv\contraspacevec{B},\statevec{W}}=\\
&~~~ \sum\limits_{i,j,k=0}^N\omega_{ijk}\left(\left(\Psi^{\supMHD}_1\right)_{ijk} - \frac{1}{2}\wDotJan_{ijk}(B_1)_{ijk}\right)\left\{\sum\limits_{m=0}^N\dmat_{im}\!\left(Ja_1^1\right)_{mjk} + \sum\limits_{m=0}^N\dmat_{jm}\!\left(Ja_1^2\right)_{imk} + \sum\limits_{m=0}^N\dmat_{km}\!\left(Ja_1^3\right)_{ijm}\right\}\\[0.05cm]
& + \sum\limits_{i,j,k=0}^N\omega_{ijk}\left(\left(\Psi^{\supMHD}_2\right)_{ijk} - \frac{1}{2}\wDotJan_{ijk}(B_2)_{ijk}\right)\left\{\sum\limits_{m=0}^N\dmat_{im}\!\left(Ja_2^1\right)_{mjk} + \sum\limits_{m=0}^N\dmat_{jm}\!\left(Ja_2^2\right)_{imk} + \sum\limits_{m=0}^N\dmat_{km}\!\left(Ja_2^3\right)_{ijm}\right\}\\[0.05cm]
& + \sum\limits_{i,j,k=0}^N\omega_{ijk}\left(\left(\Psi^{\supMHD}_3\right)_{ijk} - \frac{1}{2}\wDotJan_{ijk}(B_3)_{ijk}\right)\left\{\sum\limits_{m=0}^N\dmat_{im}\!\left(Ja_3^1\right)_{mjk} + \sum\limits_{m=0}^N\dmat_{jm}\!\left(Ja_3^2\right)_{imk} + \sum\limits_{m=0}^N\dmat_{km}\!\left(Ja_3^3\right)_{ijm}\right\},
\end{aligned}
\end{equation}
which gives the desired result \eqref{eq:MHDVolTerms}, provided the metric identities are satisfied discretely, i.e., that
\begin{equation}\label{eq:discMetIds_MHD}
\sum_{m=0}^N \dmat_{im}\!\left(J  a^1_n\right)_{mjk}+\dmat_{jm}\!\left(Ja^2_n\right)_{imk}+\dmat_{km}\!\left(Ja^3_n\right)_{ijm}= \sum_{l=1}^{3}\frac{\partial}{\partial\xi^{l}}\interpolant{Ja^l_n}  = 0                                                                                                   
\end{equation}
for $n=1,2,3$ at all LGL nodes $i,j,k=0,\ldots,N$ within an element.

\section{Proof of Lemma \ref{lem:GLM_EC_vol} (Entropy contribution of the curvilinear GLM volume terms)}\label{Sec:GLM_vol}
In this section, we demonstrate that the GLM volume contributions of \eqref{eq:schemeFinal} reduce to zero in entropy space, i.e,
\begin{equation}\label{eq:volInProof_GLM}
\iprodN{\DD\bigcontravec F^{\ec,\supGLM},\statevec W} + \change{\iprodN{\GalDT\cdot\DDncgrad\psi,\statevec{W}}}= 0.
\end{equation}
We begin with the term arising from the $\bigcontravec F^{\ec,\supGLM}$. Similar to the proof in \ref{Sec:MHD_vol}, we first expand each of the volume contributions
\begin{equation}
\begin{aligned}
\iprodN{\DD\bigcontravec F^{\ec,\supGLM},\statevec{W}} = 
 \sum\limits_{i,j,k=0}^{N}\omega_{ijk}\statevec{W}^T_{ijk}&\left[2\sum_{m=0}^N \dmat_{im}\!\left(\bigstatevec{F}^{\ec,\supGLM}(\statevec{U}_{ijk}, \statevec{U}_{mjk})\cdot\avg{J\spacevec{a}^{\,1}}_{(i,m)jk}\right)\right.\\[0.1cm]
&+     2\sum_{m=0}^N \dmat_{jm}\!\left(\bigstatevec{F}^{\ec,\supGLM}(\statevec{U}_{ijk}, \statevec{U}_{imk})\cdot\avg{J\spacevec{a}^{\,2}}_{i(j,m)k}\right)\\[0.1cm]
&+     \left.2\sum_{m=0}^N \dmat_{km}\!\left(\bigstatevec{F}^{\ec,\supGLM}(\statevec{U}_{ijk}, \statevec{U}_{ijm})\cdot\avg{J\spacevec{a}^{\,3}}_{ij(k,m)}\right)\right].
\end{aligned}
\end{equation}
Again, focus is given to the $\xi-$direction term, as the other spatial directions follow from an analogous argument. The sum can be written in terms of the SBP matrix \eqref{eq:SBPProperty}, $\qmat_{im}=\omega_{i}\dmat_{im}$,
\begin{equation}
\begin{aligned}
\sum\limits_{i,j,k=0}^{N}\omega_{ijk}\statevec{W}^T_{ijk}&2\sum_{m=0}^N \dmat_{im}\!\left(\bigstatevec{F}^{\ec,\supGLM}(\statevec{U}_{ijk}, \statevec{U}_{mjk})\cdot\avg{J\spacevec{a}^{\,1}}_{(i,m)jk}\right)\\[0.05cm]
=&\sum\limits_{j,k=0}^{N}\omega_{jk}\sum\limits_{i=0}^N\statevec{W}^T_{ijk}\sum_{m=0}^N 2\omega_i\dmat_{im}\!\left(\bigstatevec{F}^{\ec,\supGLM}(\statevec{U}_{ijk}, \statevec{U}_{mjk})\cdot\avg{J\spacevec{a}^{\,1}}_{(i,m)jk}\right)\\[0.05cm]
=&\sum\limits_{j,k=0}^{N}\omega_{jk}\sum\limits_{i=0}^N\statevec{W}^T_{ijk}\sum_{m=0}^N 2\qmat_{im}\!\left(\bigstatevec{F}^{\ec,\supGLM}(\statevec{U}_{ijk}, \statevec{U}_{mjk})\cdot\avg{J\spacevec{a}^{\,1}}_{(i,m)jk}\right).
\end{aligned}
\end{equation}
We apply the summation-by-parts property $2\qmat_{im} = \qmat_{im} - \qmat_{mi} + \bmat_{im}$, perform a reindexing of $i$ and $m$ to subsume the $\qmat_{mi}$ term and use the symmetry with respect to the index $i$ and $m$ of $\bigstatevec{F}^{\ec,\supGLM}(\statevec{U}_{ijk}, \statevec{U}_{mjk})$ and the average operator of the metric term $\avg{Ja^1}_{(i,m)jk}$ to rewrite the $\xi-$direction contribution as
\begin{equation}\label{eq:summ1_GLM}
\begin{split}
\sum\limits_{i=0}^N\statevec W^T_{ijk}\sum_{m=0}^N& 2\qmat_{im}\!\left(\bigstatevec F^{\ec,\supGLM}(\statevec{U}_{ijk}, \statevec{U}_{mjk})\cdot\avg{J\spacevec{a}^{\,1}}_{(i,m)jk}\right) \\
= \sum_{i,m=0}^N &\statevec W^T_{ijk}\,(\qmat_{im} - \qmat_{mi} + \bmat_{im})\left( 
\bigstatevec F^{\ec,\supGLM}(\statevec{U}_{ijk},\statevec{U}_{mjk})\cdot\avg{J\spacevec{a}^{\,1}}_{(i,m)jk}\right)\\
= \sum_{i,m=0}^N &\qmat_{im}\left(\statevec W_{ijk} - \statevec W_{mjk}\right)^T\!
\left(\bigstatevec F^{\ec,\supGLM}(\statevec{U}_{ijk},\statevec{U}_{mjk})\cdot\avg{J\spacevec{a}^{\,1}}_{(i,m)jk}\right)\, \\
&+ \bmat_{im}\statevec W^T_{ijk}\!\left(\bigstatevec F^{\ec,\supGLM}(\statevec{U}_{ijk},\statevec{U}_{mjk})\cdot\avg{J\spacevec{a}^{\,1}}_{(i,m)jk}\right).
\end{split}
\end{equation}

The current proof only contains GLM terms, so we are only concerned with the entropy conservation condition \eqref{eq:GLMPartDiscrete}, which we use to replace the first terms in \eqref{eq:summ1_GLM} with
\begin{equation}\label{eq:moreBs_GLM}
\left(\statevec W_{ijk} - \statevec W_{mjk}\right)^T\statevec F_\ell^{\ec,\supGLM}(\statevec{U}_{ijk},\statevec{U}_{mjk}) = \left(\Psi^{\supGLM}_\ell\right)_{ijk}-\left(\Psi^{\supGLM}_\ell\right)_{mjk} \,,\quad \ell = 1,2,3\,.
\end{equation}
Furthermore, recall that the entries of the boundary matrix $\bmat$ are only non-zero when $i=m=0$ or $i=m=N$, thus
\begin{equation}\label{eq:boundaryBs_GLM}
\bmat_{im}\statevec W^T_{ijk}\statevec F_\ell^{\ec,\supGLM}(\statevec{U}_{ijk},\statevec{U}_{mjk}) = \bmat_{im}\left(\Psi_\ell^{\supGLM}\right)_{ijk}\,,\quad \ell = 1,2,3.
\end{equation}

We substitute \eqref{eq:moreBs_GLM} and \eqref{eq:boundaryBs_GLM} into the final line of \eqref{eq:summ1_GLM} to find
\begin{equation}\label{eq:awfulSum_GLM}
\begin{aligned}
\sum\limits_{i=0}^N&\statevec W^T_{ijk}\sum_{m=0}^N 2\qmat_{im}\!\left(\bigstatevec F^{\ec,\supGLM}(\statevec{U}_{ijk}, \statevec{U}_{mjk})\cdot\avg{J\spacevec{a}^{\,1}}_{(i,m)jk}\right)\\
&=\sum\limits_{i,m=0}^N\qmat_{im}\!\left(\left[\spacevec{\Psi}^{\supGLM}_{ijk}-\spacevec{\Psi}^{\supGLM}_{mjk}\right]\cdot\avg{J\spacevec{a}^{\,1}}_{(i,m)jk}\right)
+ \bmat_{im}\!\left(\spacevec{\Psi}^{\supGLM}_{ijk}\cdot\avg{J\spacevec{a}^{\,1}}_{(i,m)jk}\right).
\end{aligned}
\end{equation}
We individually examine the terms of the sum \eqref{eq:awfulSum_GLM} from left to right. Due to the consistency of the SBP matrix (i.e. the rows of $\qmat$ sum to zero),
\begin{equation}\label{eq:B11_GLM}
\begin{aligned}
\sum\limits_{i,m=0}^N\qmat_{im}\!\left(\spacevec{\Psi}^{\supGLM}_{ijk}\cdot\avg{J\spacevec{a}^{\,1}}_{(i,m)jk}\right) &= \frac{1}{2}\sum\limits_{i=0}^N\left(\spacevec{\Psi}^{\supGLM}_{ijk}\cdot\left(J\spacevec{a}^{\,1}\right)_{ijk}\right)\sum\limits_{m=0}^N\qmat_{im} + \frac{1}{2}\sum\limits_{i,m=0}^N\qmat_{im}\!\left(\spacevec{\Psi}^{\supGLM}_{ijk}\cdot\left(J\spacevec{a}^{\,1}\right)_{mjk}\right)\\[0.05cm]
&=\frac{1}{2}\sum\limits_{i,m=0}^N\qmat_{im}\!\left(\spacevec{\Psi}^{\supGLM}_{ijk}\cdot\left(J\spacevec{a}^{\,1}\right)_{mjk}\right).
\end{aligned}
\end{equation}
On the second term, we apply the summation by parts property \eqref{eq:SBPProperty}, reindex on the $\qmat_{mi}$ term, and utilize the symmetric property of the arithmetic mean to rewrite
\begin{equation}\label{eq:B12_GLM}
\begin{aligned}
-\!\!\!\sum\limits_{i,m=0}^N\qmat_{im}\!&\left(\spacevec{\Psi}^{\supGLM}_{mjk}\cdot\avg{J\spacevec{a}^{\,1}}_{(i,m)jk}\right) = -\!\!\!\sum\limits_{i,m=0}^N(\bmat_{im}-\qmat_{mi})\left(\spacevec{\Psi}^{\supGLM}_{mjk}\cdot\avg{J\spacevec{a}^{\,1}}_{(i,m)jk}\right)\\
&=-\!\!\!\sum\limits_{i,m=0}^N(\bmat_{im}-\qmat_{im})\left(\spacevec{\Psi}^{\supGLM}_{ijk}\cdot\avg{J\spacevec{a}^{\,1}}_{(i,m)jk}\right)\\
&=-\!\!\!\sum\limits_{i,m=0}^N\bmat_{im}\!\left(\spacevec{\Psi}^{\supGLM}_{ijk}\cdot\avg{J\spacevec{a}^{\,1}}_{(i,m)jk}\right) + \sum\limits_{i,m=0}^N\qmat_{im}\!\left(\spacevec{\Psi}^{\supGLM}_{ijk}\cdot\avg{J\spacevec{a}^{\,1}}_{(i,m)jk}\right)\\
&=-\!\!\!\sum\limits_{i,m=0}^N\bmat_{im}\!\left(\spacevec{\Psi}^{\supGLM}_{ijk}\cdot\avg{J\spacevec{a}^{\,1}}_{(i,m)jk}\right) + \frac{1}{2}\sum\limits_{i,m=0}^N\qmat_{im}\!\left(\spacevec{\Psi}^{\supGLM}_{ijk}\cdot\left(J\spacevec{a}^{\,1}\right)_{mjk}\right),
\end{aligned}
\end{equation}
where, again, one term in the second arithmetic mean drops out due to consistency of the matrix $\qmat$. 

Substituting the results \eqref{eq:B11_GLM} and \eqref{eq:B12_GLM} into \eqref{eq:awfulSum_GLM} we find the terms containing the boundary matrix, $\bmat$, cancel and the remaining terms combine to become
\begin{equation}
\sum\limits_{i=0}^N\statevec W^T_{ijk}\sum_{m=0}^N 2\qmat_{im}\!\left(\bigstatevec F^{\ec,\supGLM}(\statevec{U}_{ijk}, \statevec{U}_{mjk})\cdot\avg{J\spacevec{a}^{\,1}}_{(i,m)jk}\right) = \sum\limits_{i,m=0}^N\qmat_{im}\!\left(\spacevec{\Psi}^{\supGLM}_{ijk}\cdot\left(J\spacevec{a}^{\,1}\right)_{mjk}\right)
\end{equation}
Thus, the total contribution in the $\xi-$direction in found to be
\begin{equation}
\begin{aligned}
\sum\limits_{j,k=0}^N\omega_{jk}&\sum\limits_{i=0}^N\statevec W^T_{ijk}\sum_{m=0}^N 2\qmat_{im}\!\left(\bigstatevec F^{\ec,\supGLM}(\statevec{U}_{ijk}, \statevec{U}_{mjk})\cdot\avg{J\spacevec{a}^{\,1}}_{(i,m)jk}\right)\\
=&\sum\limits_{j,k=0}^N\omega_{jk}\sum\limits_{i,m=0}^N\qmat_{im}\!\left(\spacevec{\Psi}^{\supGLM}_{ijk}\cdot\left(J\spacevec{a}^{\,1}\right)_{mjk}\right)
=\sum\limits_{i,j,k=0}^N\omega_{ijk}\sum\limits_{m=0}^N\dmat_{im}\!\left(\spacevec{\Psi}^{\supGLM}_{ijk}\cdot\left(J\spacevec{a}^{\,1}\right)_{mjk}\right)
\end{aligned}
\end{equation}
where we reintroduce the derivative matrix instead of the SBP matrix from the property $\qmat_{im}=\omega_{i}\dmat_{im}$.

Similar results hold for the $\eta-$ and $\zeta-$directions of the GLM volume integral approximations leading to
\begin{equation}
\begin{aligned}
\iprodN{\DD\bigcontravec F^{\ec,\supGLM},\statevec W}& =\sum\limits_{i,j,k=0}^N\omega_{ijk}\sum\limits_{m=0}^N\dmat_{im}\!\left(\spacevec{\Psi}^{\supGLM}_{ijk}\cdot\left(J\spacevec{a}^{\,1}\right)_{mjk}\right)
+ \sum\limits_{i,j,k=0}^N\omega_{ijk}\sum\limits_{m=0}^N\dmat_{jm}\!\left(\spacevec{\Psi}^{\supGLM}_{ijk}\cdot\left(J\spacevec{a}^{\,2}\right)_{imk}\right)\\[0.05cm]
& \qquad+ \sum\limits_{i,j,k=0}^N\omega_{ijk}\sum\limits_{m=0}^N\dmat_{km}\!\left(\spacevec{\Psi}^{\supGLM}_{ijk}\cdot\left(J\spacevec{a}^{\,3}\right)_{ijm}\right).
\end{aligned}
\end{equation}
Regrouping these volume terms we have shown that \change{
\begin{equation}
\begin{aligned}
\iprodN{\DD\bigcontravec F^{\ec,\supGLM},\statevec W}&=
\sum\limits_{i,j,k=0}^N\omega_{ijk}\Psi^{\supGLM}_{1,ijk}\left\{\sum\limits_{m=0}^N\dmat_{im}\left(Ja_1^1\right)_{mjk} + \sum\limits_{m=0}^N\dmat_{jm}\left(Ja_1^2\right)_{imk} + \sum\limits_{m=0}^N\dmat_{km}\left(Ja_1^3\right)_{ijm}\right\}\\[0.05cm]
& + \sum\limits_{i,j,k=0}^N\omega_{ijk}\Psi^{\supGLM}_{2,ijk}\left\{\sum\limits_{m=0}^N\dmat_{im}\left(Ja_2^1\right)_{mjk} + \sum\limits_{m=0}^N\dmat_{jm}\left(Ja_2^2\right)_{imk} + \sum\limits_{m=0}^N\dmat_{km}\left(Ja_2^3\right)_{ijm}\right\}\\[0.05cm]
& + \sum\limits_{i,j,k=0}^N\omega_{ijk}\Psi^{\supGLM}_{3,ijk}\left\{\sum\limits_{m=0}^N\dmat_{im}\left(Ja_3^1\right)_{mjk} + \sum\limits_{m=0}^N\dmat_{jm}\left(Ja_3^2\right)_{imk} + \sum\limits_{m=0}^N\dmat_{km}\left(Ja_3^3\right)_{ijm}\right\},
\end{aligned}
\end{equation}}
which makes this term vanish, provided the metric identities are satisfied discretely, i.e., that
\begin{equation}\label{eq:discMetIds_GLM}
\sum_{m=0}^N \dmat_{im}\,\left(J  a^1_n\right)_{mjk}+\dmat_{jm}\,\left(Ja^2_n\right)_{imk}+\dmat_{km}\,\left(Ja^3_n\right)_{ijm}= 0                                                                                                   
\end{equation}
for $n=1,2,3$ at all LGL nodes $i,j,k=0,\ldots,N$ within an element.

\change{Next, we verify that the non-conservative GLM volume term vanishes. So, we consider the following
\begin{equation}
\begin{aligned}
\iprodN{\GalDT\cdot\DDncgrad\psi,\statevec{W}} = \sum_{i,j,k=0}^N\omega_{ijk}\statevec{W}^T_{ijk} &\left[\;\;\sum_{m=0}^N \dmat_{im}\,\psi_{mjk}\left(\GalD_{ijk} \cdot J\spacevec{a}^{\,1}_{ijk}\right)\right. \\
&\left.+     \sum_{m=0}^N \dmat_{jm}\,\psi_{imk}\left(\GalD_{ijk}\cdot J\spacevec{a}^{\,2}_{ijk}\right)\right. \\
&\left.+     \sum_{m=0}^N \dmat_{km}\,\psi_{ijm}\left(\GalD_{ijk}\cdot J\spacevec{a}^{\,3}_{ijk}\right)\right].
\end{aligned}
\end{equation} 
We factor out the $\GalD$ term because it has no $m$ dependence to find
\begin{equation}
\begin{aligned}
\iprodN{\GalDT\cdot\DDncgrad\psi,\statevec{W}} =\quad& \\
\sum_{i,j,k=0}^N\omega_{ijk}\left(\statevec{W}^T_{ijk}\GalD_{ijk}\right)&\cdot
\left[\sum_{m=0}^N \dmat_{im}\,\psi_{mjk} J\spacevec{a}^{\,1}_{ijk}
.+    \sum_{m=0}^N  \dmat_{jm}\,\psi_{imk} J\spacevec{a}^{\,2}_{ijk}
+   \sum_{m=0}^N \dmat_{km}\,\psi_{ijm} J\spacevec{a}^{\,3}_{ijk}\right].
\end{aligned}
\end{equation} 
As in the continuous case \eqref{eq:contractSourceGLM}, $(\statevec{W}^T_{ijk}\GalD_{ijk}) = 0$ holds point-wise, leading to the desired result
\begin{equation}
\iprodN{\GalDT\cdot \DDncgrad\psi,\statevec{W}} = 0.
\end{equation}
}

\end{document}